\newif\iffull
\ifdefined\abstryes
\fullfalse
\else
\fulltrue
\fi

\date{}
\title{Embeddability in the 3-sphere is decidable
}
\newif\ifcmts
\cmtsfalse

\documentclass[11pt,a4paper]{article}
\iffull
\usepackage{a4wide}
\else
\usepackage{a4}
\usepackage{lineno}  % wanted for SoCG abstract
\fi

\usepackage{latexsym}
\usepackage{amssymb,amsmath,amsthm}
\usepackage{graphicx}
\usepackage{url}
\usepackage{color}
\usepackage{verbatim}
\usepackage{epsfig}
\usepackage{authblk} %A package for multiple affiliations
\usepackage{enumerate}

\author[1,2,a,b]{Ji\v{r}\'{\i} Matou\v{s}ek}
\author[3,a]{Eric Sedgwick}
\author[1,4,5,a]{Martin Tancer}
\author[4,a,c]{Uli Wagner}
\affil[1]{Department
of Applied Mathematics,
Charles University, Malostransk\'{e} n\'{a}m.
25, 118~00~~Praha~1, Czech Republic.}
\affil[2]
{Department of Computer Science, ETH Z\"urich, 8092 Z\"urich,
Switzerland}
\affil[3]{School of Computing, DePaul University, 243 S.~Wabash Ave,
Chicago, IL 60604, USA}
\affil[4]{IST Austria, Am Campus 1, 3400 Klosterneuburg, Austria}
\affil[5]{Part of this work was done when the third author was affiliated with
  Institutionen f\"{o}r matematik, Kungliga Tekniska H\"{o}gskolan,
Linstedsv\"{a}gen 25, 100 44 Stockholm}
\affil[a]{Supported by the ERC Advanced Grant No.~267165.}
\affil[b]{Partially supported by Grant GRADR Eurogiga GIG/11/E023.}
\affil[c]{Supported by the Swiss National Science Foundation (Grants SNSF-PP00P2-138948 and SNSF-200020-138230).}

\newcommand{\A}{\mathcal A}
\newcommand{\C}{\mathcal C}

\newcommand{\T}{\mathcal T}
\newcommand{\E}{\mathcal E}
\newcommand{\F}{\mathcal F}

\newcommand{\BB}{\mathcal B}

\newcommand{\R}{{\mathbb{R}}}

\newcommand{\vv}{\vec{v}}
\newcommand{\hs}{~\widetilde +~}
\DeclareMathOperator{\compl}{cpx}
\DeclareMathOperator{\sd}{sd}
\DeclareMathOperator{\wt}{wt}

\newcommand\bd{\partial}
\newcommand\lmax{{\ell_{\rm max}}}

\theoremstyle{plain}
\newtheorem{theorem}{Theorem}[section]
\newtheorem{proposition}[theorem]{Proposition}
\newtheorem{lemma}[theorem]{Lemma}
\newtheorem{corollary}[theorem]{Corollary}

\newtheorem{definition}[theorem]{Definition}

\newtheorem{assume}[theorem]{Assumption}

\newtheorem{defn}[theorem]{Definition}
\newtheorem{obs}[theorem]{Observation}

\newcommand{\heading}[1]{\vspace{1ex}\par\noindent{\bf #1}}

\long\def\onefigure#1#2{%  #1 picture,  #2  caption
\begin{figure*}[tbp]
\begin{center}
#1
\end{center}
\caption{#2}
\end{figure*}
}

\def\immediateFigure#1{%
\smallskip\begin{center}#1\end{center}\smallskip }

\newcommand{\labfig}[2]  % labeled figure
{\onefigure{\mbox{\includegraphics{#1}}}{\label{f:#1} #2} }

\newcommand{\labfigw}[3]  % labeled figure with prescribed width
{\onefigure{\mbox{\includegraphics[width=#2]{#1}}}{\label{f:#1} #3}}

\newcommand{\immfig}[1]  % immediate figure
{\immediateFigure{\mbox{\includegraphics{#1}}}}

\newcommand{\immfigw}[2] % immediate figure with prescribed width
{\immediateFigure{\mbox{\includegraphics[width=#2]{#1}}}}

\ifcmts
\newcommand{\marrow}{\marginpar{\boldmath$\longleftarrow$}}
\newcommand{\jirka}[1]{\ifhmode\newline\fi\marrow \textsf{\textcolor{blue}{\bf JIRKA:} #1\newline}}
\newcommand{\uli}[1]{\ifhmode\newline\fi\marrow \textsf{\textcolor{magenta}{\bf ULI:}  #1\newline}}
\newcommand{\eric}[1]{\ifhmode\newline\fi\marrow \textsf{\textcolor{red}{\bf ERIC:} #1\newline}}
\newcommand{\martin}[1]{\ifhmode\newline\fi\marrow \textsf{\textcolor{green}{\bf
MARTIN:} #1\newline}}

\else
\newcommand{\marrow}{}
\newcommand{\jirka}[1]{}
\newcommand{\eric}[1]{}
\newcommand{\uli}[1]{}
\newcommand{\martin}[1]

\fi

\newcommand\EMBED[2]{{\rm EMBED$_{#1\to#2}$}}

\newcommand\EMB{{\tt EMB}}
\newcommand{\twist}{\widetilde \times}

\begin{document}
\iffull\else\linenumbers\fi
\maketitle

\begin{abstract}
We show that the following algorithmic problem is decidable:
given a $2$-dimensional simplicial complex, can it be embedded
(topologically, or equivalently, piecewise linearly) in $\R^3$?
By a known reduction, it suffices to decide the embeddability
of a given triangulated 3-manifold $X$ into the 3-sphere $S^3$.
The main step, which allows us to simplify $X$ and recurse, is
in proving that if $X$ can be embedded in $S^3$, then
there is also an embedding in which $X$  has a
\emph{short meridian}, i.e., an essential curve in the boundary of $X$
bounding a disk in $S^3\setminus X$ with length bounded
by a computable function of the number of tetrahedra of~$X$.
\end{abstract}

\section{Introduction}

\heading{The embeddability problem. }
Let \EMBED kd be the following algorithmic
problem: given a finite simplicial complex $K$ of dimension
at most $k$, does there exist a (piecewise linear)
embedding of $K$ into $\R^d$?

A systematic investigation of the computational
complexity of this problem was initiated in
\cite{MatousekTancerWagner:HardnessEmbeddings-2011};
earlier it was known that \EMBED 12\ (graph planarity) is solvable
in linear time, so is \EMBED 22\ \cite{grossRosen},
and   for every $k\ge 3$ fixed,
\EMBED k{2k} can be decided in polynomial time
(this is based on the work of Van Kampen, Wu, and Shapiro;
see \cite{MatousekTancerWagner:HardnessEmbeddings-2011}).

For dimension $d\ge 4$, there is now a reasonably good understanding
of the computational complexity of \EMBED kd: for all $k$ with
$(2d-2)/3\le k\le d$ it is NP-hard (and even undecidable
if $k\ge d-1\ge 4$) \cite{MatousekTancerWagner:HardnessEmbeddings-2011},
while for $k< (2d-2)/3$ it is polynomial-time solvable, assuming $d$ fixed,
as was shown in a series of papers on computational homotopy theory
\cite{CKMSVW11,polypost,pKZ1,aslep}. (However, the cases with
$(2d-2)/3\le k$ known to be NP-hard but not
proved undecidable are still intriguing.)

Thus, the most significant gap up until now has been the
cases $d=3$ and $k=2,3$, and in particular, after graph planarity
(\EMBED 12), the problem \EMBED23 can be regarded as the most
intuitive and probably practically most relevant case.

\heading{Embeddability in $\R^3$. }
Here we close this gap, at least as far as decidability is concerned.

\begin{theorem}\label{t:main} The problem \EMBED 23\ is
algorithmically decidable. That is, there is an algorithm that,
given a $2$-dimensional simplicial complex $K$, decides whether $K$ can be
embedded (piecewise linearly, or equivalently, topologically)
in~$\R^3$.
\end{theorem}

Let us remark that one can naturally consider (at least) three different
kinds of embeddings of a simplicial complex $K$ in $\R^d$,
illustrated in the next picture for a $1$-dimensional complex (graph):
\immfig{k4embs}
%\begin{itemize}
%\item
For \emph{linear embeddings}, also referred to as
 \emph{geometric realizations}, each simplex
of $K$ should be mapped affinely
to a (straight) geometric simplex in $\R^d$. This kind of embeddability
is decidable in PSPACE regardless of the dimensions, and it is
\emph{not} what we consider here.

For \emph{piecewise linear}, or \emph{PL}, embeddings,
one seeks a linear embedding of some (arbitrarily fine) subdivision
of~$K$. Finally, for a
%\item
\emph{topological embedding}, $K$ is
embedded by an arbitrary injective continuous map.
%\end{itemize}

While topological and PL embeddability need not coincide for some
ranges of dimensions, for ambient dimension $d=3$, they do,\footnote{
For complexes of dimension $k=2$, this follows from\cite{Bing-3dmanifoldtriangulated,Papakyriakopoulos:NewProofInvarianceHomologyGroups-1943},
see also \cite{MatousekTancerWagner:HardnessEmbeddings-2011}
for more details and references\iffull; for complexes of dimension $k=3$, this follows from
the reduction in Section~\ref{s:emb33}\fi.}
 and this is the notion of embeddability
considered here.

An algorithm for \EMBED33\ can be obtained from Theorem~\ref{t:main}
by a simple reduction\iffull, given in Section~\ref{s:emb33}\else, which
is omitted in this extended abstract\fi.

\begin{corollary}\label{c:emb33} The problem \EMBED33 is decidable as well.
\end{corollary}

\heading{Thickening to 3-manifolds. } For a $2$-complex $K$,
(PL) embeddability in $\R^3$ is easily seen to be equivalent to
embeddability in $S^3$, and from now on,
we work with $S^3$ as the target.

The first step in our proof
of Theorem~\ref{t:main} is testing whether a given simplicial
$2$-complex $K$ embeds in any 3-dimensional manifold at all.

Let us suppose that there is an embedding $f\:K\to M$ for some
$3$-manifold $M$ (without boundary),
and take a sufficiently small closed neighborhood $X$
of the image $f(K)$ in $M$---the technical term here is a
 \emph{regular neighborhood}. Then $X$ is a $3$-manifold
with boundary, called a \emph{$3$-thickening} of~$K$.

There is an algorithm, due to Neuwirth~\cite{Neuwirth:An-algorithm-for-the-construction-of-3-manifolds-from-2-complexes-1968}
(see also \cite{Skopenkov:Thickening-1995} for an exposition) that,
given $K$, tests whether it has any $3$-thickening, and if yes,
produces a finite list of all possible $3$-thickenings, up to homeomorphism,
as triangulated 3-manifolds with boundary (\emph{without}
the knowledge of~$M$). Then $K$ embeds in $S^3$ iff one of its
3-thickenings does. Hence it suffices to prove the following.

\begin{theorem}\label{t:3m-in-S3}
There is an algorithm that, given
a triangulated $3$-manifold $X$ with boundary, decides whether $X$
can be embedded in~$S^3$.
\end{theorem}

\heading{Concerning the running time.} Our proof does provide an explicit
running time bound for the algorithm, but currently a rather high one,
certainly primitive recursive but even larger than an iterated
exponential tower.
Thus, we prefer to keep the bounds unspecified, in the interest
of simplicity of the presentation.

By refining our techniques, it might be possible to show the problem
to lie in the class NP. Going beyond that may be quite challenging:
indeed, as observed in \cite{MatousekTancerWagner:HardnessEmbeddings-2011},
\EMBED23 is at least as hard
as the problem of recognizing $S^3$  %\eric{I removed 'the' in front of $S^3$}
(that is,
given a simplicial complex, decide whether it
is homeomorphic to $S^3$).
The latter problem is in NP \cite{Ivanov:ComputationalComplexityDecisionProblems3DimensionalTopology-2008,S3inNP}, but nothing more seems to be known
about its computational complexity (e.g., polynomiality or NP-completeness).

\heading{Related work. }
There is a vast amount of literature on computational problems
for $3$-manifolds and knots. Here we give just a sample;
further background and references can be found in
the sources cited below and in~\cite{agol-hass-thurston}.
A classical result is Haken's algorithm deciding
whether a given polygonal knot in $\R^3$ is trivial \cite{Haken:TheorieNormalflaechen-1961}. More recently, this problem was shown to lie in NP
\cite{hass-lagarias-pippenger}, and, assuming the Generalized Riemann
Hypothesis, in coNP as well \cite{kuperberg2011knottedness}.
The knot equivalence problem is also decidable
\cite{Haken:TheorieNormalflaechen-1961,Hemion:ClassificationHomeomorphisms2Manifolds-1979,Matveev:ClassificationSufficientlyLarge3Manifolds}, but
nothing seems to be known about its complexity status.

Closer in spirit to the problem investigated here
are algorithms for
deciding whether a given $3$-manifold is homeomorphic to $S^3$,
already mentioned above
\cite{Rubinstein:AlgorithmRecongnizing3Sphere-1995,Thompson:ThinPositionrecognitionProblemS3-1994,Ivanov:ComputationalComplexityDecisionProblems3DimensionalTopology-2008,S3inNP}.

An important special case of Theorem~\ref{t:3m-in-S3} is testing
embeddability into $S^3$ for an $X$ whose boundary is a single torus;
this amounts to recognizing knot complements and was solved in
\cite{Jaco:Decision-problems-in-the-space-of-Dehn-fillings-2003}.
Some of the ideas in that work are used in our proof, but most
of the argument is fairly different.

In a different direction, Tonkonog \cite{tonkonog}
provided an algorithm for deciding whether a given 3-manifold $X$ with
boundary embeds into \emph{some} homology 3-sphere\footnote{A $3$-manifold
whose homology groups are the same as those of~$S^3$.}
(which may depend on~$X$).
His methods are completely different from ours (except for
using a 3-thickening to pass from 2-dimensional complexes
to 3-manifolds), and it seems to be only loosely
related to the problems investigated here.

\heading{Future directions. } Besides the obvious questions of finding
a more efficient algorithm, say one in NP, and/or proving hardness results,
one may consider embeddability into other 3-manifolds $M$ besides $S^3$.
We believe that this may be within reach of the methods used here,
but definitely a number of issues would have to be settled.

%\eriC{The following generalization is quite possibly within reach:   Let $M$ be a 3-manifold with a recognition algorithm.  There is an algorithm to decide whether $X$ embeds in $M$. }

%\jirka{Add: Masbaum?!?! (wrote to him)}

\heading{The main technical contribution. } Our algorithm relies on a large
body of work in 3-dimensional topology.

When we talk about a \emph{surface} in $X$, unless explicitly stated
otherwise, we always mean a $2$-dimensional manifold
$F$ with boundary \emph{properly
embedded} in $X$, that is, with $\bd F\subset \bd X$.
Similarly, curves are considered properly embedded in a surface,
so a connected curve can be a loop in the interior of the surface or
an arc connecting two points of the boundary.
Two properly embedded surfaces $F$ and $F'$ are
\emph{isotopic} if they are embeddings of the same surface $F_0$
and there is a continuous family of proper embeddings $F_0\to X$
starting with $F$ and ending with~$F'$. An similar definition of isotopy applies to curves embedded in surfaces.

As in almost all algorithms working with 3-manifolds, we use
Haken's method of \emph{normal curves and surfaces}, actually
in a slightly extended form. Here we recall them very briefly;
we refer to
\cite{Hempel:3-manifolds-1976,Jaco:Algorithms-for-the-complete-decomposition-of-a-closed-3-manifold-1995}
for background\iffull, and in Section~\ref{s:normal-surfaces} below we will discuss
a variant\fi.

%\eriC{changed cite to JacoTollefson.   Jaco's book doesn't address normal surface theory}

A \emph{normal curve} in a triangulated 2-dimensional surface $F$
intersects every triangle in finitely many disjoint pieces, which
we can think of as straight segments, as in the left picture:
\immfig{normal}
The main point is that such a curve is described, up to isotopy,
purely combinatorially: namely, for every triangle $T$, there are
just three types of segments of the curve inside, and it is enough
to specify the number of segments for each type, for each $T$.
In the picture, the numbers are $5,2,1$.

Similarly, a
\emph{normal surface} in a triangulated 3-manifold
intersects each tetrahedron in finitely many
of disjoint pieces, each of them a triangle or a quadrilateral,
as in the right picture above.
This time there are seven types of
pieces, four triangular and three quadrilateral, per tetrahedron
(although no two types of quadrilateral pieces may coexist in
a single tetrahedron, since they would have to intersect, which is
not allowed).
So a normal surface $F$ in a 3-manifold with $t$ tetrahedra can be described
by a vector of $7t$ nonnegative integers. This vector is
called the \emph{normal vector} of~$F$.

A \emph{normal isotopy} is an isotopy during which the intermediate
curve or surface stays normal; in particular, it may not cross any
vertex of the triangulation.

%Finally, in an \emph{almost normal} surface is like a normal surface
%except that in at most one tetrahedron we also allow,
%in addition to the patches mentioned above, one of two
%\emph{exceptional pieces}, namely, a \emph{tube} or an \emph{octagon}:

Going back to embeddings, we first simplify the situation
using a result of Fox \cite{Fox:On-the-imbedding-of-polyhedra-in-3-space-1948},
which allows us to assume that the complement of the supposed embedding of $X$
in $S^3$ is a union of handlebodies.\footnote{A \emph{handlebody}
is a ball with (solid) handles, or equivalently, a $3$-thickening
of a finite connected 1-dimensional complex (graph).}
(These handlebodies may be knotted or linked in~$S^3$, though, as in the
picture at the beginning of the next section.)
 This assumption is quite important and
nontrivial; for example, we note that if $X$ is a solid torus, it can
also be embedded in $S^3$ in a knotted way, so that the complement
is not homeomorphic to a solid torus.

Thus, now we ask if there is a way of ``filling'' each component
of $\bd X$ with a handlebody so that the resulting closed manifold
is homeomorphic to $S^3$. Spherical boundary components are easy,
since there is only one way, up to homeomorphism,
of filling a spherical boundary component with a ball.
However, already for a toroidal component there are infinitely
many nonequivalent ways of filling it with a solid torus.
Indeed, the filling can be done in such a way that
a circle $\alpha$ on the toroidal component of $\bd X$, as in the left picture,
\immfig{twisttorus}
is identified with a curve $\beta$ on the boundary of the solid
torus, shown in the right picture, where $\beta$ may wind
around the solid torus as many times as desired.
For boundary components of higher genus, there are also infinitely
many ways of filling, and their description is still more complicated.
For every specific way of filling the boundary components
of $X$ with handlebodies
we could test whether the resulting closed manifold is an $S^3$,
but we cannot test all of the infinitely many possibilities.
This is the main difficulty we have to overcome to get an algorithm.

Next, by more or less standard considerations, we can make sure
that there is no ``way of simplifying $X$ by cutting
along a sphere or disk''---in technical
terms, we may assume that $X$  is \emph{irreducible},
that is, every $2$-sphere embedded in $X$ bounds a ball in $X$,
and that $X$ has an \emph{incompressible boundary},
i.e., any curve in $\partial X$ bounding a disk in $X$
also bounds a disk in~$\partial X$.

For dealing with such an $X$, the following result is the key:

\begin{theorem}\label{t:short-meridian}
Let $X$ be an irreducible $3$-manifold, neither
a ball nor an $S^3$, with incompressible boundary and with a
$0$-efficient triangulation $\T$.  If $X$ embeds in $S^3$, then there
is also an embedding for which $X$ has a
\emph{short meridian} $\gamma$, i.e., an
essential\footnote{Meaning that $\gamma$ does not bound a disk in~$\bd X$.}
 normal curve
$\gamma\subset \bd X$ bounding a disk in $S^3\setminus X$
such that the length of $\gamma$, measured as the number
of intersections of $\gamma$ with the edges of $\T$, is bounded
 by a computable function of the number of tetrahedra in~$\T$.
\end{theorem}

In this theorem, \emph{$0$-efficient triangulation}
is a technical term introduced in \cite{Jaco:0-efficient-triangulations-of-3-manifolds-2003}\iffull, whose definition will be recalled later
in Section~\ref{s:def-0-eff}\else;
we omit the definition in this extended abstract\fi. We are using
$0$-efficient triangulations
in order to exclude non-trivial normal disks and 2-spheres
in~$X$.

We should also mention that the triangulations
commonly used in 3-dimensional topology, and also here,
are not simplicial complexes in the usual
sense---%
%they are \emph{$\Delta$-complexes} in the sense of \cite[Sec.~2.1]{hatcher}:
they are still made by gluing (finitely many)
tetrahedra by their faces, but any set of gluings that produces a manifold is allowed, even those that identify
faces of the same tetrahedron.  As a result, a particular tetrahedron may not have four distinct faces, six distinct edges and four distinct vertices.   In particular,  0-efficient triangulations of  the manifolds we consider have a single vertex in each boundary component and none in the interior, all edges in the boundary form loops.  This is the necessary result of modifying a triangulation by collapsing simplices, a triangular
face to an edge or to a vertex, etc.;
see \cite[Sec.~2.1]{Jaco:0-efficient-triangulations-of-3-manifolds-2003} for a thorough discussion.   There is even a mind-boggling
one-tetrahedron one-vertex triangulation of the solid torus obtained by gluing a pair of faces of a single tetrahedron,
see~\cite{Jaco:Decision-problems-in-the-space-of-Dehn-fillings-2003}.
%(at the time
%of writing, a visualization was available at
%{\tt http://\allowbreak sketchesoftopology.wordpress.com/\allowbreak2008/05/15/\allowbreak the-\allowbreak one\allowbreak -tetrahedron-\allowbreak one\allowbreak -vertex\allowbreak -triangulation\allowbreak -of\allowbreak -the\allowbreak -solid\allowbreak -torus/}).

%\eriC{Unfortunately, the one-tetrahedron visualization isn't right.  I know Ken and will let him know.}  \eriC{again something Jaco and Rubinstein came up with, think it may have first appeared in Jaco and Sedgwick}

Let us remark that  $X$ as in the theorem need not have a short meridian
for every possible embedding, even
if we assume that the complement consists of handlebodies.
For example, if $X$ is a thickened torus (a torus times an interval),
we can embed it so that the curves bounding disks in $S^3\setminus X$
are arbitrarily long w.r.t.\ a given triangulation of~$X$. We must
sometimes change
the embedding to get short meridians.

It is also worth mentioning that this problem does not occur if
$\bd X$ is a single torus, i.e., the knot complement case.
Here a celebrated theorem of Gordon and Luecke
\cite{gordon-luecke-knotcompl}
makes sure that there is only
one embedding, up to a self-homeomorphism of $S^3$,
and the meridian is unique up to isotopy.
%\jirka{Right?} \eriC{yes, right.  From the embedding perspective, there are two embeddings, mirrors of each other (meridian same in each).   }
This is why the single-torus boundary case solved in
\cite{Jaco:Decision-problems-in-the-space-of-Dehn-fillings-2003}
is significantly easier than the general case.

%\eriC{ I think we should say for the knot case, there is a unique meridian (the one homeomorphism just reverses its orientation)}

\section{An outline  of the arguments }\label{s:outline}

Our algorithm for Theorem~\ref{t:3m-in-S3}, deciding the
embeddability of a given $3$-manifold $X$ in $S^3$, for the case of $X$
irreducible and with incompressible boundary, consists in testing
every possible normal curve $\gamma\subset\bd X$ of length
bounded as in Theorem~\ref{t:short-meridian}. For each such candidate
$\gamma$, we construct a new manifold $X'=X'(\gamma)$ by
\emph{adding a $2$-handle} to $X$ along $\gamma$, which means that
we glue a disk bounded by $\gamma$ to the outside of $X$
and thicken it slightly, as illustrated in the following picture:
\immfig{meridi}
%\immfigw{foxEmbed-eps-converted-to}{3in}
Here $X$ is the complement of the union of two (linked)
handlebodies, a knotted solid 3-torus and a solid torus,
and for $X'$, the solid 3-torus in the complement
has been changed to a solid double torus.

Then we test the embeddability
of each $X'(\gamma)$ recursively, and $X$ is embeddable
iff at least one of the $X'(\gamma)$ is. It is not hard to show
that the algorithm terminates, using the vector of genera
of the boundary components of $X$; see Section~\ref{s:algo}.

The proof of Theorem~\ref{t:short-meridian} \iffull
occupies most of the paper and \fi
has many technical steps. In this \iffull section \else
extended abstract\fi we give an outline.

 We assume $X$ to be embedded in $S^3$, the complement being
a union of handlebodies,
and we apply a result of  Li \cite{Li:Thin-position-and-planar-surfaces-for-graphs-in-the-3-sphere-2010}
stating that there is a planar surface (i.e., a disk with holes) $P\subset X$
that is ``stuck'' in its position in a suitable sense
(namely, $P$ is either essential,\footnote{The precise definitions
of essential, strongly irreducible, and boundary strongly irreducible are
somewhat complicated and \iffull we postpone them to Section~\ref{s:irredu}\else we refer to the full version, or e.g., to \cite{Li:Thin-position-and-planar-surfaces-for-graphs-in-the-3-sphere-2010}\fi.
}
 or strongly irreducible and boundary
strongly irreducible) and
is meridional or almost meridional.

Here an essential curve $\gamma\subset \bd X$
is a \emph{meridian} in a given embedding of $X$ in $S^3$ if it
bounds a disk in $S^3\setminus X$. The surface $P$ is
\emph{meridional} if each component of $\bd P$ is a meridian,
and it is \emph{almost meridional} if all components of $\bd P$
but one are meridians. (Actually, Li has yet another case
in his statement, but as we will check, that case can be reduced
to the ones given above\iffull; see
Lemma~\ref{lemCompressAlmostMeridional}\fi.)
The next picture illustrates a meridional $P$
in the case where $X$ is embedded in $S^3$ as the complement of a solid torus neighborhood of the figure `8' knot:
\immfig{fig8Thick}

%\eric{changed this to mention fig 8 knot, hope that is ok}

Next, by choosing $P$ as above with suitable minimality properties,
one can make sure that $P$ is normal or almost normal\footnote{An
\emph{almost normal} surface is like a normal surface
except that in at most one tetrahedron we also allow,
in addition to the triangular and quadrangular
pieces, one of two types of
\emph{exceptional pieces}, namely, a \emph{tube} or an
\emph{octagon}\iffull; see Section~\ref{s:almostnormal-expl}.
\else; see the full version or, e.g.,
\cite{Jaco:Decision-problems-in-the-space-of-Dehn-fillings-2003}.\fi
} for the given
triangulation. For the case of $P$ essential, this is an old
result going back to Haken and Schubert (and for our notion of
complexity of $P$, a proof is given in \iffull Section~\ref{sec:summands_of_annuli})\else in the full version\fi,
while for $P$ strongly irreducible and boundary strongly irreducible
this follows from \cite{bachmanDerbyTalbotSedgwick};
also see \cite{stocking-trams} for the case of a
strongly irreducible surface in a closed manifold.
It remains to show that, in this setting, at least one of the meridians in
$\partial P$ must be short.

%\eriC{essential is normal is old result, probably goes back to Haken or Schubert.  we have a proof for our complexity in Sec 5.  SI/BSI is relatively new due to BDTS12.   (should also state, for closed surfaces SI is due to Stocking  M. Stocking. Almost normal surfaces in 3-manifolds. Trans. Amer. Math. Soc., 352(1):171– 207, 2000.)}

Here we apply an \emph{average
length estimate}, which is an idea of Jaco and Rubinstein
appearing in
\cite{Jaco:Decision-problems-in-the-space-of-Dehn-fillings-2003,Jaco:Finding-planar-surfaces-in-knot--and-link-manifolds-2009}.

Let $\gamma_1,\ldots,\gamma_b$ be the components of $\bd P$, and let
$\ell(\bd P)=\sum_{i=1}^b \ell(\gamma_i)$ be the boundary length of $P$.
We know that all the $\gamma_i$ but at most one are meridians.
The length of the shortest meridian is bounded by
the average $\ell(P)/(b-1)$, and we want to bound this average
by a (computable) function of $t$, the number of tetrahedra in
the triangulation $\T$ of~$X$.

Now by the theory of normal surfaces, the (almost) normal surface $P$
can be written as a normal sum\footnote{For normal surfaces
$F,F_1,F_2$ in a triangulation $\T$ of $X$, $F$ is called
the \emph{normal sum} of $F_1$ and $F_2$ if $\vv(F)=\vv(F_1)+\vv(F_2)$,
where $\vv(F)$ denotes the normal vector of~$F$.
Similarly for almost normal surfaces, where we have extra coordinates
in $\vv(F)$ for the exceptional types of pieces; in this case,
at least one of $F_1,F_2$ has to be normal. \iffull Also see
Section~\ref{s:normal-surfaces}.\fi}
 of fundamental surfaces in $X$,
\begin{equation}\label{e:PsumF}
P=\sum_{i} k_i F_i,
\end{equation}
where the $k_i$ are positive integers
and the $F_i$ are surfaces from a finite collection;
their number, as well as $\lmax:=\max_i\ell(\bd F_i)$
can be bounded by a (computable) function of $t$ alone, and does not
depend on~$P$.

Since the boundary length is additive w.r.t.\ normal sum, we have
$\ell(\bd P)=\sum_i k_i\ell(\bd F_i)\le \lmax K$,
where $K:=\sum_i k_i$ is the number of fundamental summands
in the expression for $P$, and
so it suffices to show that $K\le Cb$, with some computable function $C=C(t)$.

The basic version of the average-length estimate uses the
Euler characteristic %\footnote{
%an orientable $2$-manifold of genus $g$ with $h$ holes has
%Euler characteristic $2-2g-h$}
 $\chi$ as an accounting device. Since
$\chi$
is additive as well,  $\chi(P)=\sum_i k_i\chi(F_i)$.
Since $P$ is a planar surface with $b$ boundary components,
we have $\chi(P)=2-b$.

Now an ideal situation for the average-length estimate (which
we cannot guarantee in our setting) is when $\chi(F_i)\le -1$
for every $i$; in other words, none of the summands is a disk,
2-sphere, annulus, M\"obius band, or torus (or projective plane
or Klein bottle, but these cannot occur in $X$ embedded in $S^3$).
Then we get
$b-2=-\chi(P)=\sum_i k_i(-\chi(F_i))\ge K$, and we are done
(even with $C=1$).

In our actual setting, the summands with $\chi>0$,
i.e., spheres and disks, are excluded by the $0$-efficient
triangulation of $X$.
% \eriC{0 efficiency itself is sufficient, don't think we need the minimality.   in the short meridian proof we do rule out by other means, but maybe that should be changed to 0 efficient to simplify the presentation}
We also need not worry about torus summands, since they have
empty boundary and thus do not contribute to $\ell(\bd P)$.
The real problem are annuli (and M\"obius bands, but since twice
a M\"obius band, in the sense of normal sum, is an annulus,
M\"obius bands can be handled easily once we deal with annuli).

There are two kinds of annuli, which need very different treatment:
the essential ones, and the \emph{boundary parallel} ones.
Here an annulus $A\subset X$ is \emph{boundary parallel}
if it can be isotoped to an annulus $A'\subset\bd X$
with $\bd A'=\bd A$ while keeping the annulus boundary fixed.
Boundary parallel
annuli do not occur for $P$ essential, but they might occur
for the case of $P$ strongly irreducible and boundary
strongly irreducible.

To deal with the annulus summands, we first construct what
we call an \emph{annulus curve}
$\alpha\subset\bd X$. This is the boundary of a maximal collection $\A$
of essential annuli, maximal in the sense that each of the two
boundary curves of every other essential
annulus, after a suitable normalization, either intersects $\alpha$
or is normally isotopic to a component of $\alpha$.
We bound the length of $\alpha$ by a computable
function of $t$, and $|\alpha\cap P|$,
the number intersections of $\alpha$ with $P$,
by $C'b$, for some computable $C'=C'(t)$, again assuming $P$ minimal in a suitable sense.
%\uli{Again, maybe stress that $C'(t)$ is computable?}
%\martin{Computable stressed. Actually, I was slightly confused what do we want
%to say with the initial part of sentence about bounding the length of $\alpha$.
%We do not use it later on in the intro (if I am not mistaken). Thus, I would
%either remove it or further explain in the intro (at the moment I am more close
%to the first option). If kept, then computable should be stressed also in this
%case.}
For obtaining this bound we may need to change the embedding
of $X$, and we also
 use results about ``untangling'' a system of curves on a surface
by a boundary-fixing self-homeomorphism
from~\cite{matousekSedgwickTancerWagner}.

%\eriC{I think I have a bad ordering in the proof of the main result and it has propagated here as well.  $\B$ applies only to boundary parallel annuli that miss $\alpha$.       So first, anything that meets $\alpha$ is bounded, this include annuli and M-bands that are both essential and boundary parallel.  (This does use minimality of $P$ and may require re-embedding because it uses the curve argument).   Then we rule out boundary parallel that miss $\alpha$ altogether as they are incompatible.  Then we rule out essential annuli that are disjoint from $\alpha$ hence their boundary is parallel into  $\alpha$.  }

%\eric{we also didn't say anything about the interval bundle, not sure if we should }

Similarly, we construct a collection $\Gamma$ of curves that helps to deal
with boundary parallel annuli: those that have minimal boundary
in a suitable sense either
intersect $\alpha$, or their boundaries are normally isotopic to components of
$\alpha$ or curves from~$\Gamma$.

Having constructed such an $\alpha$ and $\Gamma$, we work with normal curves and
surfaces in a  ``marked'' sense, which also takes into account the
position of the curves and surfaces w.r.t.~$\alpha$ and $\Gamma$.
This, in particular, makes the number of intersections with $\alpha$
additive w.r.t.\ the marked normal sum, which in turn
allows us to bound the number of annulus summands in (\ref{e:PsumF}), both
boundary parallel and essential, that intersect $\alpha$ by $C'b$.

Then we might have  boundary-parallel annulus summands
that avoid $\alpha$, but we show that
those do not occur at all, since they would contradict the minimality of~$P$.

Finally, there remain essential annuli that
have a boundary component parallel to a component of $\alpha$.
Here we show that if such an annulus had the coefficient
$k_i$ in (\ref{e:PsumF}) at least $|\alpha\cap \partial P|\le C'b$, then there
is a self-homeomorphism of $X$, namely,
a \emph{Dehn twist} in the annulus,
that makes $P$ simpler,
contradicting its supposed minimality. (Here we may again modify the
assumed embedding of $X$ in $S^3$ in order to get a short
meridian---and, as we have remarked, some such modification is necessary
in the proof, since some embeddings may not have short meridians.)
Hence for these essential annuli, too, the coefficients are bounded
by a linear function of~$b$. This concludes the proof.

\section{The algorithm}\label{s:algo}

If $X$ embeds in $S^3$, then it is orientable,
and orientability can easily be tested algorithmically
(e.g., by a search in the dual graph of the triangulation,
or by computing the relative homology group $H_3(X,\bd X)$).
So from now on, we assume $X$ orientable. In this situation,
the boundary of $X$ is a compact orientable
2-manifold, and thus each component is a 2-sphere with handles.

We describe a recursive procedure \EMB(X) that accepts a triangulated
orientable 3-manifold with boundary and returns TRUE or FALSE depending
on the embeddability of $X$ in $S^3$. (With some more effort, for the TRUE case,
we could also recover a particular embedding, but we prefer simplicity of
presentation.) The procedure works as follows.

\begin{enumerate}
\item\label{step:cpts} (Each component separately)
Let $X_1,\ldots,X_k$ be the connected components of $X$.
If $k>1$, test if
there  is an $S^3$ among the $X_i$ (several algorithms are available
for that \cite{Rubinstein:AlgorithmRecongnizing3Sphere-1995,Thompson:ThinPositionrecognitionProblemS3-1994,Ivanov:ComputationalComplexityDecisionProblems3DimensionalTopology-2008,S3inNP}), and if yes, return FALSE. Otherwise,
still for $k>1$, return the conjunction
$\EMB(X_1)\wedge\cdots\wedge\EMB(X_k)$.
\item\label{step:holes} (Fill spherical holes)
Now we have $X$ connected. If it is an $S^3$, return TRUE.
If there are components of $\bd X$ that are $S^2$'s, form $X'$
by attaching a 3-ball to each spherical component of $\bd X$, and
return $\EMB(X')$.
\item \label{step:redu}
(Connected sum) Form a decomposition $X=X_1\#\cdots\# X_k$ of $X$ into a
connected sum\footnote{For two $3$-manifolds
$X$ and $Y$, the \emph{connected sum} $X\# Y$ is obtained by
removing a small ball from the interior of $X$, another small ball
from the interior of $Y$, and identifying the boundaries of these
two balls.}
 of prime manifolds\footnote{A \emph{prime} 3-manifold is one
that has no decomposition as a connected sum $X\#Y$ with neither
$X$ nor $Y$ an $S^3$.} that are not $3$-spheres.\footnote{
The algorithm for prime decomposition goes back to
Schubert \cite{Schubert-primedecomposition},
for closed manifolds it is presented in detail in
\cite{Jaco:Algorithms-for-the-complete-decomposition-of-a-closed-3-manifold-1995}, and a version for manifolds with boundary is implicit in
\cite{Jaco:0-efficient-triangulations-of-3-manifolds-2003}.}
If $k>1$, i.e., $X$ is not prime, return
$\EMB(X_1)\wedge\cdots\wedge\EMB(X_k)$.
If $X$ is prime but not irreducible, i.e., contains an $S^2$ that
does not bound a ball, then return FALSE.
\item \label{step:compress}
(Boundary compression)
Test if there is a compressing disc $D$ for $\bd X$ (i.e., $\partial D\subset
\bd X$ does not bound a disk in $\partial X$).\footnote{The
idea of an algorithm is due to Haken, %\cite{Haken-compressingdisk???},
and the algorithm is implicit
in~\cite{Jaco:0-efficient-triangulations-of-3-manifolds-2003}. }
%\eriC{footnote 11:  I would just say due to Haken.   His algorithm is right, but just has the advanced knowledge that the manifold embeds in S3}
If yes, cut $X$ along~$D$, obtaining a new manifold~$X'$. Three
cases may occur:
\begin{enumerate}
\item\label{step:a} If $X'$ has two components, $X'_1$ and $X'_2$, return
$\EMB(X'_1)\wedge \EMB(X'_2)$. This case may occur, for example, for $X$
a  handlebody with two handles (a ``thickened 8'')
when $D$ separates the two handles.
\item\label{step:b}
 If $X'$ is connected and the two ``scars'' after cutting along $D$
lie in the same component of $\bd X'$, return \EMB$(X')$.
This case may occur, e.g., for $X$ a solid torus.
\item\label{step:c}
 If neither of the previous two cases occur, then $X'$ is connected
but the scars lie in different components of $\partial X'$.
Return FALSE.
To get an example of $X$ fitting this case, we can start
with a thickened torus (i.e., torus times $[-1,1]$) and connect
the two boundary components by a $1$-handle---which cannot be done
in $\R^3$, but it does give a 3-manifold (with double torus boundary).
\end{enumerate}
\item (Short meridian)\label{step:meridi}
Now $X$ is irreducible and with incompressible boundary.
Using \cite[Thm.~5.20]{Jaco:0-efficient-triangulations-of-3-manifolds-2003}, retriangulate
$X$ with a $0$-efficient triangulation.
Then proceed as described at the beginning of Section~\ref{s:outline}:
let $\gamma_1,\ldots,\gamma_n$ be a list of all closed essential normal
curves in $\bd X$ up to the length bound as in Theorem~\ref{t:short-meridian},
for each $i$ form $X'(\gamma_i)$ by attaching a $2$-handle along $\gamma_i$,
and return the disjunction
$\EMB(X'(\gamma_1))\vee\cdots\vee \EMB(X'(\gamma_n))$.
\end{enumerate}

\begin{lemma} The above procedure always terminates
and returns a correct answer, assuming the validity
of Theorem~\ref{t:short-meridian}.
\end{lemma}

\begin{proof} First we show that the algorithm always terminates.
Let $C_1,\ldots,C_k$ be the components of $\bd X$
numbered so that $g(C_1)\ge \cdots\ge g(C_k)$, where $g(.)$ stands
for the genus, and let $\vec g_{\ge}(X)$ be the vector
$(g(C_1),\ldots,g(C_k))$. We consider these vectors ordered
lexicographically (if two vectors have a different length,
we pad the shorter one with zeros on the right).

Let us think of the computation of the algorithm as a tree, with
nodes corresponding to recursive calls. The branching degree
is finite, so it suffices to check that every branch is finite.

It is easy to see that $\vec g_{\ge}(X)$ cannot increase
by passing to a connected component or to a prime summand,
and that it decreases strictly by a boundary compression
and also by the short meridian step. Indeed, we observe that in the boundary
compression step or the short meridian step,
exactly one of the boundary components $C_i$ is affected, and it
is either split into two components
$C'$ and $C''$ of nonzero genus and with $g(C_i)=g(C')+g(C'')$,
or it remains in one piece but the genus decreases by one.
Since after steps~\ref{step:cpts}--\ref{step:redu} we have
a connected irreducible manifold without spherical boundary components,
for which the next step either finishes the computation
or reduces $\vec g_{\ge}(X)$ strictly,  every branch is finite as needed.
%\eriC{I may not be reading correctly.  But the only step that doesn't reduce complexity is (3) Connected sum.  Is that right?   Don't we need to say that no branch does that twice in a row?  So complexity goes down at least every other step.  }
%\jirka{Strictly speaking, as written, the connected components
%step also might not reduce strictly, if the other cpt is a ball.
%Added explicitly that these two steps can't be repeated.}

%\eric{BTW, I think Jaco would argue with us that the 0-efficiency paper does several of these things for us automatically.  Of course, their results aren't stated that way ....}

It remains to show that the returned answer is correct.
For Step~\ref{step:holes}, we need that there is a unique way of filling
a spherical hole; this is well known and
can be inferred, for example, from the fact
that there is only one orientation-preserving self-homeomorphism of $S^2$
up to isotopy \cite[Sec.~2.2]{farbMargalit}.
%(see Section~\ref{s:emb33}), every PL embedding of $S^2$ into $S^3$
%extends to a PL embedding of the ball bounded by the $S^2$.

%\eriC{do we need PL Schoenflies?   don't we just need that the homeomorphism type of $X \cup ball$ isn't dependent on gluing?    The mapping class group of $S^2$, i.e., orientation preserving homeomorphism modulo isotopies, is {1} }

%\eriC{we can cite farb and margalit for the fact that there is only one homeomorphism (orient preserving) of the 2 sphere up to isotopy (sec 2.2)}

For Step~\ref{step:redu}, it is easily checked
that a connected sum embeds iff the summands do.
Moreover, every $S^2$ embedded in $S^3$ separates it,
and hence if $X$ contains a non-separating $S^2$, then it is not embeddable.

For Step~\ref{step:compress}, it is clear that if $X$ is embeddable,
then so is $X'$.

If, in case~(\ref{step:a}), $X'_1$ and $X'_2$ are both embedded,
then it is easy to construct an embedding of $X$:
Denote $D$'s scars by $D_1$ and $D_2$.
Then a regular neighborhood of $D_i$ is a ball $B_i$ with boundary
$S_i = \partial B_i$, and that meets both $X'_i$ and $S^3\setminus X'_i$
in balls.  Think of each $X'_i$ as embedded in its own copy of $S^3$,
and take a connected sum of these two $S^3$'s so that
$S^3 = S^3 \# S^3 \supset  X'_1 \#_{D_1 = D_2} X'_2=X$.
Similarly, if $X'$ is embedded in case~(\ref{step:b}), then we can connect the
scars by a thin handle in $S^3\setminus X'$ and obtain an embedding
of~$X$.
%\eriC{i changed 'thin fiber' to 'thin handle' as I would understand a fiber to be a fiber in a fibration. ok?}

%\eriC{this is pretty much the same as the previous case but one of the manifolds is a solid torus.   thought I would remark, but no need to change}

In case~(\ref{step:c}), let $C_1\ne C_2$ be the components
of $\bd X'$ containing the scars. Since the disk $D$ does not
separate $X$, we can choose a loop $\delta\subset $X
meeting $D$ in a single point, and such that $\delta$
also meets $C_1$ in a single point. But then, if $X$
were embedded in $S^3$, $C_1$ would yield a nonseparating
surface in $S^3$---a contradiction.

Finally, if one of the $X'(\gamma_i)$ is embeddable in Step~\ref{step:meridi},
then so is $X$ (since in $X'$ we have $2$-handle that was added to $X$, and we
can just assign it to the complement of $X$),
and if $X$ is embeddable, then at least one of the $X'(\gamma_i)$ is by
Theorem~\ref{t:short-meridian}.
\end{proof}

\iffull
\section{Intersections of curves and surfaces}\label{s:irredu}

In this section we collect terminology,  definitions and basic results concerning properly embedded curves in surfaces and properly embedded surfaces in 3-manifolds.    In particular, for latter sections we need that any pair of properly embedded surfaces, each either essential, or, strongly irreducible and boundary strongly irreducible, can be isotoped to intersect essentially.  There are few new results in this section.   The reader is referred to Hempel
\cite{Hempel:3-manifolds-1976} and Jaco \cite{Jaco:Lectures-on-three-manifold-topology-1980} for more background.

We assume throughout that all curves and surfaces have been isotoped
to have transverse intersection.

\subsection{Curves}

A \emph{curve} is a properly embedded $1$-dimensional
 manifold in a surface $F$, each component either a \emph{loop}, which is closed,  or an \emph{arc}, which has two endpoints in $\bd F$.

A loop is \emph{trivial} if it bounds a disk in $F$ and an arc is \emph{trivial} if it co-bounds a disk in $F$ with some arc in $\bd F$.
A curve is \emph{essential} if no component is trivial.

Pairs of curves are assumed to intersect transversally. If $\alpha$ and $\beta$ are a pair of curves, then their  \emph{geometric intersection number} $i(\alpha,\beta) = \min( |\alpha' \cap \beta'|)$  taken over all pairs of curves $(\alpha',\beta')$ where $\alpha'$ and $\beta'$ are isotopic to $\alpha$
and $\beta$ within $F$, respectively. (The isotopies are also
allowed to move endpoints of arcs within the boundary.)

We say that $\alpha$ and $\beta$ bound a \emph{bigon} if there is a disk bounded by a pair of sub-arcs, one from each curve;  see Figure~\ref{figBigons}
in Section~\ref{s:tightsnug} below.    We say that they bound a
\emph{half-bigon} if there is a disk bounded by a pair of sub-arcs, one from each curve, along with an arc in $\partial F$.   If $\alpha$ and $\beta$ bound a bigon or half-bigon, then they can be isotoped to reduce their intersection.

We need this converse, a mild generalization of Farb and Margalit's \emph{bigon criterion}:

\begin{lemma}[Bigon criterion \cite{farbMargalit}]
A pair of curves $\alpha$ and $\beta$ realize their geometric intersection number if and only if they do not bound a bigon or half-bigon. \end{lemma}

\begin{proof} Farb and Margalit show that any pair of connected loops that intersect non-minimally form a bigon.  They also note that this extends to disconnected curves consisting of loops.

If either curve has an arc component, then the doubled curves are properly embedded closed curves in the double\footnote{Meaning that we glue two copies
of $F$ by identifying their boundaries.} of the surface.    If they intersect non-minimally in the original, they intersect non-minimally in the double and hence bound a bigon there.   Thus, they bound a half-bigon in the original.
\end{proof}

\subsection{Essential surfaces}
We will assume that our surfaces are properly embedded in a 3-manifold $X$ that is \emph{irreducible}, i.e., every sphere embedded in $X$ bounds a ball in $X$, and \emph{boundary incompressible}, i.e., any curve in $\partial X$ bounding a disk in $X$ also bounds a disk in $\partial X$ (is trivial).

Let $F$ be a surface properly embedded in $X$.
A \emph{compressing disk} for a $F$ is an embedded disk $D \subset X$  whose interior is disjoint from $F$ and whose boundary is an essential loop in $F$.
 A \emph{boundary compressing disk} is an embedded disk $D \subset X$
whose boundary, $\partial D = f \cup x$, is the union of $f = \partial D \cap F = D \cap F$, an essential arc properly embedded in $F$, and $x = \partial D \cap \partial X  = D \cap F$, an arc properly embedded  in $\partial X$.
Here is an illustration:
\immfig{compressingDiscs}%{3in}

A surface $F$ is \emph{compressible} if it has a compressing disk, \emph{boundary compressible} if it has a  boundary compressing disk, and \emph{incompressible} and \emph{boundary incompressible} if not, respectively.   A surface is \emph{essential} if it is incompressible, boundary incompressible, and not a sphere bounding a ball, or a disk co-bounding a ball with a disk in $\partial X$.   %In an irreducible manifold with incompressible boundary, the boundary of each essential surface with boundary is a curve that is essential in the boundary of the manifold.
%\eriC{removed sentence about essential surface having essential boundary as that is covered in the next prop}

We establish some basic facts about surfaces in $X$.

\begin{proposition}
\label{propSurfaceFacts}
The following statements hold for properly embedded surfaces in $X$, an irreducible, orientable $3$-manifold with non-empty incompressible boundary:
%\eric{added orientable, omitted by mistake.  Definitely need it, at least for (v) }
\begin{enumerate}
\item[\rm(i)] Every disk co-bounds a ball with a disk in~$\partial X$.
\item[\rm(ii)] Every connected surface with an inessential boundary curve is either compressible or a disk.
\item[\rm(iii)] The boundary curve of every compressible annulus is trivial in $\partial X$boundary.
\item[\rm(iv)] Every boundary compressible annulus is boundary parallel (parallel to an annulus in $\partial X$).
\item[\rm(v)] No surface is a projective plane.
\item[\rm(vi)] Every M\"obius band is essential.
\end{enumerate}
\end{proposition}

\begin{proof}
Because $X$ has incompressible boundary, the boundary of a properly embedded disk bounds a disk in $\partial X$.   The union of these disks is a sphere that, because $X$ is irreducible, bounds a ball, yielding~(i).

For (ii), suppose that some boundary curve of connected $F$ bounds a disk in $\partial X$. Then, among disks in $\partial X$ bounded by boundary curves of $F$, an innermost such disk can be pushed slightly into the interior of $X$ while keeping its boundary in $F$.  The boundary curve is either trivial in $F$, in which case $F$ is a disk, or essential in $F$, in which case $F$ is compressible.
%\uli{Should we add a brief explanation of ``innermost disk'', e.g. in a footnote?}
%\eric{edited slightly, is it clearer now?}

Concerning (iii), let $D$ be a compressing disk for an annulus $A$.  Then $\partial D$ separates $A$ into two annuli $A'$ and $A''$.   So $D \cup A'$ and $D \cup A''$ are properly embedded disks, each with one boundary curve of $A$.  Because $\partial X$ is incompressible,  the boundary curves are both
 trivial in~$\partial X$.

As for (iv),
let $B$ be a boundary compressing disk for an annulus $A$.  Then $\partial N(A \cup B)$, the boundary of a regular neighborhood of their union,  has two components, an annulus isotopic to $A$ and a disk.  By (i), the disk co-bounds a ball with a disk in $\bd X$.  But then the union of $N(A \cup B)$ with the ball is a solid torus, across which $A$ is parallel to an annulus in~$\partial X$.

Concerning (v),
if $P$ is a projective plane, then $\partial N(P)$, the boundary of its regular neighborhood, is a sphere which separates $P$ from $\partial X$.   Then $X$ is reducible, for the sphere cannot bound a ball---no ball has interior boundary or contains an embedded projective plane.

Finally, in (vi), let $M$ be a a M\"obius band. Suppose first that $M$ is
compressible and let $D$ be a compressing disk for  $M$.  Then $\partial D$ cannot meet $M$ in a core curve of $M$ for this would imply that the core curve is orientation reversing in $X$.   So $\partial D$ is a 2-sided curve in $M$
%an orientation-preserving loop in $M$,
%\uli{Should we add a brief explanation
%concerning ``orientation-preserving loop''? In particular, this seems to use
%the assumption that $X$ is orientable, so maybe we should repeat that
%assumption in the statement of the proposition?}
%\martin{If I understand well, the orientation preserving loop is what we later
%call two-sided. More importantly, is it possible, Eric, that there is a missing
%case? What if $\partial D$ is the core curve of the M. band? (I assume that it
%can be easily excluded)}
%\uli{I think the core curve can be easily excluded since a M\"obius band with a
%disk attached to the core curve does not embed into any orientable manifold.
%I guess a short argument would be that the orientation of the disk together with
%the global oriantation of the manifold introduces an orientation of a neighborhood of
%the curve in $M$ ($2$-sidedness); in particular, $\partial D$ cannot be the core curve
%(which is not $2$-sided).}
%\eric{I forgot to add $X$ orientable and have added it now.  Then, as Uli says,  D can't meet M in a core curve, because $M \cup D$  won't embed in an orientable $X$.  I have edited slightly.  Please edit further if you think more needs to be said. }
and separates it into an
annulus and a narrower M\"obius band $M'$.  Then the union $M' \cup D$ is an
embedded  projective plane contradicting~(v).

Suppose a M\"obius band $M$ is boundary compressible.  Then $\partial N(M)$ is a boundary compressible annulus $A$. By (iv), $A$ is boundary parallel,
 and co-bounds an solid torus with an annulus in the boundary.  But the parallel region cannot contain the M\"obius band $M$, and hence $X$ is the union of two solid tori, $N(M)$ and the solid torus parallel region.
\end{proof}

We say that a a pair of surfaces,  $F$ and $G$, \emph{intersect essentially} if each component of the curve $F \cap G$ is essential in both $F$ and $G$ (they are allowed to be disjoint).  It is well known that essential surfaces can be arranged to intersect essentially:

\begin{lemma}
\label{lemMakeIntersectionEssential}
Let $F$ and $G$ be properly embedded essential surfaces in an irreducible manifold with incompressible boundary.   Then $G$ can be isotoped so that they intersect essentially.
\end{lemma}

%\eric{reference for this? or provide argument?   there is one that can be uncommented out, but it would need editing}
%\jirka{\confirm{Edited that proof and it looks OK to me.}}
%
\begin{proof}
Assume that we have isotoped $G$ to minimize the number of curve components
in $F \cap G$.  We will show by contradiction that $F$ and $G$ intersect essentially. 
%\uli{I changed ``isotoped $F$'' to ``isotoped $G$'', since that's what the lemma speaks about. Also, if we start with this assumption then, formally, maybe we should say that the remainder of the proof proceeds by contradiction?} \eric{made some edits}

We first note that if there is an intersection curve that is inessential in $F$, then there is an intersection curve that is inessential in $G$ and vice-versa:   If an intersection curve bounds a disk in $F$, choose one whose disk is innermost.   Since $G$ is incompressible,  this disk is not a compressing disk for $G$ and it follows that its boundary, an intersection loop, is inessential in $G$.  The same observation applies to inessential intersection arcs. %\eric{added this paragraph, it buys symmetry between $F$ and $G$}

Then, assuming that some intersection loop is trivial, we can pass to one that is innermost on $F$, i.e., choose $\alpha$ to be an intersection loop that bounds a disk $D \subset F$ whose interior is disjoint from $G$. Since $G$ is not compressible, $\alpha$ also bounds a disk $D' \subset G$.  The union $D \cup D'$ is a sphere that,
because $X$ is irreducible, bounds a ball.  And there is an isotopy of $G$ that is restricted to a neighborhood of $D'$, and that pushes $D'$ across the ball and past $D$. This isotopy of $G$ eliminates $\alpha$ and any other intersection curves in the interior of $F \cap D'$,
and it does not introduce any new intersection curves since $\alpha$ was innermost.
%See Figure~\ref{lemLeafDisks}.

%While we have described an isotopy of $G$, we can also achieve the same result by an isotopy of $F$ that pushes the portions of
%$F$ inside the ball through $D'$ and outside the ball.
%\uli{There seems to be some inconsistency here regarding whether we want to isotope $F$ or $G$. More importantly, maybe the ``reversal'' (passing from an isotopy of one of the surfaces to one of the other surface) should be described a bit more carefully? If we start with an innermost trivial loop in $G$ then the analogous argument as above yields an isotopy of $F$.
%In order to get one for $G$, we have to push the portion of $G$ inside the ball through $D$, but in a careful way so as to avoid introducing self-intersections of $G$, since in this case $D$ need not be an innermost disk. Maybe the cleanest way would be a general theorem that says that
%an isotopy of a properly embedded disk in a $3$-ball always comes from an ambient isotopy of the ball that keeps the boundary fixed?
%I think this follows from 3D Schoenflies plus Alexander--Tietze (a PL homeomorphism of a 3-ball keeping the boundary fixed comes from a PL isotopy of the ball keeping the bd fixed. Maybe there is an easier way? Or maybe this is too much detail anyway?}
%\eric{have edited the proof so that the isotopy is always of $G$.   hopefully clearer this way.   Please remove comments if it is ok.}

Now assume some intersection arc is trivial in one of the surfaces, and as noted, we can let $\alpha$
denote such an arc that is outermost in $F$.  That is, $\alpha$ cuts off a disk $D \subset F$ whose interior is disjoint from $G$ and whose boundary meets $\partial X$ in an arc.     And $\alpha = D \cap G$ cuts off a, not necessarily outermost, disk $D' \subset G$ that also meets $\partial X$ in an arc.

 The union $D \cup D'$ is a disk with its boundary in $\partial X$ that,
because $\partial X$ is incompressible, bounds a disk $D'' \subset \partial X$.  Since $X$ is irreducible, $D \cup D' \cup D''$ is a sphere bounding a ball.   Moreover, there is an isotopy of $G$ that pushes a neighborhood of $D'$ past $D$ and outside the ball.
\end{proof}

\subsection{Almost meridional surfaces}

Suppose that $X$ is an irreducible manifold with incompressible boundary that is embedded in $S^3$.   We recall that
an essential curve $\mu \subset \partial X$ is a \emph{meridian} if
it bounds a disk in $S^3 \setminus X$.   A properly embedded surface is
\emph{meridional} if each of its boundary curves is a meridian,
and \emph{almost meridional} if all but exactly one of its boundary curves
is a meridian.

Let $D$ be a boundary compressing disk for an orientable surface  $P$.  Then $\partial N(P \cup D)$ is a surface with at least two components.  One component is isotopic to $P$; let $P'$ be the union of the other components.  Then $P'$ is said to be the result of \emph{boundary compressing} $P$ along $D$.

\begin{lemma}
\label{lemCompressAlmostMeridional}
Suppose that a manifold $X$ is   embedded  in $S^3$.   If $P$ is a connected   almost meridional planar surface properly embedded in $X$, then any surface $P'$ obtained by boundary compressing  $P$  contains an almost meridional component.
\end{lemma}

\begin{proof}
Let $P'$ be obtained from $P$ by boundary compressing along the  disk $D$.
What happens to $\partial P$?   The disk $D$ meets at most two boundary components of $\partial P$.  Any component not met by $D$ has two parallel copies in $\partial N(P \cup D)$, one for $P$ and one for $P'$, so those are unchanged.   Let $\beta$ be the one or two loops meet by the arc $x = D \cap \partial X$.   Since $D$ lies on one side of the the 2-sided planar surface $P$, when $\beta$ is a single loop, $x$ approaches it twice from the same side.  It follows that $\partial N(x  \cup \beta)$ is a \emph{pair of pants}, i.e., an $S^2$ with three holes
bounded by loops.
 One of these loops belongs to $P$ and two to $P'$, or vice-versa.

%As observed, this operation makes a 2-1 or 1-2 trade of the loops on the boundary of a pair of pants in $\partial X$.
If any two of these three loops are meridians,  then so is the third, since
it bounds a disk, namely the union of the pants and the two disks
pushed slightly into $S^3 \setminus \mathrm{interior}(X)$.

We apply this ``two meridians implies three meridians'' principle to show that $P'$ has an almost meridional component, regardless of how the boundary
compressing disk meets the boundary components of $P$.

If the boundary compressing disk meets the non-meridional component twice, then the compression eliminates the non-meridional curve, and creates two new curves, each belonging to a separate component of $P'$.   At  least one of the new curves is not meridional, and hence its component is almost meridional.

If the boundary compressing disk meets a meridian and  the non-meridian, then the compression does not separate $P$, and trades these curves for a new non-meridional curve.  Thus $P'$ is almost meridional.

If the boundary compressing disk meets two distinct meridians, then they are eliminated and  a new one is created.  The connected surface $P'$ is almost meridional.

If the boundary compressing disk meets a single meridian twice, then $P'$ has two components, each with one of the two new curves, either both meridional or both non-meridional.   If both are meridional, then the component with the original non-meridian on its boundary is almost meridional.  If both are non-meridional, then the component without the original non-meridian is almost meridional.   One of the two components of $P'$ is almost meridional.
\end{proof}

\begin{lemma}
\label{lemEssentialAlmostMeridional}
Suppose that $X$, an irreducible manifold with incompressible boundary,  is embedded in $S^3$.
If $X$ contains an incompressible, almost meridional planar surface, then $X$ contains an essential almost meridional planar surface. \end{lemma}

\begin{proof}
An incompressible almost meridional surface can be sequentially boundary compressed until it is incompressible and boundary incompressible.  By the prior lemma, each surface in the sequence, hence the final one,  has an almost meridional component.   This final component is not a disk because $X$ is boundary incompressible. Hence it is an essential almost meridional planar surface.
\end{proof}

\subsection{Strongly irreducible surfaces}

A two-sided surface properly embedded in $X$ is \emph{bi-compressible} if it has a \emph{compressing pair} $(D_+,D_-)$, a pair of disks, each a compressing or boundary compressing disk, one for each side of the surface.  The pair is \emph{simultaneous} if $\partial D_+ \cap \partial D_- = \emptyset$.

A surface is \emph{weakly reducible} if it is simultaneously bi-compressible using compressing disks only.   A \emph{strongly irreducible} surface is one that is bi-compressible using compressing disks but not simultaneously so.   A surface is \emph{boundary weakly reducible} if it is simultaneously bi-compressible using any combination of compressing disks and boundary compressing disks.  A surface is \emph{boundary strongly irreducible} if it is bi-compressible, using any combination of compressing or boundary compressing disks,  but not simultaneously so.

Some of our results assume that a surface is both strongly irreducible \emph{and} boundary strongly irreducible.   It may seem that the strongly irreducible hypothesis is vacuous.   But this is not the case---it guarantees that the surface has at least one (non-boundary) compressing disk for each side.

%\eriC{are these definitions comprehensible?, they are stated a bit differently than usual (using bi-compressible and simultaneous) but are equivalent.  Trying to make them more natural, not sure it works.}

\begin{lemma}[\cite{bachmanDerbyTalbotSedgwick}, Lemma~3.8]
\label{lemSIBSIhasEssentialBoundary}
In an irreducible manifold with incompressible boundary, the boundary of a strongly irreducible surface is essential in the boundary of the manifold.
\end{lemma}

We state here a special case of Lemma~4.2 of \cite{BachmanStabilizingAndDestabilizing}.  This

\begin{lemma}[Lemma 4.2 of \cite{BachmanStabilizingAndDestabilizing}]
\label{lemSurfacesMeetEssentially}
Let $F$ be an essential surface and $G$ a surface that is strongly irreducible and boundary strongly irreducible.  Then $G$ may be isotoped so that $F$
and $G$ intersect essentially.
\end{lemma}

%
%It is stated there that the proof is a direct generalization of Corollary 3.8 of \cite{BachmanTopologicalIndexTheory}.NEED A CITATION OR SIGNIFICANT PROOF HERE: This is  a bit tricky.   It is certainly true and we could give a proof at the cost of probably 2 pages.  It is essentially the same proof given in Bachman, Derby-Talbot and Sedgwick.
%
%It is also a special case of Lemma 4.2 of
%\url{http://arxiv.org/pdf/1201.3438v1.pdf} by Bachman which appeared in Math. Annalen.  Unfortunately he does not give a proper proof there.   He has however supplied me with a manuscript that does give a proper proof, but it is not submitted anywhere at the moment.   WILL WORRY ABOUT THIS LATER.

Let us remark that Bachman does not give a proof but
claims it to be a direct generalization of \cite[Corol.~3.8]{BachmanTopologicalIndexTheory}. He has also provided us with an unpublished manuscript
with a proof.

%\eriC{Recall that Bachman does not give a proof there but states that it is a direct generalization of  \cite[Corol.~3.8]{BachmanTopologicalIndexTheory}.   He has also give me a manuscript of the proof which I have not read.   And it also due to the same methods in another paper.   I could also give a proof in a couple of pages.   Are you ok with what is written or do you want me to add some of this to the "proof"?}

\section{Theory of normal curves and surfaces in a marked triangulation}
\label{s:almostnormal-expl}
\label{s:normal-surfaces}

%
%A \emph{triangulation} of a manifold $X$ is an expression of $X$ as a quotient space $X = \bigcup_{i=1,..,t} \Delta_i/\sim$, where $\{\Delta_i\}_{i=1,..t}$ is a collection of $t$ tetrahedra and $\sim$ is a \emph{face pairing}, a set of equivalences between pairs of faces.  Note that we allow tetrahedra to have self identifications.
%
%A \emph{marked tetrahedron} is tetrahedron with a finite collection of points \emph{marked}, designated, along its edges.

%Fix once and for all, an enumeration of the tetrahedra of $\T$, within each tetrahedron an enumeration of its edges, and within each edge an enumeration of its sub-edges as decomposed by the marking $M$.    Label each (tetrahedron, edge, sub-edge) triple.     Compared lexicographically, the set of triples is well ordered, a fact we make use of when defining the complexity of curves.

In this section we introduce a mild generalization of the theory
of normal curves and surfaces.

\begin{defn}
A  \emph{marked triangulation} is a pair $(\T,M)$ consisting of a  triangulation $\T$ of a $2$- or $3$-manifold along with a \emph{marking} $M \subset \T^1$, a finite set of points  along the  edges of $\T$.
\end{defn}

 If $M = \emptyset$, then $(\T,M)$ is a triangulation in the usual sense and we will usually omit $M$ and refer directly to $\T$.   Similarly, when $M = \emptyset$, we will describe objects as being normal rather than $M$-normal, and note that our definitions restrict to the standard ones.

An arc in a triangle is \emph{$M$-normal} if its endpoints lie in distinct edges of the face and it misses $M$.   A properly embedded curve $\alpha$ in the boundary surface is \emph{$M$-normal} if it is the union of $M$-normal arcs.   The length of $\alpha$, $\ell(\alpha) = |\alpha \cap \T^1|$, is its number of intersections with the 1-skeleton.

There are several types of elementary surfaces contained in a tetrahedron $\Delta$.  An \emph{$M$-normal disk} is a disk in $\Delta$ whose boundary is an $M$-normal curve of length 3 or 4 in $\partial \Delta$.   We also consider two types of \emph{$M$-exceptional pieces}:  An octagon is a disk in $\Delta$ whose boundary is an $M$-normal curve of length 8 in $\partial \Delta$.   A \emph{tube} is an unknotted annulus in $\Delta$ whose boundary consists of two $M$-normal curves whose total length is at most~8.
\immfigw{almostnormal}{3in}

An \emph{$M$-normal surface} is a properly embedded surface that is the union of $M$-normal disks.   An \emph{almost $M$-normal surface}  is a properly embedded surface that is the union of a single $M$-exceptional piece and a collection of $M$-normal disks.

The weight of an (almost) $M$-normal surface $A$ is $\wt(A) = |A \cap \T^1|$, the number of intersections with the 1-skeleton.   Its length is the length of its boundary: $\ell(A) := \ell(\partial A)$.

An \emph{$M$-normal isotopy} is a normal isotopy that does not pass through any point in $M$.   An \emph{$M$-type} is the equivalence class of an $M$-normal arc in a face, or, an $M$-normal disc or $M$-exceptional piece in a tetrahedron.  Two types are \emph{$M$-compatible} if they have disjoint representatives.  A pair of curves or surfaces are \emph{$M$-compatible} if each pair of types they possess are $M$-compatible.   That is, a pair of curves is $M$ compatible if, for each face of the triangulation, their arcs in that face are pairwise disjoint after an $M$-normal isotopy.   An analogous statement holds for surfaces.

We note that $M$-compatibility is a local condition; in general
it may not be possibly to make $M$-compatible curves or surfaces
globally disjoint by an $M$-normal isotopy.

The \emph{$M$-normal vector} or \emph{$M$-normal coordinates} of an $M$-normal curve, surface, or almost $M$-normal surface $A$ is a uniquely determined vector $\vv_M(A)$, indexed over the set of normal types and with each entry recording the number of $M$-normal objects of the index type.

If $A,B,C$ are $M$-normal surfaces such that
$\vv_M(C)=\vv_M(A)+\vv_M(B)$, then $C$ is an \emph{$M$-normal
sum} of $A$ and $B$, and we write $C=A+B$. The same definition applies
if $A$ is $M$-normal and $B$ and $C$ almost $M$-normal,
or if $A,B,C$ are $M$-normal curves.

We note that not every two $M$-normal surfaces, for example, can be
normally added---this is possible exactly if they
are $M$-compatible.

 If $A$ and $B$ are $M$-compatible, then one can construct
an $M$-normal sum as follows.
 In each face or tetrahedron $\Delta$, the $M$-normal pieces $A \cap \Delta$
and $B \cap \Delta$ can be $M$-normally isotoped to be disjoint,
and then attached across each facet of $\Delta$  to the pieces in an
adjacent face/tetrahedron.    This produces a properly embedded
$M$-normal curve, $M$-normal surface, or almost $M$-normal surface,
respectively, which is the $M$-normal sum.

However, in our considerations,
we will mostly use a different geometric construction of
an $M$-normal sum, where we assume that the curves or surfaces
in question intersect minimally, in a suitable sense, but then
we do not isotope them to be disjoint as above, but rather
they stay in place and we deal with their intersections as well;
see Section~\ref{secGeomSum} below.

It is well known that Euler characteristic, weight and length are all additive with respect to normal sum, and this works without change
for the $M$-normal case.
 If $A$ and $B$ are compatible (almost) $M$-normal curves or surfaces then the following hold:
\begin{enumerate}
\item $\chi(A+B) = \chi(A) + \chi(B)$
\item $\wt(A+B) = \wt(A) + \wt(B)$
\item $\ell(A+B) = \ell(A) + \ell(B)$.
\end{enumerate}

An (almost) $M$-normal curve or surface is \emph{fundamental} if it cannot be expressed as the sum of other (almost) $M$-normal curves or surfaces.  Every (almost) $M$-normal curve/surface is a non-negative integer combination
of fundamentals.

\begin{figure}[ht]
\begin{center}
\includegraphics{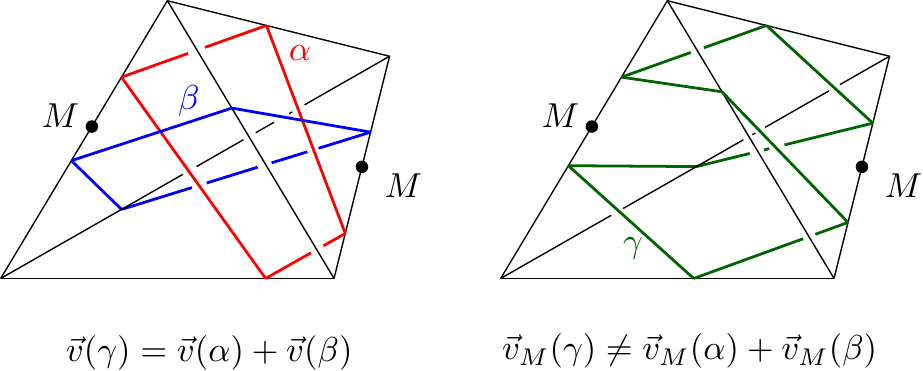}
\end{center}
\caption{A length 8 curve that is not the sum of two length 4 curves in the marked triangulation.}
\label{figInCompatible}
\end{figure}

Here $M$-normal curve theory differs from standard normal curve theory.   While all normal curves are compatible, $M$-normal curves have distinct compatibility classes, and this increases the number of fundamentals.   In  Figure \ref{figInCompatible}, we see the boundary of a tetrahedron with two marked points, one on each of a pair of opposite edges.    Let $\gamma$ be the length 8 $M$-normal curve that meets each of the  sub-edges once.   As a normal curve $\gamma$ is not fundamental---it is the sum of the two distinct length 4 curves $\alpha$ and $\beta$. But in the marked triangulation these curves are incompatible and $\gamma$ is fundamental.
%This distinction becomes quite important in
%Section~\ref{sec:distinction-Mnormal-quite-important}.

If $\alpha$ is an (almost) $M$-normal curve or surface, then $\vv_M(\alpha)$ is a solution to a set of \emph{matching equations}:  For a triangulated surface, this set consists of  one equation for each sub-edge in the interior of the surface.   It sets equal the sum of those coordinates meeting the sub-edge from one side to the sum of those meeting it on the other side.  In a triangulated 3-manifold, the set of matching equations consists of one equation for each $M$-normal arc type contained in an interior face.  This equation sets equal the sum of the coordinates for elementary types using the arc type on one side to the sum of those using it on the other side.

We say that a vector $\vv_M$ of the correct dimension is \emph{$M$-admissible} if all its coordinates are non-negative, it satisfies the matching equations, and is \emph{self-compatible}, i.e., it does not possess non-zero coordinates for any pair of non-$M$-compatible types.   If $\vv_M$ is an $M$-admissible vector, there is an (almost) $M$-normal curve/surface $\alpha$ for which $\vv_M = \vv_M(\alpha)$.

The following proposition is a straightforward generalization of
a well known fact from normal surface theory to $M$-normal surfaces;
see \cite{hass-lagarias-pippenger}
for a nice exposition.

\begin{proposition}\label{propFundamentals}
Given a $3$-manifold with a marked triangulation $\T_M$,
the set  $\F$ of fundamental $M$-normal surfaces
is computable, and both $|\F|$ and the maximum
weight $\wt(F)$ of an $F\in\F$ are bounded by a computable
function of $t$ and~$m$.
\end{proposition}

The bound on $|\F|$ and
$\wt(F)$ has the form $\exp(p(t,m))$, where $p(t,m)$ is a suitable polynomial.

\begin{proof} It is well known that without the marking,
there are $7t$ normal disk types, $3t$ exceptional octagons and $25t$ exceptional tubed pairs of disks. Moreover, the presence of a tubed pair of disks
may split one type of normal disks into two,
but certainly we have no more than $42t$ types in total.

The points of $M$ divide each edge into at most $m+1$ subarcs.
In order to specify an $M$-normal type of a triangle, for example,
we need to specify the subarc containing each of the three vertices,
which leads to the bound $(m+1)^3$. The worst bound is obtained for
tubes and octagons, with $(m+1)^8$, so a rough bound for
the total number of $M$-types is
$42t(m+1)^8$.

A similar way of counting applies to the number of \emph{matching equations},
which represent compatibility of the coordinates of the $M$-normal vector
across the pieces of the edges of $\T$ delimited by the points of $M$.
Indeed, the matching equations correspond to
$M$-arc types.   There are at most $4t$ interior faces,
 each with $3$ underlying normal arc types.
 A given $M$-arc type is thus determined by this normal type and by
the sub-arcs it meets, and so there are at most $12t(m+1)^2$ matching
equations.

Then, reasoning as in \cite[Sec.~6]{hass-lagarias-pippenger}, using a Hilbert basis of the appropriate
integral cone, we obtain the bounds of the claimed form.
\end{proof}

%\eriC{there is a subtlety here.  There are 7t disk types. But if the almost normal surface is a tube between two disks of the same type, the it can split the normal disks ---of that type in two, one to each side of the tubed surface(the tubed piece acts like a marking).  many ways to fix it, the easiest is to just double the 7t to 14t }

%In $\T$, without the marking $M$, it is well known \cite{???} that there are $7t$ normal disk types, $3t$ exceptional octagons and $25t$ exceptional tubed pairs of disks, for a total of $35t$ types.   The marking $M$ increases this number considerably.      Each edge is sub-divided into at most $m+1$ sub-edges.    Each $M$-type is unique after determining its underlying normal type and the sub-edges it meets.  As all of the elementary types have at most length 8, each type meets at most $8$ sub-edges.    For each underlying normal type, there are at most $(m+1)^8$ $M$-types, for a total of fewer than $35t(m+1)^8$ $M$-types.   Thus each vector $\vv_M$ has at most $35t(m+1)^8$  coordinates. \eriC{any point in this, or just point out it is computable}

\subsection{Snug pairs of curves and surfaces, Haken sums, and normal sums}
\label{secGeomSum}

The normal sum $F+G$ of a pair of (almost) $M$-normal curves or surfaces $F$ and $G$ has been defined, if they are $M$-compatible, to be an $M$-normal
surface whose $M$-normal vector is the sum of the $M$-normal vectors
of $F$ and $G$, $\vv_M(F+G) = \vv_M(F) + \vv_M(G)$.

It is desirable to show that qualities of the sum, such as essentiality or minimality, also apply to the summands.   Here we describe a well known geometric interpretation of the sum that makes this possible; also see, for example, \cite{Jaco:An-algorithm-to-decide-if-a-3-manifold-is-a-Haken-manifold-1984,Jaco:Algorithms-for-the-complete-decomposition-of-a-closed-3-manifold-1995}. We also present some
related material.

\heading{Snug pairs. }
We begin with a definition of a ``placement with
no unnecessary intersections'' for a pair of curves or surfaces.

\begin{defn}
A pair $(F,G)$ of properly embedded curves or surfaces is \emph{snug}
if it is transverse and the number of components of the intersection
$F\cap G$
is minimized over pairs $(F',G')$, where $F'$ and $G'$ are isotopic to $F$
and $G$, respectively.
The pair $(F,G)$ is \emph{locally snug} if $F \cap G$ is disjoint from the 1-skeleton $\T^1$, and, they are snug in
the interior of each simplex of the triangulation (here we only allow isotopies moving each intersection of $F$ or $G$ with a face only within that face).
\end{defn}

If $F$ and $G$ are locally snug $M$-normal surfaces then it follows that:
\begin{enumerate}
%\item $F$ and $G$ are transverse;
%\item $F \cap G$ is disjoint from the 1-skeleton $\T^1$;
\item each pair of $M$-normal arcs, one from $F$ and one from $G$, meets in 0 or 1 points;
\item each pair of $M$-normal disks, one from $F$ and one from $G$,
 meets in 2 or fewer arcs, and the union of the arcs has at most one
endpoint in any face;
\item no loop of $F\cap G$ lies inside a tetrahedron.
\end{enumerate}

Any pair of compatible $M$-normal curves or surfaces can be made locally snug by $M$-normal isotopies that first make their intersections with edges disjoint and then ``straighten'' them so that: normal arcs are straight, normal triangles are flat, and normal quads are the union of two flat triangles.   We do not define locally snug when $F$ is an almost normal surface and $G$ is a normal surface, for in that case we require only the definition of the normal sum $F+G$  and not its geometric interpretation.

%      In that case, instead of defining locally snug, we refer the reader to \cite{Schleimer:Sphere-recognition-lies-in-NP-2011}.

%For an almost normal surface and a normal surface, we also include an $M$-normal isotopy that separates the exceptional piece from other components.
%
%\eric{don't we want to remove that last sentence about almost normal?}

\heading{Haken sum and normal sum of curves. }
Now, for a while, we deal only with curves, and we develop a geometric
interpretation of their normal sum. Here we consider only
unmarked triangulations, i.e., $M=\emptyset$.

Let $D$ be a regular neighborhood of an intersection point $x$ of a pair of
transverse curves $\alpha$ and $\beta$.  We can remove the intersection by
deleting the arcs in the interior of the disk and then attaching $\alpha$ to
$\beta$ along a pair of antipodal sub-arcs of $\partial D$.
Thus, we replace the
``$\times$'' in $\alpha \cup \beta$ with either ``$)($'' or ``$\asymp$''.
 This is called an
\emph{exchange} or a \emph{switch} at $x$.  A curve is said to be a \emph{Haken sum}
$\alpha \hs \beta$ of $\alpha$ and $\beta$ if it is obtained by an exchange at
each of their intersection points.  Of course, $\alpha \hs \beta$ is dependent
on the direction of the switches and is therefore not well determined.

If, however, $\alpha$ and $\beta$ are locally snug normal curves, then each
intersection point is of the form $x = \alpha' \cap \beta'$ where $\alpha'$ and
$\beta'$ are normal arcs in some face.   Then $\alpha'$ and $\beta'$ meet at least one common edge $e$ of the face.   The  \emph{regular exchange} is the exchange that does not produce an \emph{abnormal arc}, a non-normal arc with both endpoints attached to~$e$; see Fig.~\ref{figgeometricSum} top.

\begin{figure}[tb]
\begin{center}
\includegraphics{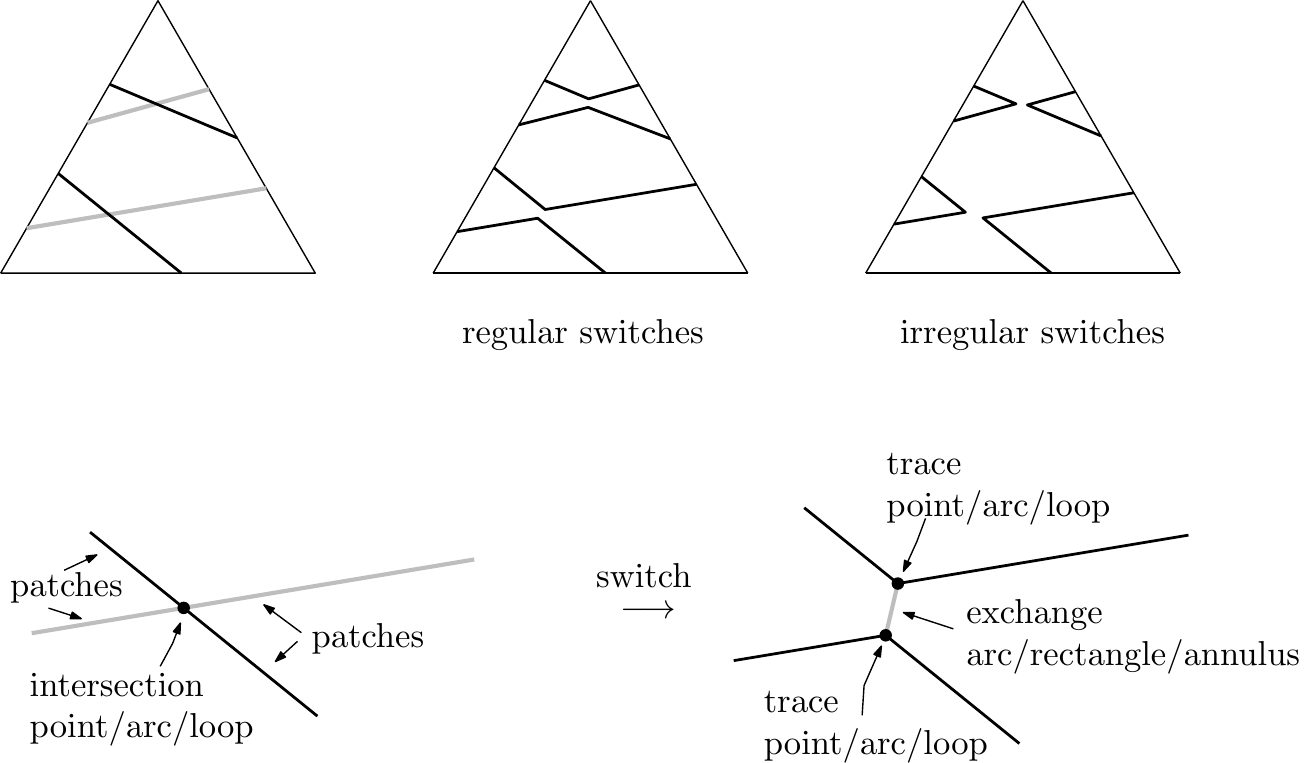}
\end{center}
\caption{Switches for curves and surfaces}
\label{figgeometricSum}
\end{figure}

As we will see, the normal sum of $\alpha+\beta$ of locally snug curves
can be obtained by doing all the regular exchanges.

\begin{lemma}
\label{l:regular=>normal}
Let $\alpha$ and $\beta$ be locally snug normal curves.  Then
the Haken sum $\alpha\hs\beta$ obtained by making all the regular
exchanges is the normal sum $\alpha+\beta$;
i.e., $\vv(\alpha+\beta) = \vv(\alpha) + \vv(\beta)$.
\end{lemma}

\begin{proof}
We show that the result holds in each face of the triangulation.   In an abuse of notation, let $\alpha$ and $\beta$ be restriction of the curves to a particular face.   For contradiction suppose that they are a counterexample that minimizes $|\alpha \cap \beta|$.   Then $\alpha$ and $\beta$ are not disjoint, for in that case, the  union is normal and normal vectors add.

Since they intersect in a face, we can identify an outermost half-bigon bounded by a sub-arcs of $\alpha$ and $\beta$ and an edge of the face; see Figure \ref{figBigons}.   The regular exchange trades these sub-arcs and results in a pair of normal curves, $\alpha'$ normally isotopic to $\alpha$ and $\beta'$ normally isotopic to $\beta$,  that are locally snug but with fewer intersections.  By assumption, these $\alpha'$ and $\beta'$ satisfy the conclusion, hence so do $\alpha$ and $\beta$.
\end{proof}

\begin{lemma}
\label{lemHakenSum}
Let $\alpha \hs \beta$ be a Haken sum of locally snug properly embedded curves.   Then $\alpha \hs \beta$ is normal if and only if $\alpha$ and $\beta$ are normal and all switches are regular,  i.e., $\alpha \hs \beta = \alpha+\beta$.
In addition, if $\alpha$ and $\beta$ are normal and $\alpha \hs \beta$ contains
at least one irregular switch, then $\alpha \hs \beta$ contains an abnormal
arc.
\end{lemma}

\begin{proof} ($\Leftarrow$)  This is by Lemma~\ref{l:regular=>normal}.   ($\Rightarrow$)  We show that if either $\alpha$ or $\beta$ is not normal then neither is $\alpha \hs \beta$ for any Haken sum of the curves.

Suppose then that  $\alpha' \subset \alpha$ is an outermost non-normal arc, one that co-bounds a disk with a sub-arc of an edge $e' \subset e$.    If $\beta$ meets the disk, it meets it in a collection of $n$ arcs, each with one endpoint in $\alpha'$ and one endpoint in $e'$,
because $\alpha'$ is outermost and $\alpha$ and $\beta$ are snug.
Let $D$ be a regular neighborhood of the disk.  Then, regardless of the switches, $\alpha \hs \beta$ meets $D$ in a collection of $n+1$ arcs that have $n+2$ endpoints along the edge and $n$ endpoints not on the edge.  It follows that at least one arc meets the edge in 2 points and is not normal. A symmetric argument applies if the outermost non-normal arc belongs to $\beta$.  Nor can either $\alpha$ or $\beta$ possess a loop in a face.  Local snugness implies that any loop is disjoint from the other curve and survives any Haken sum.

We now know that $\alpha$, $\beta$ are normal. To conclude the proof, it is
sufficient to show that $\alpha \hs \beta$ contains an abnormal arc if at least
one switch is irregular (this contradicts the normality of
 $\alpha + \beta$, and
thus proves the last claim of the lemma). In
an abuse of notation, let $\alpha$ and $\beta$ refer to the collection of
normal arcs in a particular face.

We perform the specified switches in order
according to the following scheme: If $\alpha \cap \beta \neq \emptyset$,
 then $\alpha$ and $\beta$
form an outermost half bigon $B$ with an edge as in Figure~\ref{figBigons}.
The regular switch
produces collections $\alpha'$ and $\beta'$ that are normally isotopic to
$\alpha$ and $\beta$, but with one fewer intersections.  An irregular switch
produces a disjoint abnormal arc that survives any and all additional
exchanges.  If each exchange is regular, we can continue and the process
produces a disjoint union $\alpha \sqcup \beta.$    If any exchange is not
regular, the resulting curve contains an abnormal arc.
\end{proof}

\begin{figure}[tb]
\begin{center}
\includegraphics{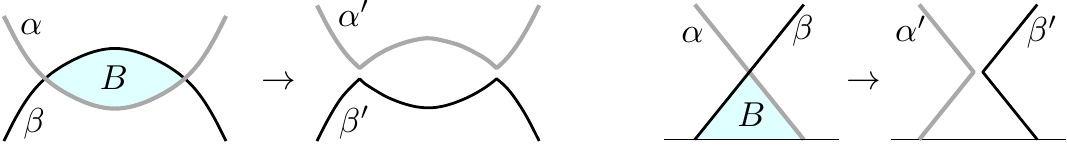}
\end{center}
\caption{Exchanges at the corners of a bigon and a half-bigon.}
\label{figBigons}
\end{figure}

\heading{Normal sign. }
When $\alpha$ and $\beta$ lie in an oriented surface, for example the boundary
of an oriented manifold, we can define the \emph{normal sign} of each point of
$\alpha \cap \beta$.   Viewing $\alpha$ as horizontal and $\beta$ as vertical,
the regular exchange at the point connects a pair of quadrants.  The point has
\emph{positive sign} if the exchange connects the southwest quadrant to the
northeast quadrant, and it has \emph{negative sign} if it connects the northwest to
the southeast; see Figure~\ref{f:exchanges}.   This is equivalent to the definition given in \cite{bachmanDerbyTalbotSedgwick}.
The definition depends on the ordering of the pair of
curves and on an orientation on the surface: reversing the order or the orientation reverses every sign.

\begin{figure}[ht]
\begin{center}
\includegraphics{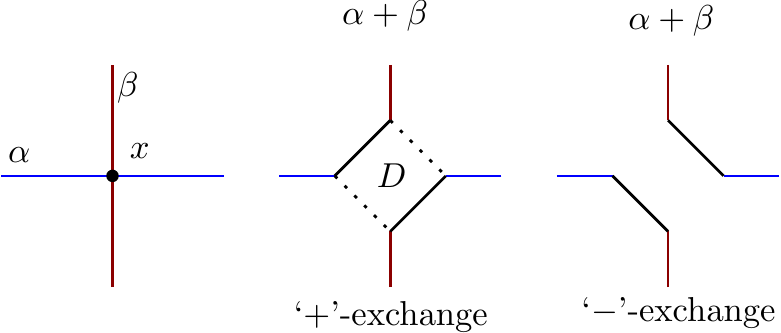}
\end{center}
\caption{Exchanges with positive and negative sign.}
\label{f:exchanges}
\end{figure}

\heading{Normal sum of surfaces. } Similar to the case
of curves above, one can also construct the normal sum $F+G$ of normal surfaces
geometrically, using suitable switches.  We assume that $F$ and $G$
are locally snug.

We construct $F+G$ by specifying its intersection with the 1-, 2-, and 3-skeleta
of the triangulation, respectively.   First, we let the intersection of
$F+G$ with the 1-skeleton to be the union of the intersection points from $F$
and those from $G$.

Second, in each face we perform regular switches on all intersecting
pairs of arcs $(f,g)$, where $f$ comes from $F$ and $g$ from $G$.

Finally, we construct the normal sum $F+G$ in the interior of each
tetrahedron $T$.   As discussed earlier, each normal disk is either a flat triangle or a quadrilateral made of 2 flat triangles.   It follows that every intersection between normal disks from compatible surfaces is either 1 or 2
arcs,  not necessarily straight.  Compatibility ensures that the regular switches prescribed at the endpoints of each arc are consistent with each other and can be extended across the entire arc of intersection.  The normal sum $F$ and $G$ is the result of performing such regular switches along every such arc of intersection.

Note that any intersection arc between normal disks can be extended from a tetrahedron through a face to a neighboring tetrahedron.   In its entirety this \emph{intersection curve} between $F$ and $G$ is either a loop, or an arc with both endpoints in $\partial X$.   Compatibility ensures that the regular switches in each face and through the interior of each tetrahedron agree.  Thus we can regard the switch as a regular switch along the entire intersection curve.

\heading{Exchange arcs and surfaces, trace curves.}
Here we introduce some additional terminology.
First, we consider a regular switch of two curves.
Inside the neighborhood where the regular switch was performed,
 we identify an \emph{exchange arc} that connects  the points of
 the newly formed arcs
corresponding to the former intersection points;
see Figure~\ref{figgeometricSum}.

Next, we consider two locally snug normal surfaces $F$ and $G$.
A \emph{patch} is a component of $F \cup G \setminus F \cap G$.  Regular switches reconnect the patches, and \emph{trace curves} are the seams between patches after performing regular switches along all intersection curves; see
Figures~\ref{figgeometricSum} and~\ref{figExchangeBand}.

\begin{figure}[h]
\begin{center}
\includegraphics[width=5in]{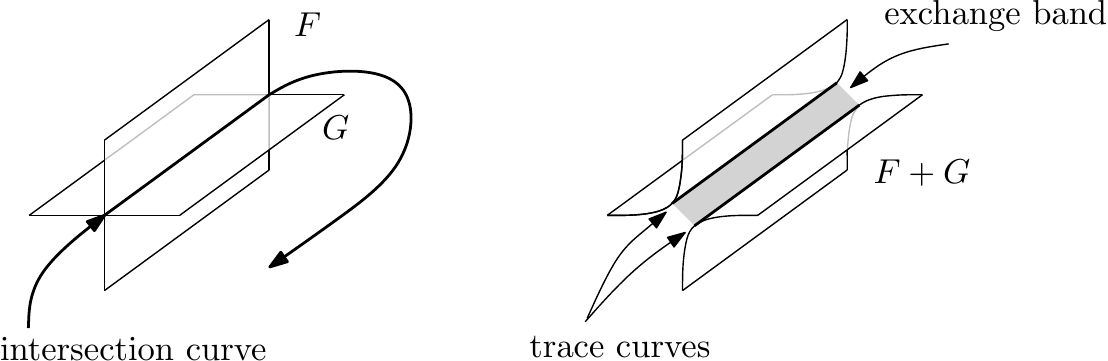}
\end{center}
\caption{The exchange between surfaces along an intersection curve.}
\label{figExchangeBand}
\end{figure}

If the intersection curve $\alpha$ is an arc, then, after performing a regular switch, we can identify an \emph{exchange rectangle}, a rectangle whose top and bottom, say, are bounded by trace arcs and whose left and right sides are \emph{exchange arcs} lying in $\partial X$.

If $\alpha$ is a loop, then our assumption that $X$ is orientable means that a regular neighborhood of $\alpha$ is a solid torus, not a solid Klein bottle.  Again, since $X$ is orientable, $\alpha$ is either orientation preserving in both $F$ and $G$, or, orientation reversing in both $F$ and $G$.   In the former case, there is an \emph{exchange annulus}, a zero-weight annulus bounded by the trace curves and with core $\alpha$.  In the latter case, there is a single trace curve which bounds an \emph{exchange M\"{o}bius band} (we will be able to
exclude this case in our proofs, though).
% This will not occur in our application, since either $F$ or $G$ will be orientable.
%\eriC{ we could have exchange M-bands.  But the only time we need to prove anything, sec 5 Prop??, we can assume not. It is mentioned there so I think we can remove "This will not occur ..."}

As observed in \cite{Hatcher-bsurf} and, in the context of normal surfaces, in \cite{Jaco:Decision-problems-in-the-space-of-Dehn-fillings-2003}, every intersection arc between surfaces connects intersection points of the boundary curves that have opposite normal sign:

\begin{lemma}
\label{lemPlusToMinus}
Let $A = B + C$ be a normal sum of surfaces in an orientable manifold $X$ with an induced orientation on $\partial X$.  Then every arc in $B \cap C$ joins a pair of points in $\partial B \cap \partial C$ with opposite normal sign.
\end{lemma}

%\eriC{are you ok with the cite to Hatcher/JS?  if not we can see that it is true by inspection for each pair of compatible normal disk types in a tetrahedron.  Follow the intersection arc through multiple tetrahedra and compute the ``sum'' of the signs.   Since the tetrahedra on the boundary of a face induce opposite orientations on that face, the signs on interior cancel, and the opposite signs are pushed to the endpoints of the arcs.   }

\section{Complexity and tight curves}\label{s:tightsnug}

In this section,  we consider properly properly embedded curves in a triangulated surface.   We assume that they are transverse to the 1-skeleton but, a priori, they are not assumed to be normal.

Fix, once and for all, an ordering of all normal arc types of the triangulated surface.  For this purpose we \emph{do not} take into account any marking present.  As in the previous section, a normal curve $\alpha$ determines a vector $\vv(\alpha)$ which records the number of normal arcs of the indexed type.  Order these normal vectors lexicographically.

Recall that the length of a properly embedded curve $\alpha$ is the number of intersections with the 1-skeleton, $\ell(\alpha) = |\alpha \cap \T^1|$.     We say that a curve is \emph{least length} if it minimizes length over all curves to which it is isotopic.

\begin{lemma}
A least length essential curve is normal.
\end{lemma}

\begin{proof}
A loop in face demonstrates that the curve is not essential and any abnormal arc is either inessential or yields an isotopy reducing the length.
\end{proof}

If $\alpha$ is a normal curve, then we define its \emph{complexity} to be the pair consisting of its length and its normal vector,
\[
\compl(\alpha):=(\ell(\alpha),\vv(\alpha)).
\]
We reiterate that we do not take into account any marking $M$ in the definition of complexity.   If $\alpha$ is not normal, we define its complexity to be $\compl(\alpha) = (\ell(\alpha),\vec 0)$.   Complexities will also be ordered lexicographically.

\begin{defn}
A  curve $\alpha$ is \emph{tight} if it minimizes the
complexity $\compl(\alpha)$ over all curves to which it is isotopic.
\end{defn}

The interior of a connected inessential curve can be made disjoint from the 1-skeleton, so a tight inessential loop has $\compl = (0,\vec 0)$ and a tight inessential arc has $\compl = (2,\vec 0)$.

\begin{lemma}
\label{lemTightUnique}
A tight essential curve is normal and unique up to normal isotopy.
\end{lemma}

\begin{proof} Indeed, the complexities of two normal curves are equal if and only if their normal vectors are identical.
\end{proof}

\begin{lemma}
\label{lemHakenSumComplexity}
Let $\alpha \hs \beta$ be a Haken sum of locally snug properly embedded curves.  Then $\compl(\alpha \hs \beta) \leq \compl(\alpha)+\compl(\beta)$ with equality holding if and only if $\alpha, \beta$ and $\alpha \hs \beta$ are all normal, or, all not normal.
\end{lemma}

\begin{proof}
The curve $\alpha \hs \beta$ is constructed by performing an exchange at every
intersection point of $\alpha \cup  \beta$.   This is done away from the
1-skeleton, so we have $\ell(\alpha \hs \beta) = \ell(\alpha)+\ell(\beta)$.
Thus any difference in complexity is determined solely by the normal vectors of
the curves.   If $\alpha \hs \beta$ is  normal,  then by the previous two
lemmas $\vv (\alpha +\beta) = \vv (\alpha) + \vv( \beta)$ and equality holds.  If $\alpha \hs \beta$ is  not normal, then its normal vector is $\vec 0$.  Then complexity is additive when both $\alpha$ or $\beta$ are not normal, and sub-additive otherwise.
\end{proof}

If a tight curve is written as a sum, then the exchange arcs are essential in the complement of the curve:

\begin{lemma}
\label{lemExchangeArcNotParallel}
Suppose that a tight normal curve is written as a sum $\alpha + \beta$ of two
normal curves.  Then no exchange arc co-bounds a disk with a sub-arc of the curve.
\end{lemma}
%\martin{Added: `of two normal curves'. Hoping that it is right.}
%\eriC{yes, good}

\begin{proof}
Perform an irregular exchange only at the intersection corresponding to this
exchange arc.   The new curve is a Haken sum $\alpha \hs \beta$ with one
component a trivial loop, the rest isotopic to $\alpha + \beta$, and the same
total length.
It follows that the trivial loop has zero length otherwise $\alpha + \beta$
would not be tight. Therefore $\alpha + \beta$ and the second component of
$\alpha \hs \beta$ are normally isotopic.
However, by Lemma~\ref{lemHakenSum} there is an abnormal
arc in the second component of $\alpha \hs \beta$, a contradiction.
\end{proof}

\begin{lemma}
\label{lemBigonCurves}
Let $B$ be a bigon or half-bigon bounded by a pair of locally snug normal curves
$\alpha$ and $\beta$; see Figure \ref{figBigons}.  Let $\alpha'$ and $\beta'$ be the pair of isotopic curves obtained by corner exchange(s) that trade the sides of $B$.  Then one of the following holds:

\begin{enumerate}
\item[\rm(1)] $\compl(\alpha')=\compl(\alpha), \compl(\beta)=\compl(\beta')$, thus $\alpha'$ and $\beta'$ are normally isotopic to $\alpha$ and $\beta$, respectively, and $|\alpha' \cap \beta'|<|\alpha \cap \beta|$;
\item[\rm(2)] $\compl(\alpha')<\compl(\alpha)$;
\item[\rm(3)] $\compl(\beta')<\compl(\beta)$.
\end{enumerate}
\end{lemma}

\begin{proof}
Let $\alpha,\alpha',\beta$ and $\beta'$ be as indicated in Figure~\ref{figBigons}.  Note that the exchange doesn't add or remove intersections with the 1-skeleton, and so the total length is unchanged.   If the traded arcs differ in length then one curve increases and the other decreases in length, hence in complexity.   In this case, either (2) or (3) holds.   So we continue assuming $\ell(\alpha) = \ell(\alpha')$ and $\ell(\beta)=\ell(\beta')$.

If any exchange is irregular, then one of the curves, say $\alpha'$, is not normal.   Then its complexity $\compl(\alpha') = (\ell(\alpha'),\vec 0) = (\ell(\alpha),\vec 0) < (\ell(\alpha),\vv(\alpha))= \compl(\alpha)$ has decreased, yielding conclusion (2).  Conclusion (3) results when $\beta'$ is not normal.

We are left in the case that the exchange trades length fairly and $\alpha'$ and $\beta'$ are both normal.  Because length and normal vectors are both additive with respect to normal addition, we have $\compl(\alpha)+\compl(\beta)=\compl(\alpha+\beta)=\compl(\alpha')+\compl(\beta')$. If $\compl(\alpha')=\compl(\alpha)$, then $\compl(\beta')=\compl(\beta)$ and by Lemma \ref{lemTightUnique} the trade yields normally isotopic curves, conclusion (1).   Otherwise, either (2) or (3) holds.
\end{proof}

\begin{lemma}
\label{lemAddToSnug}
Let $\alpha$ be a tight essential curve and $\C$ set of pairwise snug,  tight essential curves.  Then, after a normal isotopy of $\alpha$, $\{\alpha\} \cup \C$ is pairwise snug.
\end{lemma}

\begin{proof}
Normally isotope $\alpha$ to minimize the total of all intersections with $\C$.
By way of contradiction, suppose some pair is not snug, that there is $\beta
\in \C$ for which $|\alpha \cap \beta| > i(\alpha,\beta)$.
Among all such $\beta$ take one that, together with $\alpha$, determines
an innermost bigon; then any other
curves from $\C$ meeting that bigon  meet it in arcs that run straight across.

Apply Lemma \ref{lemBigonCurves}.   Since all curves are tight, we must have
the first conclusion.  But, trading across the bigon reduces intersections
between $\alpha$ and $\beta$ without raising intersections of any other pair---a
contradiction.
\end{proof}
%
%
%\begin{lemma} \label{lemTightSnug} Let $\alpha$ and $\beta$ be tight normal
%curves.  Then $\alpha$ and $\beta$ are, after a normal isotopy, snug.
%\end{lemma}
%
%\begin{proof}
%
%
%Normally isotope $\alpha$ and $\beta$ to minimize $|\alpha \cap \beta|$.   By
%way of contradiction, suppose that they are not snug, that $|\alpha \cap
%\beta| > i(\alpha,\beta)$.  Then there is an innermost bigon or half-bigon $B$
%and  Lemma \ref{lemBigonCurves} says either: 1) there is a normal isotopy
%reducing $|\alpha \cap \beta|$, or 2) $\alpha$ is not tight, or 3) $\beta$ is
%not tight.  All are contradictions.  \end{proof}
%

\begin{lemma}
\label{lemCurveSummandsLeastComplexity}
Suppose that a tight essential normal curve is a normal sum $\alpha + \beta$.  Then $\alpha$ and $\beta$ are tight, essential, and after a normal isotopy, snug.
\end{lemma}

\begin{proof}

Normally isotope $\alpha$ and/or $\beta$ to minimize $|\alpha \cap \beta|$.  This does not change their sum.

First we show that the pair is snug:  If not, then some pair of sub-arcs of $\alpha$ and $\beta$ bound a bigon or half-bigon $B$.   Apply Lemma \ref{lemBigonCurves}.   The first conclusion does not hold, so without loss of generality assume that $\compl(\alpha')<\compl(\alpha)$.   Isotope $\alpha'$ back slightly so that $\alpha'$ and $\beta$ still overlap and form a very thin bigon.   Since $\alpha \cup \beta$ and $\alpha' \cup \beta$ are isotopic as graphs, $\alpha + \beta$ is isotopic to some Haken sum
 $\alpha' \hs \beta$.   But by Lemma \ref{lemHakenSumComplexity}, $\compl(\alpha' \hs \beta)\leq \compl(\alpha')+\compl(\beta)<\compl(\alpha)+\compl(\beta)=\compl(\alpha+\beta)$, a contradiction.

It follows that $\alpha$ and $\beta$ are both essential.  If either possesses a component that bounds a disk, then the fact that $\alpha$ and $\beta$ are snug implies that this component misses the other curve,  survives normal addition, and $\alpha+\beta$ contains an inessential component, a contradiction.

It remains to show that  each summand is tight. Without loss of generality, suppose  that $\alpha$ is not tight, that there is a tight curve $\alpha_t$ with lower complexity, $\compl(\alpha_t)<\compl(\alpha)$, that is isotopic to $\alpha$ but not normally so.    Isotope $\alpha_t$ to intersect $\alpha \cup \beta$ minimally.

Then any innermost (half-) bigon in the complement of $\alpha_t \cup \alpha \cup \beta $ is bounded by $\alpha$ and $\alpha_t$, since  it cannot be bounded
by $\alpha$ and $\beta$, which are snug.  And because any patch of $\beta$ is a sub-arc of $\alpha+\beta$, any innermost (half-) bigon bounded by $\beta$ and $\alpha_t$ is also a (half-) bigon bounded by the tight curves $\alpha+\beta$ and $\alpha_t$ which, using Lemma \ref{lemBigonCurves} again, can be eliminated by a normal isotopy of $\alpha_t$.  This contradicts  the minimality of the intersection between $\alpha_t$ and $\alpha \cup \beta$.

Then, sub-curves of $\alpha$ and $\alpha_t$ co-bound a product region $R$
as in Figure~\ref{f:product_regions}.
\begin{figure}[tb]
\begin{center}
  \includegraphics[width=6in]{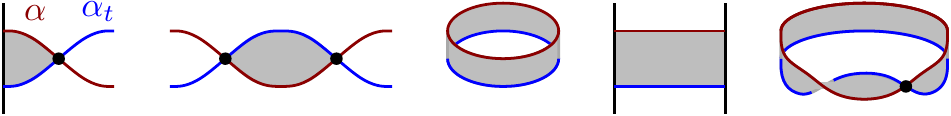}
\end{center}
\caption{Possible product regions derived from isotopy of $\alpha$ and
$\alpha_t$.}
\label{f:product_regions}
\end{figure}

%\martin{By the way,
%it is intuitively clear that there are only these options for the product
%regions. But I would have some troubles with a rigorous proof.}
  If they are not snug, $R$ is a bigon or half-bigon. If they are snug,
%  they are
%  disjoint
%\martin{\confirm{Eric originally claimed that if the curves are snug, then they
%are disjoint. I think that this was a typo because of one-sided
%loop case. I hope I do not overlook anything.}}
$R$ is a rectangle when $\alpha$ is an arc, an annulus  when
  $\alpha$ is a two-sided loop,  and a bigon with corners identified when
  $\alpha$ is a one-sided loop.

In all of these cases,
as observed above, no arc of $\beta$ forms a (half-) bigon inside $R$, and must therefore run across $R$ and have an endpoint in both $\alpha$ and $\alpha_t$.

 In the non-snug case, let $\alpha'$ be the curve of less complexity obtained by routing $\alpha$ along $\alpha_t$ when it meets the bigon or half-bigon.  In the snug case, let $\alpha' = \alpha_t$.   In either case, $\compl(\alpha')<\compl(\alpha)$.    Moreover, the complex $\alpha' \cup \beta$ is isotopic to $\alpha \cup \beta$ and because they are isotopic, there are exchanges, not necessarily regular, so that the Haken sum $\alpha' \hs \beta$ is a curve isotopic to $\alpha + \beta$.   But by Lemma \ref{lemHakenSumComplexity}, $\compl(\alpha' \hs \beta) \leq \compl(\alpha')+\compl(\beta) < \compl(\alpha)+\compl(\beta) = \compl(\alpha+\beta)$.  This contradicts the fact that $\alpha + \beta$ is tight.
\end{proof}
%
%
%DO WE NEED THE FOLLOWING.  IT IS BASICALLY TIGHT=>SNUG LEMMA
%
%
%
%\begin{lemma}
%Let $\C$ be a normal curve system curve consisting of tight normal essential curves.  Then, after a normal isotopy of its elements, $\C$ is snug.
%\end{lemma}
%
%
%\begin{proof}
%Assume that we have normally isotoped the system to minimize the total number of intersection points between pairs of curves from $\C$.  If some pair is not snug, then they co-bound a bigon or half-bigon.   Then by Lemma \ref{lemBigonCurves}, the exchange indicated in Figure \ref{figBigons} eliminates 1 or 2 intersections, either by a normal isotopy, a contradiction, or by reducing the complexity of one of the bounding curves, also a contradiction.
%\end{proof}
%
%This one is better:

\heading{Rails and fences. } Now we again consider a triangulation
with a marking $M$, and auxiliary curves in it that, unlike $M$-normal
curves, go through the points of~$M$.

A \emph{rail} is a normal arc with its endpoints in $M$, and a \emph{fence} is a normal curve that is the union of rails.
%
%An \emph{$M$-arc} is a normal arc with its endpoints in $M$.   An \emph{$M$-curve} is a normal curve that is the union of $M$-arcs.

We note that  if  a face contains an $M$-normal arc $\alpha$ and a
rail $\mu$ that are locally snug,  then $|\alpha \cap \mu|$ is either $0$ or $1$, depending only on the endpoints of $\mu$ and the $M$-normal type of $\alpha$.

%\begin{proof}
%Any $M$-normal arc $\alpha$ splits the face, hence those  points of $M$ on the boundary of the face, in two.  This partition is uniquely determined by the $M$-normal arc type of $\alpha$.   Any arc of the same type of $\alpha$ meets $\mu$ if and only if the endpoints of $\mu$ are in opposite sides of the partition.    If they do intersect, they do so in a single point.
%\end{proof}

The following lemma can also be considered obvious:

\begin{lemma}
\label{lemMarkedCurveSum}
Intersection number with fences is additive with respect to normal addition of $M$-normal curves:  If $\mu$ is an  fence and $\alpha$ and $\beta$ are $M$-compatible, $M$-normal curves, then $|(\alpha+\beta) \cap \mu| =  |\alpha \cap \mu| + |\beta \cap \mu|.$
\end{lemma}

%\begin{proof}

%This should be obvious though there may be something else that needs to be said here.  Leaving it as a placeholder for the moment.

%%Let $\vec m$ be the vector whose $i$th entry is $|n_i \cap \beta|$, where $n_i$ is the $i$th $M$-normal arc type.   Each entry is well-defined by the previous lemma.   The intersection of an $M$-normal curve with $\mu$ is calculated by taking the dot product of the curve's vector with $\vec m$.  Then $|(\alpha+\beta) \cap \mu| =  \vv(\alpha+\beta)\cdot\vec m = \vv(\alpha)\cdot\vec m + \vv(\beta)\cdot\vec m = |\alpha \cap \mu| + |\beta \cap \mu|.$
%\end{proof}

\begin{proposition}
\label{propMSummands}
Let $\mu$ be a fence that is a tight essential curve (w.r.t.\ the unmarked
triangulation). Suppose that a sum  $\alpha+\beta$  of $M$-normal curves is tight, essential and snug with $\mu$.  Then
\begin{enumerate}

\item[\rm(1)] $\alpha$ and $\beta$ are both snug with respect to $\mu$;

 \item[\rm(2)] $i(\alpha+\beta,\mu) = i(\alpha,\mu) + i(\beta,\mu)$ where $i(.,.)$
is the geometric intersection number;

 \item[\rm(3)]
 if $\beta$ is two-sided, connected and normally isotopic to $\mu$  then, after a normal isotopy, every point of $\alpha \cap \beta$ has the same normal
sign.
\end{enumerate}
\end{proposition}

%\martin{\doNow{Uli and me, we would vote for calling $\beta$ rather two-sided than
%orientable. Every loop is orientable as a $1$-manifold.}}

%In the last case, when $\beta$ is closed  we say that $\alpha + \beta$ is a \emph{fractional Dehn twist} of $\alpha$ in $\beta$.

\begin{proof}
Among counterexamples to  conclusion (1) of the proposition, choose one that minimizes $|\alpha \cap \beta|$.    Then $\alpha,\beta$ and $\mu$ are pairwise locally snug, and we will show that they are in fact pairwise snug.   Suppose not and let $B$ be an innermost (half-) bigon bounded by some pair of the curves.

If $B$ is bounded by $\mu$ and either of the other curves, say $\alpha$, then every sub-arc of $\beta$ in $B$ crosses $B$ and meets both $\alpha$ and $\mu$.    Let $\alpha'$ be the result of rerouting $\alpha$ around $B$ as in Lemma \ref{lemBigonCurves}.  Then $\alpha + \beta$ is isotopic to some Haken sum $\alpha' \hs \beta$ that has fewer intersections with $\mu$.  This contradicts our assumption that $\alpha+\beta$ and $\mu$ are snug.

If $B$ is bounded by $\alpha$ and $\beta$, then every sub-arc of $\mu$ in $B$ crosses $B$ and meets both $\alpha$ and $\beta$.     Let $\alpha'$ and $\beta'$ be the curves given by Lemma \ref{lemBigonCurves}.  Because $\alpha+\beta$ is tight, $\alpha$ and $\beta$ are tight and normally, but not necessarily $M$-normally, isotopic  to  $\alpha'$ and $\beta'$, respectively.   Because the normal sum $\alpha+\beta$ was defined, $\alpha$ and $\beta$ are locally snug.   The isotopy doesn't create intersections, and so $\alpha'$ and $\beta'$ are also locally snug.   Then $\alpha + \beta = \alpha' \hs \beta'$ for some generalized Haken sum of $\alpha'$ and $\beta'$.    By Lemma \ref{lemHakenSum} that sum is a normal sum, $\alpha + \beta = \alpha' + \beta'$.   Note that $\alpha'$ and $\beta'$ are snug with $\mu$ if and only if $\alpha$ and $\beta$ are, as the move did not change the number of times they meet $\mu$.     Since $|\alpha \cap \beta| > |\alpha' \cap \beta'|$ we obtain a contradiction and establish conclusion (1).

Since $\alpha$, $\beta$ and their sum are all snug with respect to $\mu$ and intersections with respect to $\mu$ are additive, we have additivity of geometric intersection number, conclusion~(2).

%\eric{ this part is really crucial but not sure that this says it well.  }

We now prove the final statement of the proposition. Assume that $\beta$ is
normally isotopic to $\mu$.  Then $i(\beta,\mu)=0$ since $\beta$ is two-sided.  By (2) and the fact that
$\alpha+\beta$ and $\alpha$ are both snug with $\mu$ we have: $|(\alpha+\beta)
\cap \mu| = i(\alpha+\beta,\mu) = i(\alpha,\mu) = |\alpha \cap \mu|$.
%\martin{Added that $i(\beta,\mu)=0$ follows from orientability of $\beta$.
%Should be changed to `two-sided' if Eric agrees with the change
%'orientable' into `two-sided'.}

\begin{figure}[ht]
\begin{center}
\includegraphics[width=4in]{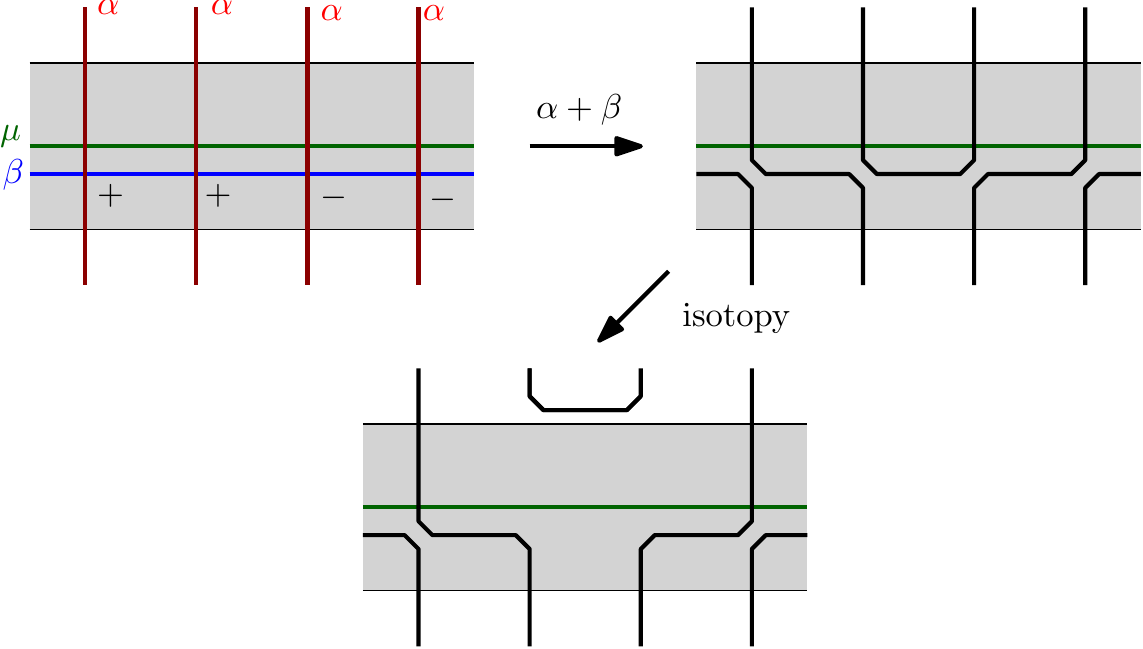}
\end{center}
\caption{Mixed normal signs imply that $i(\alpha+\beta,\beta)<i(\alpha,\beta)$.}
\label{figNotATwist}
\end{figure}

And since $\beta$ is normally isotopic to $\mu$, there is a normal, not necessarily $M$-normal,  isotopy taking $\alpha+\beta$ to $\alpha+\mu$.  Now $\alpha$ cuts across a thin regular neighborhood of $\mu$ in a collection of arcs that span
(cut across)
 the neighborhood.  Together they cut this neighborhood into rectangles; see Figure \ref{figNotATwist}.

Each regular exchange connects a pair of rectangles at a corner of each.   In fact, every rectangle that meets $\alpha \cap \mu$ twice must be attached to another rectangle at one of its corners.  Otherwise, an arc  of $\alpha + \mu$ bounds the unattached rectangle, showing that the arc it is trivial in the  neighborhood of $\mu$ and can be isotoped out of it.   This would imply $i(\alpha+\beta,\mu) < |\alpha\cap\mu|$, contradicting the equality shown earlier.   The fact that each rectangle is attached at exactly one corner implies that as we follow $\mu$ every intersection with $\alpha$ must have the same normal sign. Since $\beta$ is normally isotopic to $\mu$ we have our desired conclusion~(3).
\end{proof}

\section{Normal summands of incompressible annuli}
\label{s:def-0-eff}
\label{sec:summands_of_annuli}

%
%\eric{should really cite Jaco Oertel probably Schubert and probably Haken}

We would like to apply two well known results from normal surface theory:   (1)
an essential surface is isotopic to a normal surface, and (2) every summand of
a least weight essential normal surface is also least weight and essential
(Theorem 6.5 of Jaco and Tollefson
\cite{Jaco:Algorithms-for-the-complete-decomposition-of-a-closed-3-manifold-1995}).
But, as will be seen shortly, our notion of surface complexity prioritizes the
reduction of boundary complexity over the reduction of total surface weight.
Thus the results (1) and (2) cannot be applied as stated.

Proposition~\ref{propNormalize} recovers the first result using our notion of
complexity.   Proposition~\ref{propSummandsEssential} gives a weaker version of
the second for incompressible annuli.   While we expect the full version to
hold with our notion of complexity, we prove a restricted version both to
simplify the proof and to incorporate boundary parallel annuli which  are
non-essential.   Our proof follows the strategy of
\cite{Jaco:An-algorithm-to-decide-if-a-3-manifold-is-a-Haken-manifold-1984} and
\cite{Jaco:Algorithms-for-the-complete-decomposition-of-a-closed-3-manifold-1995}.

The \emph{complexity} $\compl(F)$ of a properly embedded surface $F$ is the
triple
\[
\compl(F) = (\compl(\partial F),|F \cap \T^1|,|F \cap \T^2|)=
((\ell(\partial F),\vv(\partial F)),|F \cap \T^1|,|F \cap \T^2|).
\]
We compare complexities lexicographically.
Thus, the complexity of $F$ is measured first by the complexity
of its boundary, then by the \emph{weight} of $F$, $\wt(F) =|F \cap \T^1|$,
and then by the number of components of the intersections
with the $2$-skeleton of~$\T$.

A normal surface is \emph{least complexity} if it minimizes complexity among
normal surfaces to which it is isotopic (but not necessarily normally
isotopic).

A surface is \emph{tight} if it minimizes complexity, ranging over \emph{all}
those surfaces to which it is isotopic.

A tight normal surface is clearly
least complexity, and as  a consequence of Proposition~\ref{propNormalize},
a normal essential surface of least complexity is tight.    But, this does
not hold in general for surfaces that are not essential:  for example,
a normal boundary parallel annulus may be least complexity
but after tightening no longer normal.

%\eriC{everyone ok with overloading tight?}
%

We first recover normalization of an essential surface.  We will apply this
with surfaces whose boundaries are tight, hence least length.

\begin{proposition}
  \label{propNormalize}
  Suppose that $X$ is a triangulated, irreducible manifold with incompressible
  boundary.  If $F \subset X$ is a tight, properly embedded, essential surface,
  then $F$ is normal.
\end{proposition}

\begin{proof}
To prove $F$ is normal we must show that it meets each tetrahedron $\Delta$ in
a collection of disks whose boundaries are normal curves of length 3 or 4.  We
adopt the view taken in \cite{bachmanDerbyTalbotSedgwick}, showing $F$ meets
each tetrahedron in pieces that are incompressible and edge incompressible.

If any component of $F \cap \Delta$ is compressible in $\Delta$, then, by an
innermost disk argument,  we obtain a compressing disk avoiding all
other components of $F \cap \Delta$, and hence
$F \cap \Delta$ is compressible inside $\Delta$.

Because $F$ is essential, the boundary of any compressing disk $D$ for $F \cap
\Delta$ is trivial in $F$.  Because $X$ is irreducible, compressing along $D$
yields a surface $F'$ that is isotopic to $F$, but for which either $|F \cap
\T^1|$ or $|F \cap \T^2|$ has been reduced, a contradiction.  It follows that
$F \cap \Delta$ is the union of disks.

An \emph{edge compressing disk} for a surface in $\Delta$ is an embedded disk
$E$  whose boundary $\partial E = e \cup f$, consists of two arcs, $e \subset
\T^1$ and $f = E \cap F = \partial E \cap F$;
see~\cite{bachmanDerbyTalbotSedgwick}.

If some component of $F \cap \Delta$ has an edge compressing disk then, by an
innermost disk argument, there is an edge compressing disk $E$ for $F \cap
\Delta$.   If $e \subset \partial X$ then, because $F$ is not boundary
compressible, $f$ is trivial in $F$.   But compressing along $E$ yields an
isotopic  surface $F_0$ ($X$ is irreducible and has incompressible boundary) whose
boundary length is reduced by at least two, contradicting the fact that
$\partial F = \partial F_0$ is least length. And if $e$ lies in an interior
edge, then $E$ can be used to guide an isotopy reducing $|F \cap \T^1|$, also a
contradiction.

Then $F$ meets each face in normal arcs.   For otherwise, there is an arc whose
ends both lie in the same edge, and an outermost such arc bounds an edge compressing
disk.   Then $F$ meets the boundary of each tetrahedron in normal curves.   And
it is well known, see Thompson
\cite{Thompson:ThinPositionrecognitionProblemS3-1994}, that if any such curve has
length greater than 4 we see an edge compressing disk for $F$ in the boundary
of the tetrahedron.
\end{proof}

\heading{0-efficient triangulations.} First we recall the
definition of $0$-efficient triangulations
from~\cite{Jaco:0-efficient-triangulations-of-3-manifolds-2003}. A
triangulation of a manifold $X$ with nonempty boundary
is
\emph{$0$-efficient} if every normal disk is vertex-linking.
%(a normal disk is
%vertex-linking at vertex $v$ if it consists exactly of normal triangles which
%separate $v$ from other vertices in some tetrahedron, each such triangle
%appears only once).
(A normal disk is vertex-linking at vertex $v$ if it consists of precisely one
normal triangle from each tetrahedral corner meeting $v$.)

Moreover, if no boundary component of
$X$ is an $S^2$, then $X$ does not contain any normal $2$-spheres
\cite[Prop.~5.15]{Jaco:0-efficient-triangulations-of-3-manifolds-2003}.
In our setting, we use $0$-efficient triangulations only
in the situations without $S^2$ boundary components
(since in the algorithm, we fill each such component with a ball).
%Moreover, if $X$ is irreducible, then an $S^2$ boundary component
%implies that $X$ is a ball, and the ball
%does not contain incompressible annuli.
Note also that in the proposition below we can assume that $X$ does not contain
$S^2$ boundary components even if do not explicitly claim that $X$ is obtained
in an intermediate stage of the algorithm. Indeed, we assume that $X$ is
irreducible. Then an $S^2$ boundary component implies that $X$ is a ball;
however, the proposition also assumes that $X$ contains an essential annulus or
M\"{o}bius band.

We now establish the second result, that some summand of a non-fundamental
incompressible annulus is an essential annulus.   This applies to boundary
compressible as well as essential annuli.

\begin{proposition}
  \label{propSummandsEssential} Let $X$ be a triangulated, orientable, irreducible manifold
  with incompressible boundary and a $0$-efficient triangulation.   Let  $A$ be an
  incompressible annulus or M\"obius band that has tight boundary and is least
  complexity and normal. Suppose that $A$ can be written as a non-trivial sum
  $A = B + C$ where $B$ is connected and $\partial B \neq \emptyset$.  Then $B$
  is an essential annulus or M\"obius band with tight boundary.
\end{proposition}

\subsection{Proof of Proposition \ref{propSummandsEssential}}

\heading{Sketch of the proof.}
Our proof is loosely modeled on Jaco and Tollefson's proof of
\cite[Th.~6.5]{Jaco:Algorithms-for-the-complete-decomposition-of-a-closed-3-manifold-1995}. Apart from using slightly different notion of complexity, we also have to add
additional ingredients when $A$ is a boundary parallel annulus.

As we will see, the core of the proof is to show that $B$ is essential. For
contradiction we assume that $B$ is not essential.
The first important step is to find out what are the possible patches when $A$
is decomposed by trace curves from the normal sum $A = B + C$; see
Figure~\ref{f:annulus_into_patches} left.
If $A$ is essential (annulus or M\"{o}bius
band), then disk patches as well as half-disk patches can be ruled out
following~\cite{Jaco:Algorithms-for-the-complete-decomposition-of-a-closed-3-manifold-1995}
(disk patches avoid $\partial X$ whereas half-disk patches contain a single arc
on $\partial X$); see Lemmas~\ref{lemNoDiskPatches}
and~\ref{lemNoHalfDiskPatches}.
After ruling out such patches we can deduce that every intersection
curve is essential in $B$, that is a spanning arc or a core curve. This already
mean that $C$ intersects $B$ in a very specific way and both cases
can be ruled out
along~\cite{Jaco:Algorithms-for-the-complete-decomposition-of-a-closed-3-manifold-1995}; see~Lemma~\ref{l:AessentialBessential}.

\labfig{annulus_into_patches}{
    The annulus $A$ separated by trace curves into patches
      (left); $E$ is an exchange annulus. The right part shows a
      cross-section of $D \cup E \cup D'$ assuming that $D$ and $D'$ are not
      adjacent across $E$. Then $D \cup E \cup D'$ separates,
      after a slight isotopy,
      two components of $A$.}

If $A$ is not essential, then $A$ is a boundary parallel annulus by
Proposition~\ref{propSurfaceFacts}. In this case we do not know how to rule out
disk patches but we still can rule out half-disk patches
(Lemma~\ref{lemNoHalfDiskPatches}); here we use that
simplification of the boundary has higher priority than simplification of
the interior
in our notion of complexity. Since $A$ is boundary parallel, there is an
annulus $A_{\partial X}$ to which $A$ is parallel and together they bound a
solid torus $T$ in $X$. Because there are no half-disk patches, we can show that one of the exchange rectangles for the
sum $A = B + C$ is inside this torus and it meets
 $A$ and $A_{\partial X}$ only in
essential arcs. However, with such a rectangle $A$ cannot be boundary
parallel; see
Lemma~\ref{l:A_BP_B_essential} for details. This finishes the sketch of the
proof and now we provide the details.

%\eric{quite happy with the sketch.  I might change "not fit" to something like "However, such a rectangle would mean that $A$ is not boundary parallel."   We can find a rectangle in a solid torus, but not in this one, because $A$ is longitudinal (don't think we defined this).  If you don't like my suggestion, you can just leave what you have}

Because $A$ is incompressible, $\partial A$ is essential by
Proposition~\ref{propSurfaceFacts}.
Without loss of generality, we will assume that the sum $A=B+C$
lexicographically minimizes $(|\partial B \cap \partial C|, |B \cap C|)$, the
number of boundary intersections and the total number of intersection curves,
over pairs $(B',C')$ where $B'$ and $C'$ are locally snug surfaces isotopic to
$B$ and $C$, respectively.   Since $\partial A (= \partial B + \partial C)$ is
assumed tight, we have, by Lemma \ref{lemCurveSummandsLeastComplexity}, that
$\partial B$ and $\partial C$ are tight, and because $|\partial B \cap \partial
C|$ is minimized, snug.

%\eric{\doNow{have changed the minimization to prioritize minimizing boundary
%intersections, will boundary snug mess up locally snug?}}

%\martin{It would be perhaps, good to remark that this WLOG assumption does not
%    break `tight' by Lemma~\ref{lemCurveSummandsLeastComplexity}. (Possibly
%    also
%      mention that/why essential annulus has essential boundary.
%    }
%
%\eric{think that I have addressed this point, if not let me know}

\begin{lemma}
Either the conclusion of Proposition~\ref{propSummandsEssential} holds,
or $B$ is a boundary parallel annulus and every
component of $C$ is an incompressible annulus, M\"obius band, torus or Klein
bottle.
\end{lemma}

\begin{proof}
No component of $C$ has Euler characteristic $\chi>0$:  Because $X$ is 0-efficient, no
normal surface is a sphere, nor a projective plane, for then its normal double
would be a normal sphere.  And, also by 0-efficiency, any disk has boundary a
trivial vertex linking curve that survives normal addition, and is present in
$\partial A$---a contradiction.

Then every component has $\chi=0$ and it is an annulus, M\"obius band, torus,
or Klein bottle.   No component is a compressible annulus since these have a
trivial
boundary component (Proposition~\ref{propSurfaceFacts}) and this contradicts
the fact that both summands have essential boundary.

Since $B$ is connected
and $\partial B \neq \emptyset$, $B$ is either an annulus or M\"obius band.
Also by Proposition~\ref{propSurfaceFacts}, a M\"obius band is essential and
satisfies the conclusion of Proposition~\ref{propSummandsEssential}.
So does an annulus, unless it is boundary compressible, and
hence boundary parallel, by Proposition~\ref{propSurfaceFacts}.
\end{proof}

%\eric{given that last lemma, it just seems that there ought to be a much
%shorter proof than the rest of this}

We proceed with the proof of Proposition~\ref{propSummandsEssential}
under the assumption that $B$ is a boundary parallel annulus.

When $A$ is formed as the normal sum $B+C$, it is partitioned
into patches coming from $B$ and $C$, as was discussed in
Section~\ref{secGeomSum}, and we have exchange surfaces attached
to the curves separating the patches;
see Figure~\ref{f:annulus_into_patches} left.

It follows that no exchange surface is a M\"obius band.  As noted in
Section~\ref{secGeomSum}, this occurs only when an intersection loop is
one-sided in both summands.

Define a \emph{half disk} to be a disk that is halfway properly embedded in $X$,
that is, an embedded disk whose boundary meets $\partial X$ in a single arc.
Note that a boundary compressing disk for a surface is a half disk whose
boundary meets the surface in the complementary arc, but the reverse does not
hold in general, for the arc may not be essential in the surface.

An exchange rectangle or annulus $E$ meets four patches of $A$.  A pair $D,D'$
of these patches are said to be \emph{adjacent across $E$}
if they meet opposite
boundary curves $\sigma$ and $\tau$ of $E$, but from the same side
(we again refer to Figure~\ref{f:annulus_into_patches} left).

%In particular this implies that no
%intersection curve bounds a disk or half disk patch in both $B$ and $C$.   For
%a regular exchange along that intersection curve would produce a pair of
%isotopic surfaces with fewer intersections.   If $B$ is a M\"obius band, then
%it is essential, and we have the desired conclusion.  Because  $B$ is
%orientable there are no exchange M\"obius bands.

\begin{lemma}
  \label{lemBoundDisks}
  $\sigma$ bounds (half) disk in $A$ if and only if $\tau$ both bounds a
  (half) disk in $A$. %See Figure~\ref{f:annulus_into_patches}, left.
\end{lemma}

\begin{proof}
We prove that if $\sigma$ bounds a (half) disk in $A$, then $\tau$ bounds a
(half) disk in $A$. The reverse implication is proved by interchanging
$B$ and $C$ and
remarking that in this proof we do not use the extra assumptions on~$B$.

The surface $A$ is either essential or a boundary parallel annulus,
and it is incompressible by the assumptions.

Suppose that $\sigma$ bounds a disk $D$ in $A$.  Then $E \cup D$ is a disk
which, after a slight isotopy, meets $A$ only in $\tau$.  Since $A$ is
incompressible, $\tau$ bounds a disk in $A$ as claimed.

The same argument works when $A$ is essential and, say,  $\sigma$ bounds a half
disk $H$.  Then $E \cup H$ is a half disk which, after a slight isotopy, meets
$A$ only in $\tau$.   Since $A$ is not boundary compressible, $\tau$ bounds a half disk in~$A$.

We conclude by showing that $\tau$ bounds a half disk
when $\sigma$ bounds a half disk and $A$ is a
boundary parallel annulus (assumed to have tight boundary).  To obtain a
contradiction, suppose that $\sigma$ bounds a half disk $H$ but $\tau$ is an
essential arc in $A$.  Then $E \cup H$ is, after a slight isotopy, a boundary
compressing disk meeting $A$ in the arc $\tau$.    Since $A$ is parallel to an
annulus $A_{\partial X} \subset \partial X$, their union bounds a solid torus
in $X$.   The rectangle $E$ is a disk properly embedded in this solid torus.

Indeed, if $E$ is outside the solid torus, consider a boundary compressing disk
$D_A$ for $A$ meeting $A$ in $\tau$ inside the solid torus.
Then the disk $D_A \cup E
\cup H$ meets the core curve of $A_{\partial X}$ exactly once, implying that it
is a non-separating disk and therefore a compressing disk for $\partial X$---a
contradiction.

As soon as we know that $E$ is inside the solid torus, we have that the boundary
of $E$ meets $A_{\partial X}$ in a pair of
exchange arcs that each span $A_{\partial X}$ by Lemma
\ref{lemExchangeArcNotParallel}, and meets the annulus $A$ in one curve
$\sigma$ that is trivial in $A$ and the other $\tau$ that is a spanning arc for
$A$. Therefore, when restricting to $A_{\partial X}$, we get that two corners of
$E$ are in one component of $\partial A$ and the other two in the second one.
When restricting to $A$, we get that three corners are in the same component and
the other corner in the second---a contradiction.
\end{proof}

%\martin{Coming back to 2D version. For a short time until next submit or Eric's
%approval/disapproval, I leave also previous figure for comparison. At the
%moment it is also open to leave out the 3D part or even leave it out and draw
%the cross-section.}

%\begin{figure}[ht]
%  \begin{center}
%      \includegraphics[width=4in]{notAdjacent}
%      \includegraphics{notAdjacent_cross_section}
%    \end{center}
%    \caption{A cross section in addition for comparison. }
%\end{figure}
%\martin{A cross section on next page for comparison, at the moment only on the left picture.}

\begin{lemma}
    \label{lemTwoDisks}
    Suppose $A$ is essential.  If $\sigma$ and $\tau$ bound disks in $A$, then
    they are adjacent across $E$.
\end{lemma}

\begin{proof}

Suppose that $\sigma$ and $\tau$ bound non-adjacent disks $D$ and $D'$.  They
either are disjoint, or one is a sub-disk of the other, say $D' \subset D$.

When disjoint, the union $D \cup E \cup D'$ is a sphere that, after a slight
isotopy, separates components of $A$
(since spheres separate in
irreducible manifolds)---a contradiction.
See Figure~\ref{f:annulus_into_patches} right.

Suppose then that $D' \subset D$; see Figure \ref{figNotAdjacent}. Let $A'$
be the surface obtained from $A$ by removing $D$ and replacing it with $E \cup
D'$.

The union of the disk $D$ and a slight offset of the disk $E \cup D'$
bounds a ball, across which the disks are isotopic ($X$ is irreducible).  So $A'$,
the result of this disk swap, is isotopic to  $A$.   Also note
 that performing
an irregular exchange (fold) at this intersection loop produces a surface with
two components: one is $A'$, and the other, $A''$,  is a torus obtained by
identifying the ends of the annulus $D \setminus D'$.

Because $E$ has zero
weight, we have $\wt(A) = \wt(A') + \wt(A'')$.   But, because this was not a
regular exchange, $A' \cup A''$ is not normal, and there is an abnormal arc
bounding a half disk in some face by Lemma~\ref{lemHakenSum}.   If this half disk meets $\partial X$, then
$A' \cup A''$ either is boundary compressible or is not least length, both
contradictions.

Thus, the half disk lies in the interior and can be used to
guide an isotopy of $A' \cup A''$ that removes  two intersections with the
1-skeleton.   But this implies that $A' \cup A''$ can be isotoped to have
strictly less weight than $A$.   This is a contradiction since the component
$A'$ has lower complexity, but is isotopic to the tight surface $A$.
\end{proof}

\begin{figure}[t]
  \begin{center}
      \includegraphics{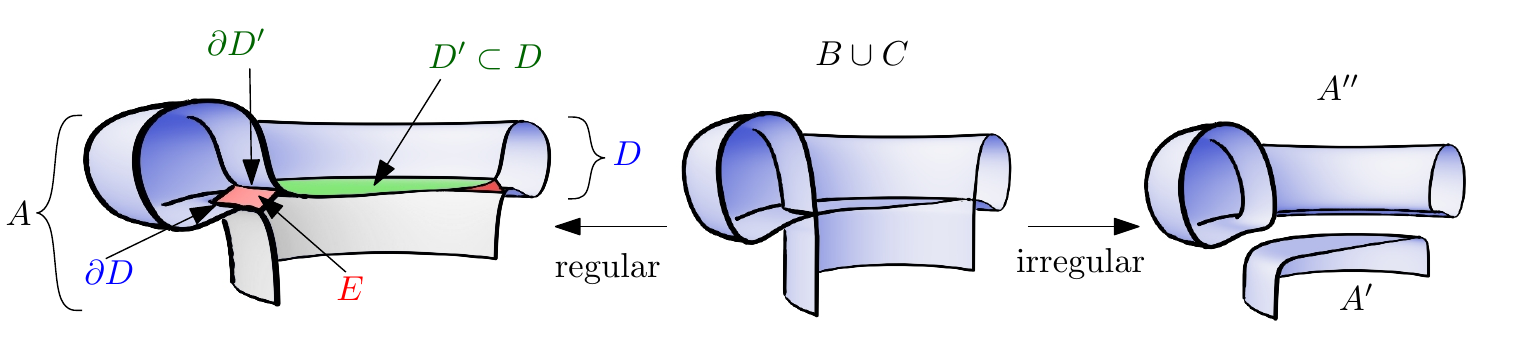}
    \end{center}
    \caption{Non-adjacent disks, $D' \subset D$ (or half disks $H' \subset H$
      by replacing $D$ with $H$ and $D'$ with $H'$). The picture is actually
      drawn for half disks and the disk case is obtained by doubling.
    }
    \label{figNotAdjacent}
\end{figure}

Unfortunately, the above proof contradicts minimal interior weight and does not
apply when $A$ is a (non-essential) boundary parallel annulus, which may not be
normal when tight.   Fortunately, the half disk version contradicts tight
boundary and can be applied when $A$ is essential or a boundary parallel
annulus.

\begin{lemma}
  \label{lemHalfDisksAdjacent}
  If $\sigma$ and $\tau$ both bound half disks in $A$, then they are adjacent.
\end{lemma}

\begin{proof}
Suppose to the contrary that $\sigma$ and $\tau$ bound half disks $H$ and $H'$
that are not adjacent across $E$.   The half disks $H$ and $H'$ are either
disjoint, or, say, $H' \subset H$.

If disjoint, then the union $H \cup E \cup H'$ is a properly embedded disk,
that after a slight isotopy,  separates components of $A$,
which is a contradiction.

Now suppose that $H' \subset H$; see Figure \ref{figNotAdjacent}.
%\jirka{I wasn't able to follow the proof of this lemma and the previous one
%very well, and Fig.~\ref{figNotAdjacent} was probably a part of the problem,
%without further explanation and with all of the mysterious framed
%letters, colors, and line styles I found it incomprehensible.}
%\martin{Used Eric's 3D picture (after a slight modification). Also changed
%slightly Fig.~\ref{figAdjacentSubset}}
Replacing
$H$ with $E \cup H'$ is a disk swap across a ball that produces a surface $A'$
isotopic to $A$. But notice that performing an irregular rather than regular
switch at this intersection curve produces a surface with two components: one
is $A'$, and the other is an annulus $A''$ formed by identifying the ends of the
rectangle $H \setminus H'$.    The irregular switch on the intersection arc
yields irregular switches at the endpoints which are intersections of the
boundary curves.

So while $\ell(\partial A) = \ell(\partial A') +
\ell(\partial A'')$,  the curve $\partial A' \cup \partial A''$ is not normal,
contains an abnormal arc by Lemma~\ref{lemHakenSum},
and so there is an isotopy reducing its length.  Since
$\partial A' \cup \partial A''$ is isotopic to a curve
of length strictly lower than $\partial A$, each of its components has length
strictly lower than $\partial A$, contradicting the minimality of the length of
$\partial A$.
\end{proof}

\begin{figure}[ht]
  \begin{center}
%      \includegraphics[width=4in]{sphere}
%      \newline
%      \includegraphics{sphere_ipe}

      \includegraphics{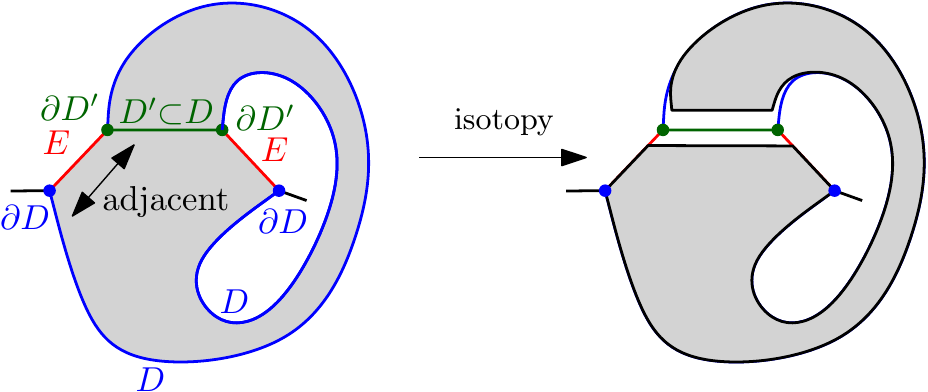}
    \end{center}
    \caption{Adjacent disks; here $D' \subset D$. The union
$D' \cup E \cup D$ bounds a ball
    after slight isotopy.}
    \label{figAdjacentSubset}
\end{figure}

In the next two lemmas we will utilize the following observation:

\begin{obs}\label{o:theobs}
Suppose that
$D$ and $D'$ are disks that are adjacent across the exchange annulus $E$.  If
the disks are disjoint, then $D \cup E \cup D'$ is, possibly after a slight
isotopy, a sphere bounding a ball.
\end{obs}

\begin{proof}
If $D$ and $D'$ are disjoint, we get the observation immediately, since
$X$ is irreducible.  If not, say if $D' \subset D$, then fix $D$ and slightly
isotope the interior of the disk $E \cup D'$ off $D$ to the side of the
exchange annulus; see Figure~\ref{figAdjacentSubset}.   After the isotopy $D
\cup E \cup D'$ is a sphere bounding a ball.
\end{proof}

Similarly, if $H$
and $H'$ are half disks adjacent across the exchange rectangle $E$, then the
union $H \cup E \cup H'$ together with $\partial X$ bound a ball, possibly
after a slight isotopy. Indeed, the union $H \cup E \cup
H'$ is a properly embedded disk, after first perhaps slightly isotoping, say
$E \cup H'$, when $H$ and $H'$ are not disjoint.  Its boundary $\partial H \cap \partial E \cap
\partial H'$, perhaps slightly isotoped,  bounds a disk in $\partial X$ and
together these disks bound a ball in $X$.

%\martin{I added a sentence below trying to say what is our aim.}
We also describe how a surface $A$ obtained as a normal sum $A= B + C$ can be
obtained, under certain additional conditions, as a normal sum $A = B' + C'$
guided by some of the normal exchanges of $A= B + C$.
Let $\E$ be the set of exchange bands for the normal sum $A=B+C$.  Note that
$|\E| = |B \cap C|$.  We say that a subset $\E' \subset \E$ is a {\it
consistent subset} if the induced patches, components of $A \setminus \E'$, can
be bicolored so that two patches have different colors if they either lie on
opposite sides of the same trace curve, or, are adjacent across an exchange
band. (If this happens for two sides of a same patch, then in particular it
cannot be bicolored.)
%and non-adjacent
%patches that meet opposite boundary curves of an exchange band have the same color.
In this case, we
can see that $\E'$ is the set of exchange bands for a normal sum $A=B'+C'$
where $B'$ and $C'$ are each the union of patches of a single color, connected
across $\E'$. See Figure~\ref{f:pes}.  The same analysis holds for subsets of exchange arcs for a
curve sum $\alpha = \beta + \gamma$.

\begin{figure}
\begin{center}
%  \includegraphics{proper_exchange_system}
%  \caption{Summing $B$ and $C$ (left picture) yields $A$ (left middle picture)
%  with the corresponding set of exchange bands/arcs $\E = \{E_1, E_2, E_3,
%E_4\}$. Some subset $\E' \subset \E$ may not be consistent (right middle
%picture); however, if it is consistent (right picture), then summing $B'$ and
%$C'$ yields $A$ again.}
  \includegraphics{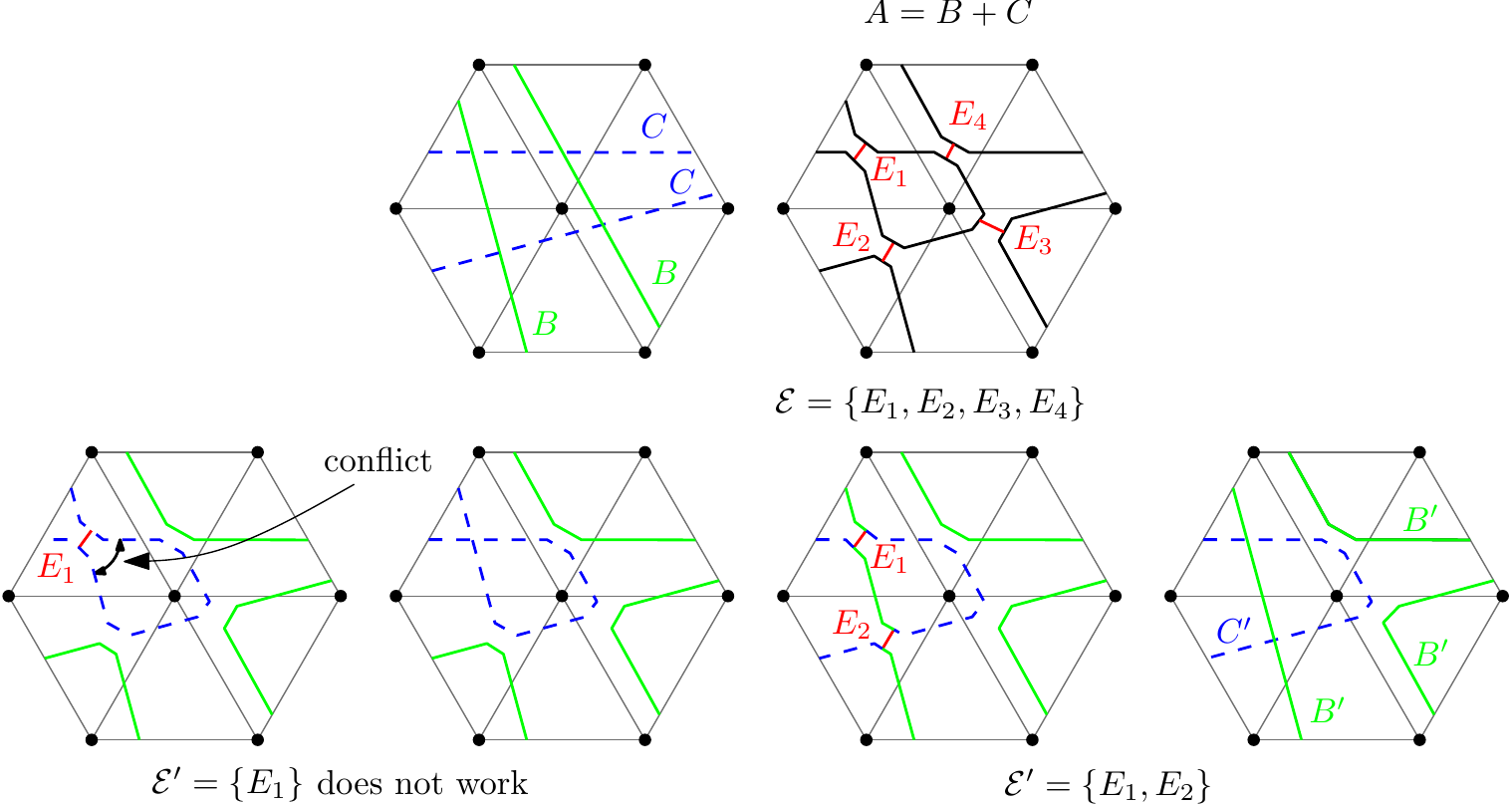}
  \caption{Summing $B$ and $C$ (upper left picture) yields $A$ (upper right picture)
  with the corresponding set of exchange bands/arcs $\E = \{E_1, E_2, E_3,
E_4\}$. Some subset $\E' \subset \E$ may not be consistent (two lower left
pictures); however, if it is consistent (two lower right pictures), then summing $B'$ and
$C'$ yields $A$ again.}
  \label{f:pes}
\end{center}
\end{figure}

\begin{lemma}
  \label{lemNoDiskPatches}
  If $A$ is essential, then there are no disk patches.  Every intersection loop is
  essential in $A$, $B$, and $C$.
\end{lemma}

This lemma does not apply to the boundary parallel case, because we cannot
assume that disks are adjacent.

\begin{figure}[ht]
  \begin{center}
  \includegraphics{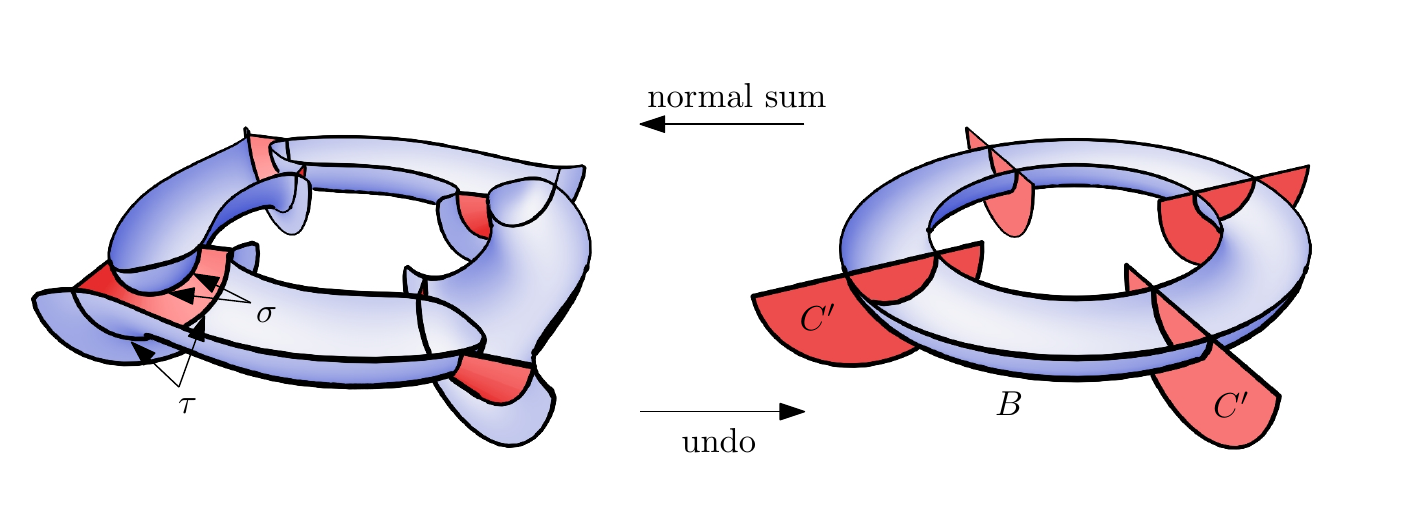}
    \end{center}
    \caption{A cycle of adjacent half disks. Double the figure for a cycle of disks.}
    \label{figCycleAdjacent}
\end{figure}

\begin{proof}
This argument appears in
\cite{Jaco:An-algorithm-to-decide-if-a-3-manifold-is-a-Haken-manifold-1984} and
\cite{Jaco:Algorithms-for-the-complete-decomposition-of-a-closed-3-manifold-1995}.
If there is a disk  patch bounded by a trace curve, then by Lemmas
\ref{lemBoundDisks} and \ref{lemTwoDisks}, it is adjacent across an exchange
annulus $E$ to another disk.   The adjacent disk is not a single patch,since
if it were,
 $\E \setminus E$ would be a consistent subset, and we could expresses $A$ as a sum
$A=B'+C'$, where $B'$ and $C'$ are isotopic to $B$ and $C$, respectively, but
$|B'\cap C'|<|B\cap C|$, a contradiction.  Thus, the adjacent disk contains
trace loops.  Pass to an innermost trace loop bounding a disk patch and repeat.

This process can be continued arbitrarily, and therefore must eventually
repeat.  Thus there is a shortest cycle of adjacent disks, each biting the tail
patch of its predecessor across some exchange annulus; see Figure
\ref{figCycleAdjacent}.  Note that the cycle has length 1 when $D' \subset D$
for some pair of adjacent disks $D$ and $D'$.  By irreducibility, the union of
a pair of adjacent disks and their exchange annulus is a sphere bounding a
ball, and the union of these balls is a solid torus.

% Thus there is a shortest cycle of  $\sigma$'s each  bounding a disk patch.
% Their companion $\tau$'s also form a cycle.  Each $\tau$ is a loop
% co-bounding an annulus in $A$ with the next $\sigma$.  Perform normal
% switches at all intersection curves between $B$ and $C$ except those
% corresponding to the cycle of $\sigma$'s and $\tau$'s.

Let $\E'$ be the subset of $\E$ along which the cycle of disks is adjacent.
Then $\E'$ is seen to be a consistent subset of $\E$ by coloring all annulus
patches on the boundary of the solid torus with one color, and all other patches with the
other.   This expresses $A$ as  the sum of surfaces $A=B'+C'$ where $B'$ is a
normal torus bounding a solid torus.

Adding $B'$ to $C'$ corresponds to a fractional Dehn twist in $B'$.  Since $B'$
bounds a solid torus, there is an isotopy of the solid torus that undoes the
twist and carries $C'$ to $A$.   So $C'$ is isotopic to $A$ but with strictly
lower complexity $\compl(C')=\compl(A)-\compl(B')$---a contradiction.
\end{proof}

We repeat the above argument to work in the context of trace arcs and half
disks.   This argument does apply to boundary parallel annuli, as their half
disks are known to be adjacent.   The proof will reach
a contradiction to the tightness of~$\partial A$.

\begin{lemma}
  \label{lemNoHalfDiskPatches}
  No trace arc bounds a half disk in $A$.
  %Every intersection arc is essential in $A$.
\end{lemma}

\begin{proof}
When $A$ is a boundary parallel annulus, Lemma \ref{lemNoDiskPatches} does not
apply, and trivial trace loops may be present.  So we will refer to outermost
(in $A$) half disks instead of half disk patches.   An outermost half disk in
$A$  does not contain trace arcs, but, is not a patch if it contains trace
loops.  If there is an outermost half disk  $H$ bounded by a trace arc, then by
Lemmas \ref{lemBoundDisks} and \ref{lemHalfDisksAdjacent}, it is adjacent
across an exchange rectangle $E$ to another half disk $H'$.

The adjacent half disk $H'$ is not outermost:  If $H$ and $H'$ are both
outermost, then $H \cup E \cup H'$ is a properly embedded disk  that meets
$\partial X$ in a disk bounding the curve $h \cup e \cup h' \cup e'$, where $h
= H \cap \partial X$, $h' = H' \cap \partial X$ and $e$ and $e'$ are exchange
arcs, the two components of $E \cap \partial X$.  Since $h$ and $h'$ are
sub-arcs of $\partial B$ and $\partial C$, we can see, by undoing the exchanges
$e$ and $e'$, that $\partial B$ and $\partial C$ bound a bigon and are thus not
snug, a contradiction.

Thus, the adjacent half disk is not outermost and contains at least one  trace
arc.  Pass to an innermost trace loop bounding a outermost half disk and
repeat. This process can be continued arbitrarily, and therefore must
eventually repeat.  Thus, there is a shortest cycle of adjacent half disks, each
biting the tail patch of its predecessor across some exchange rectangle; see
Figure \ref{figCycleAdjacent}.  Note that the cycle has length 1 when $H'
\subset H$ for some pair of adjacent half disks $H$ and $H'$.  By irreducibility and incompressibility of $\partial X$, the union of a pair of adjacent half
disks and their exchange rectangle is a disk that co-bounds a ball with a disk in $\partial X$, and the union
of all these balls is a solid torus meeting $\partial X$ in an annulus.

Let $\E'$ be the subset of $\E$ along which the cycle of half disks is
adjacent.  Then $\E'$ is seen to be a consistent subset of $\E$ by coloring all
rectangle patches on the boundary of the solid torus with one color, and all other
patches with the other.   This expresses $A$ as  the sum of surfaces $A=B'+C'$ where
$B'$ is a normal boundary parallel annulus.

Adding $B'$ to $C'$ corresponds to a fractional Dehn twist in $B'$.  Since $B'$
is boundary parallel, there is an isotopy of the solid torus that undoes the
twist and carries $C'$ to $A$.   So $C'$ is isotopic to $A$ but with shorter
length $\ell(C')=\ell(A)-\ell(B')$, a contradiction.
\end{proof}

We know that $A$ is incompressible, and thus by Proposition~\ref{propSurfaceFacts}
it is either an essential annulus or M\"obius band, or a boundary parallel annulus. We deduce that in
both cases it means that $B$ is essential.

\begin{lemma}
\label{l:A_BP_B_essential}
  If $A$ is a boundary parallel annulus, then $B$ is essential (not a boundary
  parallel annulus).
\end{lemma}

\begin{proof}
Suppose that $B$ is a boundary parallel annulus.  If $\partial C = \emptyset$
then $\partial B = \partial A$, and hence $A$ and $B$ are isotopic, contradicting
the fact that $A$ was chosen to have least complexity.

So we have $A = B + C$, where all three have non-empty tight boundary and $A$
and $B$ are both boundary parallel annuli.  Then $\partial A = 2a$ and
$\partial B = 2b$, where $a$ and $b$ are tight essential curves since $A$ and
$B$ have tight boundaries.
Since all normal coordinates of  $\partial A$ and $\partial B$
are even, it follows that $\partial C = 2c$ for some tight essential curve $c$
by Lemma~\ref{lemCurveSummandsLeastComplexity}.

Each pair, $2b$ and $2c$, of parallel curves bounds an annulus in $\partial X$.
Normally isotope $B$ and $C$ so that these annuli are very thin and intersect
in a collection of squares, each contained in a face of $\partial X$.   Pick a
particular square.  Each of its corners is an endpoint of an intersection arc
between $B$ and $C$.  All corners have  the same normal sign.  It follows that
the exchange rectangles for corners on the same edge lie on opposite sides of
$A$; see Figure~\ref{exchange_square}.

Thus, there is an exchange rectangle $R$
properly embedded inside the
solid torus $T$ that is bounded by $A$ and $A_{\partial X} \subset \partial X$,
the annulus into which $A$ is isotopic.   Then $\partial R$ is a compressing
disk for $T$ since it meets $A$ in a pair of arcs that are essential by Lemma
\ref{lemNoHalfDiskPatches} and meets $A_{\partial X}$ in a pair of arcs that
are essential by Lemma \ref{lemExchangeArcNotParallel}.
But this contradicts the fact that $A$ is boundary parallel to $A_{\partial
X}$.  The unique, up to isotopy, compressing disk for $T$ meets $A$ and
$A_{\partial X}$ each in a single essential arc.
\end{proof}

\begin{figure}[ht]
  \begin{center}
      \includegraphics{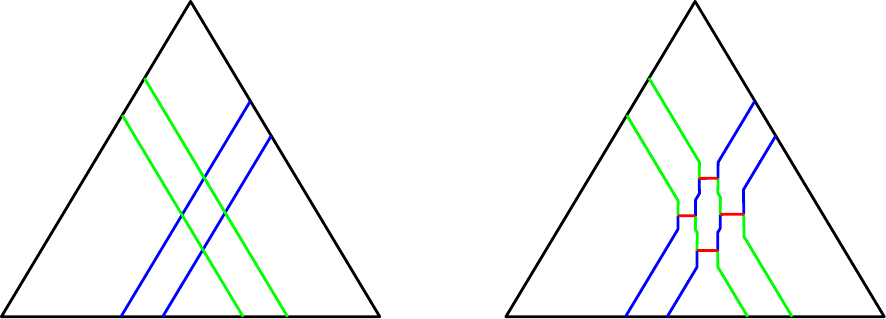}
    \end{center}
    \caption{Annuli between $2b$ and $2c$ intersecting in a square. After
      regular
      exchanges there are four exchange arcs in which exchange rectangles
      start. At
    least one of these rectangles (actually two) have to be inside $T$.}
    \label{exchange_square}
  \end{figure}

\begin{lemma}
\label{l:AessentialBessential}
  If $A$ is essential then $B$ is essential (not a boundary parallel annulus).
\end{lemma}

\begin{proof}

We first note, as in
\cite{Jaco:Algorithms-for-the-complete-decomposition-of-a-closed-3-manifold-1995},
that each disk patch is incompressible and boundary incompressible.  Any
(boundary) compressing disk for a patch has its boundary in the essential
surface $A$ and therefore meets it in a trivial curve.  This, in turn,
 implies the existence of a (half) disk patch.

In contradiction, suppose $B$ is a boundary parallel annulus.    By Lemma
\ref{lemNoDiskPatches}, $B$ and $C$ intersect in curves that are essential in
both.  These curves are either core loops or spanning arcs for $B$.

But they cannot be spanning arcs:  If so, then there is a boundary compressing
disk $D$  for $B$ so that $\partial D \cap B \cap \partial C = \emptyset$.
Choose $D$ to be such a
boundary compressing disk that meets $C$ in the minimal number of curves.   Let
$D'$ be a disk bounded by an innermost loop, outermost arc, or when $C \cap D =
\emptyset$, let $D=D'$.  By minimality $D'$ meets a patch in an essential curve
and is thus a (boundary) compressing disk for a patch.  This contradicts
Observation~\ref{o:theobs}.

We proceed assuming that $B \cap C$ is a collection of loops essential in both
$B$ and $C$.  Then the patches of $B$ and $C$ are annuli because each
component of $C$ has zero Euler characteristic.
%\martin{Added `using that each component of $C$ has zero Euler characteristic'}
%\eric{good}
As before, choose a boundary compressing disk $D$ for $B$ that meets $C$ in the
minimal number of curves.

By minimality of $D$ and essentiality of patches,
no curve of intersection is a closed loop or an arc with endpoints both in
$\partial X$.  It follows that $\gamma$, an outermost arc of intersection,
co-bounds a bigon or half bigon $D'$, with an arc $\beta \subset B \cap
\partial D $.   (Figure \ref{figBigons}, with different labelling).    By
minimality of $|D \cap C|$, $\beta$ is a spanning (essential) arc for an
annulus patch $A_B \subset B$ and $\gamma$ is a spanning arc for a patch $A_C
\subset C$.  The patches meet in either one (half bigon) or two (bigon)
intersection loops.   Moreover, because $X$ is irreducible and has
incompressible boundary, $A_B$ and $A_C$ are parallel across the solid torus
that
they bound.
Thus, Figure \ref{figBigons} is a cross-section of the total intersection,
$\textit{(half)bigon} \times S^1$.

What does the normal addition $A=B+C$ do with the patches $A_B$ and $A_C$? It
cannot trade them as in the figure, because then, performing only the trade and
no other exchanges produces normal surfaces $B'$ and $C'$ isotopic to $B$ and
$C$ but with fewer intersections, contradicting our assumption that we had
minimized $|B \cap C|$.   Nor, in the half bigon case, can it attach $A_B$ and
$A_C$.  For this means $A$ has a boundary parallel  annulus component.  This
rules out the half bigon case.   And, in the bigon case, it cannot attach $A_B$
to $A_C$ along both curves, for if it did, a component of $A$
would be a torus bounding a solid torus.
So $A_B$ and $A_C$ are attached along one intersection loop,
but not along the other.   But then $A_B \cup A_C$ is parallel to the zero
weight exchange annulus for the intersection loop where they were not attached.

Form a Haken sum $A'$ by performing all regular exchanges, except perform an
irregular exchange (fold) along the curve corresponding to the zero weight
annulus.  This is similar to the situation in Figure \ref{figNotAdjacent},
although the context is a bit different.      Then one component of $A'$ is
isotopic to $A$ and the other, call it $A''$,
 is a torus bounding a solid torus.    But
$A' \cup A''$ is not normal because of the irregular switch.   As in the end of
the proof of Lemma \ref{lemTwoDisks}, we conclude that $A' \cup A''$ is either
boundary compressible, not least length, or, $A$ is not tight.  All are
contradictions.
\end{proof}

This completes the proof of Proposition~\ref{propSummandsEssential}.
%\QED

%
%\begin{proposition}
%\label{propNormalizeEssential}
%Let $\C$ be a set of pairwise snug, tight essential curves in the boundary of
%an irreducible and boundary incompressible manifold $X$ with 0-efficient
%triangulation $\T$.  If $B_0 \subset X$ is an essential surface then $B_0$ is
%isotopic to a normal surface $B$ so that:
%\begin{enumerate}
%\item $\partial F$ is tight
%\item \{$\partial F\} \cup \C$ is pairwise snug
%\item each normal summand of $B$ is essential and tight
%\end{enumerate}
%\end{proposition}
%
%
%

\section{Constructing an annulus curve $\alpha$}
\label{s:annulus-curve}

%\eric{it is the construction here that leads to the tower of exponentials.  we iterate 4t times to construct the annulus curve, each corresponds to an exponentiation, I think.   it may be that Jaco--Tollefson can be applied all at once that would reduce complexity significantly. We should check that out}

As usual, we assume that $X$ is irreducible, orientable with incompressible boundary and presented via a 0-efficient triangulation.

%\uli{I adapted the following definition to deal with the disjointness issue that otherwise seems (to Martin and me) to arise in
%in the proof of Lemma~\ref{lemAddtoAnnuli}. I keep the original versions commented out, in case we want to revert.}
%\eric{i am fine with this.  }
%An \emph{annulus curve} $\alpha$ is any properly embedded
%curve which is a subset of the boundary curves,
%$\alpha \subset \partial \A$, of a collection of pairwise disjoint
%essential annuli $\A$. That is, for any annulus in $\A$ one or both boundary
%components of it belong to $\alpha$.
%The following proposition provides an annulus curve
%that can be used to track essential annuli.  Each boundary curve of a tight
%essential annulus either appears in the curve or meets the curve.
%\martin{Slightly changed the definition of an annulus curve since `any subset'
%may not be a curve. In reformulation, I ignore annuli that do not contribute to
%$\alpha$ hoping that it does not need additional explanation.}

\begin{definition}
An \emph{annulus curve} $\alpha$ is a properly embedded (multicomponent) normal curve in $\partial X$ with the following property: There exists a collection $\A$ of pairwise disjoint properly embedded essential annuli in $X$ such that $\alpha \subseteq \partial\A$ and $\alpha$ represents all normal isotopy classes of boundary components of $\A$ exactly once, i.e., for every annulus $A\in \mathcal{A}$ and every component $\gamma$ of $\partial A$, there is exactly one component of $\alpha$ that is normally isotopic to~$\gamma$.
\end{definition}
%\jirka{Shouldn't the definition say that $\alpha$ is normal?
%It should better be if it should be normally isotopic to something\ldots}
%\martin{Added normal; I hope that Eric agrees.}
%\eric{yes, agree}

The following proposition provides an annulus curve that can be used to track essential annuli.  Each boundary curve of a tight
essential annulus either appears in the curve or meets the curve.

\begin{proposition}
\label{propEssentialAnnuli}
Let $X$ be an irreducible, orientable manifold with incompressible boundary presented via a 0-efficient
%\eric{added 0-efficient}
triangulation with $t$ tetrahedra.   Then there is a tight normal annulus curve $\alpha$ so that:
\begin{enumerate}

\item[\rm(1)] $\alpha$ is maximal, by which we mean that
 if $A \subset X$ is an essential annulus or M\"{o}bius band whose boundary is  tight and disjoint from $\alpha$, then each boundary component of $\partial A$ is normally isotopic to a component
of~$\alpha$.

\item[\rm(2)] $|\alpha|$, the number of components of $\alpha$, is smaller than
 $4t$.
\item[\rm(3)] $\ell(\alpha)$ is bounded by a computable function of $t$.
\end{enumerate}
\end{proposition}

The bound for $\ell(\alpha)$ we obtain from our proof is an
$O(t)$-times iterated exponential, and this is currently a
bottleneck of the whole algorithm.

The proposition follows in a simple way from the next two lemmas.

\begin{lemma}
\label{lemAddtoAnnuli}
Suppose $\alpha_0$ is a tight, normal annulus curve in the boundary of an
irreducible, orientable manifold with incompressible boundary and  a
0-efficient triangulation. %\eric{added 0-efficient}
Let $A$ be an essential annulus with tight boundary that is disjoint from $\alpha_0$, and such that a component of $\partial A$ is not normally isotopic into  $\alpha_0$.

Then there is a tight normal annulus curve $\alpha$ such that $\alpha_0\subsetneq \alpha$ and $\ell(\alpha)$ is bounded by a computable function of $t$ and $\ell(\alpha_0)$.
%Then there is a tight curve $\alpha$ so that $\alpha_0 \cup \alpha$ is an annulus curve and $\ell(\alpha_0 \cup \alpha)$ is bounded by a computable function of $t$ and $\ell(\alpha_0)$.
\end{lemma}

\begin{proof}
The curve $\alpha_0$ can be regarded as a fence in the marked triangulation $(\T,\alpha_0 \cap \T^1)$.  The annulus $A$ has tight boundary disjoint from the fence $\alpha_0$. By isotoping the interior of $A$ (if necessary) while keeping its boundary $\partial A$ fixed,
we may assume that $A$ is tight and hence,  by Proposition \ref{propNormalize}, normal.

Write $A=F_1 + F_2 + \cdots + F_k$, a sum of connected fundamentals for the marked triangulation $(\T,\alpha_0 \cap \T^1)$.
By Lemma~\ref{lemCurveSummandsLeastComplexity} and by Proposition \ref{propMSummands}, respectively,
each boundary  $\partial F_i$ is tight and disjoint from the fence $\alpha_0$.   Since $\ell(\partial F_i)\leq | F \cap \T^1|$, by Proposition \ref{propFundamentals}, $\ell(\partial F_i)$  is bounded by a computable function of $t$ and $\ell(\alpha_0)$, the number of marking points.   By  Proposition \ref{propSummandsEssential}, each $F_i$ with non-empty boundary is an essential annulus or M\"obius band.
%\eric{need 0-efficient to use Prop 7.2}
%\martin{We are maybe not completely consistent regarding using the other
%assumptions as well. In this lemma, we also assume that the manifold is
%irreducible and boundary incompressible, I guess? At the beginnig of this
%section it is said `As usual we assume that $X$ is irreducible, boundary
%irreducible and orientable' (Btw., we are also perhaps not very consistent in
%using boundary irreducible vs. boundary incompressible.) Anyway, it is not
%completely clear whether these assumptions apply to this lemma. For example, we
%repeat orientable but we do not repeat other properties. I would either suggest
%to say more resolutely that we use these assumptions for all the claims in this
%section (then we also may not repeat 0-efficient); or add all assumptions in all cases.}

Moreover, there must be a summand with a boundary component that is not normally isotopic into $\alpha_0$. Otherwise, each component of the boundary sum $\partial A=\partial F_1+\partial F_2+\cdots+\partial F_k$, would be normally isotopic to a component of the fence $\alpha_0$.  This would imply that the summands are pairwise disjoint after a normal isotopy and, that each component of $\partial A$ is itself normally parallel to a component of the fence $\alpha_0$, contradicting our assumption.

Fix such a summand with a boundary component $\alpha'$ not normally isotopic into $\alpha_0$. We will construct a new annulus curve
$\alpha$ that contains both $\alpha_0$ and $\alpha'$ (which ensures $\alpha_0\subsetneq \alpha$).

Let $A'$ be the summand if it is an essential annulus or twice the summand if
it is a M\"{o}bius band.\footnote{The double of any normal surface $F$ can be
obtained by offsetting two copies, one to each side, of each normal disk of
$F$.  It follows that $2F$ is on the boundary of $F \twist [-1,1]$, an interval
bundle over $F$.  When the manifold is orientable, the bundle is a non-twisted
product if and only if the surface is orientable.   When the summand  $F$ is a
M\"obius band, $2F$ is the annulus $F \twist \{-1,1\}$.  Since every curve in
the boundary of an orientable manifold is two-sided, the boundary of the annulus is two copies of $\alpha'$.}

Let $\A_0$ be a collection of pairwise disjoint essential annuli witnessing that $\alpha_0$ is an annulus curve, i.e., $\alpha_0 \subseteq \partial \A_0$ and $\alpha_0$ represents all normal isotopy classes of boundary components of $\A_0$ exactly once.
By construction of the annulus $A'$, its boundary $\partial A'$ is disjoint from the fence $\alpha_0$ and hence from $\partial \A_0$, and one boundary component of $A'$ is the curve $\alpha'$.
%\uli{This step seems problematic if we don't know that $\partial A'$ is disjoint from the entire boundary of $\A_0$, for then it seems that $A'$ could intersect some annulus in $\A_0$ in an arc that is essential in that annulus.}
%\eric{this won't happen.   You can isotoped the intersection to be essential while maintaining disjoint from one end of each annulus.   not sure I stated that generality.   It follows that you don't meet the other end of each annulus as any essential arcs meets both ends}
%\uli{If we only assume that the fence $\alpha_0$ consists of some subcollection of (normal isotopy classes) of boundary components of $\A_0$, then it seems possible that $\partial A'$ is disjoint from $\alpha_0$ but intersects some other boundary component of $\A_0$.
%This is why I changed the definition of annulus curve above.} \eric{I think it would work out ok by essentiality of intersection of $A' \cap \A_0$ but I am perfectly ok changing}

We isotope $A'$, leaving the boundary fixed, to minimize components of intersection $|A' \cap \A_0|$, and we distinguish two cases.

The first case is that $A'$ misses $\A_0$. Then $\A = \A_0 \cup A'$ is a collection of pairwise disjoint properly embedded essential annuli. We define $\alpha$ to be the annulus curve corresponding to $\A$. More precisely, $\alpha'$ is one of the boundary components of $A'$. If the other boundary component is normally isotopic to $\alpha'$ or to some component of $\alpha_0$, we set $\alpha=\alpha_0\cup \alpha'$, and otherwise, we set $\alpha=\alpha_0\cup \partial A'$. Thus, $\alpha$ is an annulus curve, as witnessed by $\A$, $\alpha_0\subsetneq \alpha$, and $\ell(\alpha)$ is bounded by a computable function of $t$ and $\ell(\alpha_0)$ since $\ell(\partial A')$ is.

The second case is that $A'$ intersects $\A_0$. In this case, since $\partial
A$ and $\partial\A_0$ are disjoint, a standard innermost loop argument shows
that all curves of intersection are essential, i.e., core curves in the
annuli; see the proof of Lemma~\ref{lemMakeIntersectionEssential}.

If $A'$ meets $\A_0$, let  $\alpha'' \subset A'$ be the core intersection  curve closest to $\alpha'$, and let $\tilde{A}$ be the corresponding annulus in $\A_0$. Then $\alpha'$ and $\alpha''$ co-bound a sub-annulus $A'' \subset A'$ whose interior misses $\partial \A_0$.

Note that $\alpha''$ splits the annulus component $\tilde{A}$ of $\A_0$ into annuli $A_1$ and $A_2$.   Then $\alpha'$ is a boundary curve of both of the
annuli $A'' \cup A_1$  and $A'' \cup A_2$, both of these are disjoint from $\A_0$, and at least one of them is essential.\footnote{First note that for each of these annuli, any core curve is parallel to $\alpha''$, so a compression disk for either of these annuli would yield a compression disk for $A'$ (and for $\tilde{A}$ as well), a contradiction. Thus, if the annuli are not essential, they both must be boundary compressible. Since we assume that $\partial X$ is incompressible, it follows that each of the two annuli is boundary parallel, and thus
each is of them co-bounds a solid-torus with an annulus in the boundary and is a longitudinal annulus in its respective solid torus (longitudinal
means that it meets each meridional disk of the solid torus
once). Then the union of the two solid tori is a solid torus in which $\tilde{A}=A_1 \cup A_2$ is longitudinal, demonstrating that $\tilde{A}$ is boundary parallel and contradicting the assumption that $\A_0$ consists of essential annuli.}

Suppose w.l.o.g.\ that $A'''=A'' \cup A_1$ is essential. Then $\A=\A_0\cup A'''$ is a collection of pairwise disjoint essential annuli witnessing that $\alpha:=\alpha_0\cup \alpha'$ is an annulus curve (since one boundary component of $A'''$ is $\alpha'$ and the other one is a boundary component of $\tilde{A}\in \A_0$ and hence already represented in $\alpha_0$). Moreover, as in the first case, $\ell(\alpha)$ is bounded by a computable function of $t$ and $\ell(\alpha_0)$ since $\ell(\partial A')$ and hence $\ell(\alpha')$ is.
\end{proof}

%\eric{I have an issue here.  We may add both components of $\partial A'$ to $\alpha$.   But, if $A'$ meets $\A_0$, we only end up adding one, so one is unbounded by an annulus. Think this can be fixed.}
%\uli{I have rewritten the proof and changed the definition of $\alpha$ depending on the two cases, I hope the issue is resolved now.}

\begin{lemma}
\label{lemCurveBound}
Let $\alpha$ be an essential curve embedded in a a connected closed orientable triangulated surface $F$ with $f$ faces, such that no pair of components of $\alpha$ are isotopic. Then $|\alpha| < f$.
\end{lemma}

\begin{proof}

The result holds trivially for a torus, which requires at least two faces to triangulate and allows $\alpha$ to have at most one component.

Now suppose that the genus of the surface is $g\geq 2$. We claim that then $\alpha$ has at most $3g-3$ components. We may assume that $\alpha$ is a maximal collection of non-parallel curves, and hence it decomposes $F$ into $p$ pairs of pants (spheres with 3 holes) \cite[Sec.~8.3.1]{farbMargalit}.  Each pair of pants has Euler characteristic %$\chi(\textit{pair of pants})=
$-1$, and because the Euler characteristic of the boundary of a pair of pants is zero, % $\chi(\textit{boundary of pair of pants})=0$,
the Euler characteristic of the surface is additive over the pants, $\chi(F) = -p = 2-2g$.
Because each curve is on the boundary of two pairs of pants, we have
 $|\alpha| = \frac 3 2 p = 3g-3$.

We have $\chi(F)= 2-2g=f-e+v$, where $e=\frac{3}{2}f$ is the number of edges and $v$ is the number of vertices.   Then $g=\frac{f}{4}+1-v \leq \frac{f}{4}$ and $|\alpha|\leq 3g-3 \leq \frac{3}{4}f < f$, as desired.
\end{proof}

We now complete the proof of the main proposition.

\begin{proof}[Proof of Proposition \ref{propEssentialAnnuli}]
The annulus curve $\alpha$ can be constructed iteratively, starting with $\alpha = \emptyset$.  If, at any stage, the maximality property~(1) is not satisfied, we apply Lemma \ref{lemAddtoAnnuli}  to add a distinct component to $\alpha$.

We claim that the process terminates after adding at most $4t$ components: By Lemma \ref{lemCurveBound}, each boundary component of $X$ contains fewer components of $\alpha$ than it has faces (here, we are using that $\alpha$ is tight, so by Lemma~\ref{lemTightUnique}, the fact that no two components of $\alpha$ are normally isotopic also implies that no two of them are isotopic).
Thus, in total, $\alpha$ has fewer components than $X$ has boundary faces and the number of boundary faces is bounded by $4t$.  This is (2).

By Lemma \ref{lemAddtoAnnuli}, the first component added has length bounded by a computable function of $t$.  Every subsequent component added has length bounded by a computable function of  $t$ and the total length of the preceding components.   Since the number of components is less than $4t$, the total length of the curve is bounded by a computable function of $t$.
\end{proof}

\section{Curves bounding  boundary parallel annuli}

In the previous section we constructed the annulus curve $\alpha$, which will be used to bound the coefficients of essential annulus summands
in the planar (almost) meridional surface~$P$.
In this section we construct $\Gamma$,  a collection of
curves bounding normal boundary parallel annuli.
Later, the curves of $\Gamma$ will act as fences,
and will be used to rule out boundary parallel annulus summands altogether.

\begin{proposition}
\label{propB}
Suppose $X$ is an irreducible, orientable manifold with incompressible boundary and
presented via a 0-efficient
%\eric{added 0-efficient}
triangulation with $t$ tetrahedra.    Let $\alpha$ be the tight
normal annulus curve given by Proposition \ref{propEssentialAnnuli}.  Then
there is a finite set $\Gamma$ of tight essential curves, possibly
mutually intersecting, such that:

\begin{enumerate}
\item[\rm(1)] If $B$ is a normal boundary parallel annulus with tight boundary that is
  disjoint from $\alpha$, then each boundary component of $B$ is normally
  isotopic either to a component of $\alpha$ or to a curve of~$\Gamma$.
%  \martin{Isotopic changed to normally isotopic in this item.}
\item[\rm(2)] $|\Gamma|$ and $\max_{\gamma\in \Gamma}\ell(\gamma)$
are bounded by a computable function of~$t$.
\end{enumerate}
\end{proposition}

\begin{proof}
Let $\alpha$ be given by Proposition \ref{propEssentialAnnuli}.   Then $\alpha$
is a fence in the marked triangulation $(\T, \alpha \cap \T^1)$. Let $\Gamma$ be
the set of the boundaries of all boundary parallel annuli that are fundamental in the marked triangulation $(\T, \alpha \cap \T^1)$ and disjoint from the fence $\alpha$.   Then (2) follows from Proposition~\ref{propFundamentals}.

Now we want to verify (1). Let $B$ be a normal boundary parallel annulus with tight
boundary disjoint from the fence $\alpha$. By isotoping $B$ we can assume that
$B$ is least complexity.
%\martin{\confirm{Eric started by verifying (1) for annuli $B$ of least
%complexity. I think that we do not a priori know that $B$ is least complexity,
%but we can isotope it so. However, it is not fully clear to me where is the
%`least complexity' used anyway. Didn't dare to remove it.}}

%\eric{We need B least complexity as this is required by Prop 7.2.      Least complexity is over normal surfaces isotopic to B (as opposed to tight which is over all surfaces isotopic to B).   For a BPA tight, may not be normal.   Here is where we need it:   Suppose B = B' + T, where B' is a boundary parallel annulus and T is a torus. Then B is not fundamental.  But neither is it least complexity, because if B' is a BPA, then B and B' are isotopic because they have the same boundary (up to isotopy only one boundary parallel annulus with given boundary).  But B' has lower complexity.
%Please feel free to remove comments if you are ok.  We can add more if you think it clearer.}

If $B$ is a fundamental, then its
boundary has already been included in $\Gamma$ and we are done. If not, then $B$
can be written as a sum of fundamentals for $(\T, \alpha \cap \T^1)$, $B = F_1
+ F_2 + \cdots + F_k$.  By Proposition \ref{propSummandsEssential}, each $F_i$
with boundary is an essential annulus or M\"obius band.  Since $B$ is disjoint
from the fence $\alpha$, each $F_i$ has boundary disjoint from $\alpha$.
Hence, by Proposition \ref{propEssentialAnnuli}, each $F_i$ has boundary
components normally isotopic to $\alpha$.  But as observed in the proof of
Proposition~\ref{propEssentialAnnuli}, this implies that each boundary
component of $B$ is normally isotopic to $\alpha$, as required.
\end{proof}

%%%%%%%%%%%%%%%%%%%%%%%%%%%%%%%%%%%%%%%%%%%%%%%%%%%%%

\section{Planar meridional surfaces}
\label{s:planar-meridional-surf}

In this section we consider a planar (almost) meridional surface
$P$ in $X$ and a collection $\A$ of disjoint essential annuli. The collection
$\bd\A$ of the boundaries of the annuli in $\A$ forms a collection
of disjoint curves (loops) in $\bd X$, and $\bd P$ is another collection
of disjoint loops.

We want to move $P$ by means of
a self-homeomorphism $h\colon X\to X$ in such a way that
the number of intersections of these two collections, $\bd\A$ and
$\bd P$, becomes bounded; more precisely, we need a bound
of the form $C(t)|\A|\cdot|\bd P|$.
This is formulated in
Proposition~\ref{pPMeetsA} below; the self-homeomorphism $h$
is going to be one of two ways of changing the original embedding
of $X$ in $S^3$ in order to get a short meridian.

First we collect
auxiliary results.
We begin with a corollary of the main result of~\cite{matousekSedgwickTancerWagner}, 
which was developed for the purpose of proving a result in the spirit of Proposition~\ref{pPMeetsA}.

\begin{lemma}[{\cite[Cor. 1.6]{matousekSedgwickTancerWagner}}]
\label{l:untangle}
Let $S$ be connected surface, i.e., a connected compact $2$-manifold with boundary,
of genus $g$. Let
$(\alpha_1,\ldots,\alpha_m)$ be a system of disjoint
curves (properly embedded arcs and loops) in $S$, and let
$(\beta_1,\ldots,\beta_n)$ be another such system.
Then there is a homeomorphism $\varphi\colon S\to S$ fixing
$\partial S$ pointwise such that the total number of intersections
of $\alpha_1,\ldots,\alpha_m$ with $\varphi(\beta_1),
\ldots,\varphi(\beta_n)$ is at most $K(g)mn$, where
$K(g)$ is a computable function depending only on $g$ (in fact, $K(g)=O(g^4)$).
\end{lemma}

%\jirka{We should probably include this lemma in the untangling
%paper in the future and just cite it.}
%This follows from the results of
%\cite{matousekSedgwickTancerWagner} as sketched below.
We remark that for our further approach a bound of the form
$K(g,m)n$ would also be sufficient. Such a bound, even independent of $g$, was obtained independently
by Geelen, Huynh, and Richter \cite{Geelen:Explicit-bounds-for-graph-minors-2013}, but only under the additional assumption that the union of the $\beta_i$ does not separate~$S$. Thus, we cannot directly use their result here;
the extra assumption could probably removed, but it is easier to use the bounds
from~\cite{matousekSedgwickTancerWagner}.

%\begin{proof}
%As in \cite{matousekSedgwickTancerWagner},
%let $f_{g,h}(m,n)$ be the maximum number of intersections
%in the setting of the lemma for $S$ orientable of genus $g$ and with
%$h$ holes (boundary components), and let $\hat f_{g,h}(m,n)$
%be the analogous quantity for non-orientable~$S$.
%
%Then by \cite[Th.~1.1]{matousekSedgwickTancerWagner},
%for planar $S$ we have $f_{0,h}=O(mn)$. By
%\cite[Prop.~1.3(ii)]{matousekSedgwickTancerWagner},
%$f_{g,h}(m,n)\le f_{0,h+1}(cg(m+g),cg(n+g))=
%O(g^2(m+g)(n+g))=O(g^4mn)$. For the non-orientable
%case, \cite[Prop.~1.5]{matousekSedgwickTancerWagner}
%gives $\hat f_{g,h}(m,n)\le f_{\lfloor (g-1)/2\rfloor,h+1+(g \mod 2)}(c(m+g),c(n+g))=
%O(g^4 mn)$ as well.
%%\martin{I have still changed $h + 2$ into $h+1+(g \mod 2)$. I think that it is
%%easier just to use $h+1+(g \mod 2)$ as black-box than to argue that the quantity is monotone in $h$.}
%\end{proof}

We also need the following, probably standard, lemma.

\begin{lemma}
  \label{l:loops}
  Let $G = (V,E)$ be a graph with $n>2$ vertices embedded in $S^2$,
possibly with loops and multiple edges. Let us assume that
no two parallel edges (connecting the same two vertices)
and no two parallel loops (attached to the same vertex)
are isotopic by an isotopy fixing the end-vertices and
avoiding the other vertices.
 %of the edges in $S^2 \setminus V'$ where $V'$ is
 %$V$ minus the endpoints of $e$ and $e'$.
%\jirka{The previous formulation was somewhat hard to understand;
%I hope though that I haven't missed some subtle point. Also added
%the assumption that no two parallel loops are isotopic.}
 We also assume that there is no contractible  loop $\ell$;
that is, both the interior and exterior of each loop contain a vertex.
% can be contracted to its endpoint $v$ in $S^2 \setminus (V' \setminus \{v\})$.
 Then $|E|\le3n - 6$.
\end{lemma}

\begin{proof}
  If $G$ contains neither loops nor multiple edges, then this is just
the usual bound for the number of edges of a simple planar graph.
 It remains to resolve loops and multiple edges.

  First let $\ell$ be a loop with an endpoint $v$.  It splits $S^2$ into two
  regions $X$ and $Y$. Let $F_X$ resp. $F_Y$ be the face of $G$ inside $X$
  resp. $Y$ bounded by $\ell$. Then $F_X$ has to contain a vertex $v_X\ne
  v$, for otherwise, $\ell$ can be contracted to $v$ or it
  is isotopic to another loop with endpoint $v$.

  Similarly, we have a vertex $v_Y$ in $F_Y$. Note
  that $v_X$ and $v_Y$ are not connected with an edge. We can
  remove $\ell$ from the graph and connect $v_X$ and $v_Y$ with an edge,
  keeping the graph embedded in $S^2$ and satisfying the isotopy assumptions.
  This way we can remove all loops without increasing the number of edges;
%\immfig{loops}
  see Figure~\ref{f:loops}.

 Similarly, if we have two parallel edges, we can remove one of them and add a
  new edge as compensation, reducing the number of pairs of parallel edges.
  In this way, we get a simple planar graph and the desired bound.
\end{proof}

\begin{figure}[t]
\begin{center}
  \includegraphics{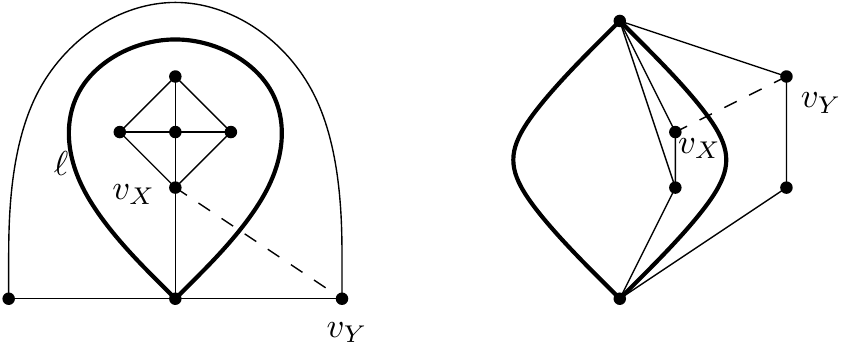}
\end{center}
\caption{Removing loops and parallel edges.}
\label{f:loops}
\end{figure}

\begin{proposition}
\label{pPMeetsA}
Let $X$ be an orientable, irreducible manifold with incompressible boundary.
%\martin{Added the assumptions on $X$. If you think that $X$
%orientable is sufficient instead of $X$ embeds into $S^3$, then modify it so,
%please.}
%\eric{We only need orientable}
Let $P \subset X$ be a properly embedded planar surface
%, not a disk,  \eric{added not a disk because bound would be negative.  can remove this depending on how bound changes}
that is either essential, or  strongly irreducible and boundary strongly irreducible. Let $\A \subset X$ be a collection of pairwise disjoint essential annuli.   Then there is a homeomorphism $h\colon X \to X$ so that $|\partial h(P) \cap \partial \A|  < C|\A|\cdot|\partial P|$, where $C=C(t)$ is a computable function of the number $t$ of tetrahedra in the triangulation of~$X$.
\end{proposition}

%\eric{ should we try to remove planar assumption?}

\begin{proof}
Using either Lemma \ref{lemMakeIntersectionEssential} or Lemma \ref{lemSurfacesMeetEssentially},  %\eric{added ref to lemma 4.3}
%\martin{OK}
we may isotope $P$  so that its
intersection with $\A$ is essential, that is every component of $P \cap \A$ is
a curve that is essential in both $P$ and $A$.  This implies that the result
holds when $P$ is a disk,
 for then $P$ contains no essential curves, and thus $P \cap
\A$ is empty.   We proceed assuming $|\partial P|>1$.

In $\A$ every
intersection arc is a spanning arc and every intersection loop is the core
curve of an annulus.  This is illustrated below in the left picture, while the
right picture shows the intersection curves in~$P$:
\immfig{bandsInP}

Say that two arcs belong to the same \emph{parallel class} if they are isotopic
in $P$.  If $|\partial P|=2$, then $P$ is an annulus and there is at most one
parallel class of intersection arcs.   When $|\partial P|>2$, form a planar
graph by treating each boundary component as a vertex and each parallel class
of arcs as an edge. The number of edges, hence parallel  classes, is bounded by
$3(|\partial P|-2)$ by Lemma~\ref{l:loops}.  We can cover the cases when $P$ is
an annulus or disk by reducing this last bound slightly.  In all cases the
number of parallel classes of arcs is bounded by $3(|\partial P|-1)<3|\partial
P|$.

A \emph{band} in $X$ is an embedded, but not properly embedded, rectangle
meeting $\partial X$ in precisely its top and bottom sides.  For each parallel
class of intersection arcs in $P$, we may choose a band $B_j$ that is a
sub-surface of $P$, contains all intersection arcs in the class, and
meets no other curves of intersection.   Then $\BB$, the union of all such
bands,  has at most $3|\partial P|$ components and contains all arcs, but
no loops, of the intersection $P \cap \A$.

%\martin{Renaming done in such a way that $\BB$ stands for bands and $\Gamma$
%for the curves taking care of boundary parallel annuli.}
%\jirka{I suggest to rename the $\B$ elsewhere to something else.
%I've used a different macro here so that it can be done easily.}

%\martin{I renamed $\B$ to $\BB$. I am not sure what is this the best choice but
%I do not have better idea (if we want to keep $\B$ solely for a collection of
%curves from the previous section).}
%\eric{could also use $\mathcal R$ for rectangles.   but then we would probably have to rename the B's too.   I am not picky on this.}

Next, let us draw the core curve $\alpha_i$ for every annulus in $A_i\in\A$,
and  a curve $\beta_i$ parallel to the top and bottom
sides (those in $\bd X$) in the middle of each band $B_j$.
Let us think of these $\alpha_i$ and $\beta_j$ as being (locally) horizontal
and lying in the same level; then, again locally, $A_i$ is a
vertical ``wall'' through $\alpha_i$ and $B_j$ is a vertical ``wall''
through $\beta_j$. We have the $A_i$ and $B_j$ fibered with segments, as in the
left picture, and so the union $\A\cup\BB$ has the structure
of an $I$-bundle $M_0$ over $(\bigcup_i\alpha_i)\cup(\bigcup\beta_j)$,
where $I$ is the interval $[-1,1]$; see the left picture below:
%\martin{Explicitly said that the left picture is interesting at the moment,
%because the picture is quite complex.}
\immfig{iband}
As the picture illustrates, some of the $A_i$ or $B_j$ may be twisted
between the intersections with the others.

Next the $I$-bundle structure on $\A\cup\BB$ can be extended to a sufficiently
small regular neighborhood $N(\A\cup\BB)$. Indeed, we can consider the regular
neighborhood as the star of $\A\cup\BB$ (say in the second barycentric
subdivision of some triangulation); see~\cite[Chapter
3]{Rourke:Introduction-to-piecewise-linear-topology-1972}. Therefore,
$N(\A\cup\BB)$ has locally structure as product of $\A\cup\BB$ with $I$.

%\martin{I have modified the explanantion above. The following setnence should
%be removed, if agreed.}
%Next, because a regular neighborhood is a mapping cylinder (Theorem 1.1.7 of \cite{Matveev:AlgorithmicTopology-2007}), the $I$-bundle structure on $\A\cup\BB$ can be extended to a sufficiently
%small regular neighborhood $N(\A\cup\BB)$.

We obtain an $I$-bundle
$M$ over a base surface $S$ forming a narrow ribbon along the
$\alpha_i$ and the $\beta_j$. This is illustrated locally in the right
picture above.
%\uli{Add a reference for the fact that the bundle structure can be
%extended to a small neighborhood.}
%\eric{ok, Martin added something that is good with me}

%Note that $\A \cup \BB \subset X$
%has the structure of an $I$ bundle over $\alpha \cup \beta$, where
%$\alpha$ consists of the cores of the annuli $\A$ and $\beta$ is the collection
%of horizontal bisectors from the bands $\BB$ and $I$ is the interval $[-1,1]$.
%The interval bundle structure extends to a regular neighborhood, $N(\A \cup \BB)
%\subset X$, a section of which is a surface $P'$ passing through $\alpha \cup
%\beta$; see Figure~\ref{f:bundle_surface}.

%\martin{\confirm{I have changed $P' = N(\alpha \cup \beta)$ into saying that
%    $P'$ passes through $\alpha \cup \beta$. Because I would understand to $P'
%  = N(\alpha \cup \beta)$ as saying that $P'$ is a (3-dimensional) regular
%neighborhood of $\alpha \cup \beta$ which is not what we want. I have also
%added a picture here trying to clarify what do we mean. I am not very good in
%hand-drawing pictures but I could not really do this one on computer. Of
%course, you may not like it and remove it.}}
%
%
%\begin{figure}[ht]
%\begin{center}
%  \includegraphics[width=10cm,height=8cm]{bundle_surface_cut}
%\end{center}
%\caption{The surface $P'$ in case that $\A$ consists of a single annulus and $\BB$
%consists of a single band.}
%\label{f:bundle_surface}
%\end{figure}

The plan is now to use Lemma~\ref{l:untangle} (untangling curves in a
surface) for the systems of curves $\alpha_i$ and $\beta_j$ within $S$,
which yields a self-homeomorphism $\varphi\colon S\to S$ fixing $\bd S$
pointwise, such that the number of intersections of the
$\alpha_i$ with the $\varphi(\beta_j)$ is suitably bounded.
Then we want to extend $\varphi$ to a bundle self-homeomorphism
$h\colon M\to M$ that is the identity over $\bd S$
(i.e., on the vertical walls bounding $M$ in the picture). After that,
$h$ can be extended identically to $X\setminus M$ and we will be done.

There are two issues to be handled. First, in order to
use Lemma~\ref{l:untangle}, bound the genus of
each component of $S$ by a computable function of $t$; we will actually
obtain an $O(t)$ bound.

To this end, we observe that $S$ is double-covered by
a surface $\tilde S := N(\A \cup \BB) \cap \partial X$.
Let $K$ be a component of
$S$ and $\tilde K$ be the corresponding double cover of $K$ in $\tilde S$.
In particular, $\chi(\tilde K) = 2\chi(K)$. For a
surface $F$, we let $b(F)$
denote the number of boundary components, and define $g_e(F) := 2 - \chi(F) -
b(F)$. If $F$ is connected, this value is known as the Euler genus of $F$. Then we get
\[
g_e(K) = 1 + g_e(\tilde K)/2 - b(K) + b(\tilde K)/2 \leq 1 + g_e(\tilde K)/2.
\]
%\eric{not sure why we have the $g_e(\tilde S)/2$ in the above equation, don't we work below with $g_e(\tilde K)$}
%\eric{more nitpicking,  $g_e$ is defined for connected surface.   If $K$ is orientable,  $\tilde K$ has two components and so we might need the sum of $g_e$'s or some other minor adjustment.}
%\martin{solved the nitpicking issue. Here I started to regret that I did not
%adapt your original approach distinguishing orientable and non-orientable.}
%since every boundary component of $K$ introduces at most two boundary
%components of $\tilde K$.
%\martin{\confirm{I did not want to argue that it introduces exactly two
%compononets. Is it OK?}}

Let $\tilde Q$ be a component of $\tilde K$ ($\tilde K$ has two components if
$K$ is orientable). Since $\tilde Q \subseteq \partial X$, we
have $g_e(\tilde Q) \leq g_e(X^{\partial}_Q)$ where $X^{\partial}_Q$ is the component of $\partial X$
containing $\tilde Q$. The Euler genus of $X^{\partial}_Q$ is bounded by $O(t)$,
since $\bd X^{\partial}_Q = \emptyset$ and $|\chi(X^{\partial}_Q)|=O(t)$ (note that the
 number of triangles, edges and vertices in triangulation of $X$ are all
 bounded by $O(t)$).
%\eric{I agree with the O(t) statement, do we need to justify it?}
%\martin{Reformulated and explained. I also changed $\partial X$ to
%$X^{\partial}_K$ in order to deal correctly with case that $\partial X$ has
%more boundary components, which I overlooked in first writing. }

Altogether $g_e(\tilde Q) = O(t)$, and
since $\tilde K$ has at most
two components, $g_e(K)=O(t)$.
Since the genus of a surface is at most twice the Euler
genus, we also obtain $g(K)=O(t)$.

By applying Lemma~\ref{l:untangle} as announced above,
working in each component $K$ of $S$ separately and then
summing up, we obtain a self-homeomorphism $\varphi$ of $S$,
fixed pointwise on $\bd S$, such that the total number of
intersections of the $\alpha_i$ with the $\varphi(\beta_j)$
is at most $C_0(t)|\A|\cdot|\BB|\le 3 C_0(t)|\A|\cdot|\partial P|$,
where $C_0(t)$ is a computable function of~$t$.%depending only on~$t$.
%\uli{I have also replaced ``depending on'' by ``computable function of'' here.}

It remains to deal with the second and last issue, namely, showing
that $\varphi$ extends to a bundle self-homeomorphism $h\colon M\to M$ that
 is identical over $\partial S$.
Here we may assume w.l.o.g.\ that $S$, and hence $M$, are connected.

By the assumption, $M$ is embedded in $X$, and so it is orientable.
%Might also mention the example over the figure 8
%that this doesn't work for $B$ more general.}
It follows that if the surface $S$ is orientable, then
  $M$ is actually the product $S \times [-1,1]$, and the extension
of $\varphi$ to $h$ is obvious.

So let $S$ be non-orientable; then $M$ is non-trivially twisted and there
are no global coordinates. For a subset $S'$ of $S$ we will use a notation $S'
\twist I$ for the subbundle of $M$ consisting of points of $M$ that project to
$S'$. (In particular, we also regard $M$ as $S \twist I$.)

For any connected non-orientable
surface $S$ there is a non-separating arc $\sigma \subset S$ with both
endpoints on the same boundary component
for which $S_{\sigma}$, which is $S$ cut along $\sigma$,
is an orientable surface. We also let $f_\sigma\colon S_\sigma\to S$
be the map gluing $S_\sigma$ back to~$S$.\footnote{It is not difficult to find
such an arc $\sigma$ in the projective plane or the Klein bottle with a single
hole. Any other
nonorientable surface with nontrivial boundary can be obtained by adding handles (not across the desired
arc $\sigma$) and holes to one of the two
surfaces above (adding a handle increases the non-orientable genus by $2$).}
%\jirka{Reference? MSTW or something else?}

Thus, after cutting $M$ along $\sigma\times I$, we obtain a product
bundle $M_\sigma$, homeomorphic to $S_\sigma\times I$.
In the boundary of $M_\sigma$  we
see two rectangle scars from cutting along $\sigma \times I$.   We get
the twisted bundle $M$ back by gluing $M_\sigma$ to itself along the
rectangles so that the top of one is glued to the bottom of the other.
%Let $f_\sigma\colon M
%
%We also consider a map $f_{\beta}\colon B_{\beta} \to B$ which identify the
%scars obtained by cutting along $\beta$.

Now the given homeomorphism $\varphi\colon S\to S$ also takes $\sigma$
 to a curve $\varphi(\sigma)$ that has
the same separation properties. We define $S_{\varphi(\sigma)}$ and
$f_{\varphi(\sigma)}$ in the same way as $S_\sigma$ and $f_\sigma$ above.
The homeomorphism $\varphi$ also
induces a homeomorphism $\varphi'\colon S_{\sigma} \to S_{\varphi(\sigma)}$
satisfying $f_{\varphi(\sigma)} \circ \varphi' = \varphi \circ f_{\sigma}$.

The homeomorphism $\varphi'$ can be extended to a homeomorphism of the product
bundles $h'\colon S_\sigma \times I\to S_{\varphi(\sigma)} \times I$ in
two ways (by either keeping $I$ or swapping it). By gluing back along the
rectangular scars, $h'$ induces a homeomorphism $h\colon M\to M$.

Recall that $M$ is orientable since it embeds into $X$.
Since we had two choices for $h'$ we select one for which $h$ is an orientation
preserving automorphism. It follows that whenever $K$ is a boundary component
of $S$, then $h$ is the identity on $K \twist I$ (it cannot flip $I$ here
since such a flip would reverse the orientation on $\partial M$, hence on $M$).
\end{proof}

%\martin{Saying that $h$ is orientation preserving solves the issue as Eric
%described me on the skype.}
%\eric{changed to 'flip', and 'reverse orientation on $M$.'   feel free to change if you don't like.  We could say ' reverse orientation on $\partial M$, hence on $M$.' }
%\martin{OK. Used a version with $\partial M$.}

%%%%%%%%%%%%%%%%%%%%%%%%%%%%%%%%%%%%%%%%%%%%%%%%%%%%%

\section{Proof of the short-meridian theorem}
\label{sec:proof-main}

We already have almost all of the ingredients ready to prove
Theorem~\ref{t:short-meridian}, following the outline
from Section~\ref{s:outline}.

We assume that $X$ is irreducible, has incompressible boundary
(which we may assume to be nonempty), embeds in $S^3$, and has a 0-efficient triangulation with $t$ tetrahedra.
%\eric{added 0-efficient}
Note that the second conclusion of the following lemma may
require a re-embedding of $X$ into $S^3$.

\begin{lemma}
\label{lemWitness}
$X$ contains a planar surface $P$ so that:

\begin{enumerate}
\item[\rm(1)] $P$ is essential, or, strongly irreducible and boundary strongly irreducible, and
\item[\rm(2)]
 $\partial P$ is meridional or almost meridional in some embedding of $X$ in $S^3$.  \label{iMerid}
\end{enumerate}
\end{lemma}

\begin{proof}
Since $X$ embeds in $S^3$, we can apply the result of Fox
\cite{Fox:On-the-imbedding-of-polyhedra-in-3-space-1948} that shows $X$ may be embedded so that $S^3 \setminus  interior (X)$ is a collection of handlebodies.

  Then we may view $X$ as the exterior, $X= S^3 \setminus N(\Gamma)$, where $\Gamma$ is a graph consisting of a spine of each   handlebody.   In this context, we may apply Theorem 3 of Li \cite{Li:Thin-position-and-planar-surfaces-for-graphs-in-the-3-sphere-2010} that states that $X$ contains a planar surface that is either: (1)  meridional, strongly irreducible, and boundary strongly irreducible, or (2) almost meridional and essential , or (3) non-separating, almost
meridional, and incompressible.  By Lemma~\ref{lemEssentialAlmostMeridional},
case (3) reduces to   case (2), and the lemma follows.
\end{proof}

\begin{lemma}
Let $P$ be a surface satisfying the conditions of Lemma~\ref{lemWitness},
 and  let $h\colon X \to X$ be a homeomorphism.
Then $h(P)$ satisfies the conditions of Lemma~\ref{lemWitness}
for some re-embedding of~$X$.
\end{lemma}

\begin{proof}
Because the homeomorphism $h$ maps any disk in $X$ to a disk  in $X$, $P$ is essential if and only if $h(P)$ is essential; and, $P$ is strongly irreducible and boundary strongly irreducible if and only if $h(P)$ is strongly irreducible and boundary strongly irreducible.

Let $e\colon X \to S^3$  be the embedding for which $P$ is (almost)
meridional.   Then $r := e \circ h^{-1}\colon X \rightarrow S^3$
is a re-embedding of
$X$, and any component $e(\mu) \subset e(\partial P)$ bounds a disk in $S^3
\setminus X$ if and only if $r(h(\mu))$ bounds a disk in $ S^3 \setminus
(r\circ h(X))$.   Then  $P$ is (almost) meridional in the original
embedding if and only if $h(P)$ is meridional in the re-embedding.
\end{proof}

We can thus place additional constraints on $P$.
Let $\alpha$ be the tight essential annulus curve given by
Proposition~\ref{propEssentialAnnuli} and let $\Gamma$ be the set of tight essential curves bounding boundary parallel annuli given by Proposition~\ref{propB}.

\begin{assume}
\label{aMinimize}
Among planar surfaces $P$ satisfying the conclusion of Lemma \ref{lemWitness} choose $P$ to minimize,  in this order:
\begin{enumerate}[\rm(1)]
\item $| \partial P \cap \alpha|$;
\item $\compl(\partial P)$, and hence $\ell(\partial P)$; and
\item $|\partial P \cap (\alpha \cup \Gamma)|$.
%\jirka{Changed from $\partial P \cap \alpha \cap \Gamma$.}
%\eric{good}
\end{enumerate}
\end{assume}

The next lemma shows that $P$'s intersections with $\alpha$ are bounded by a
linear function of  $\chi(P)$ for $t$ fixed.
%The remainder of the proposition's proof will be to show that its length is as well.

\begin{lemma}\label{lemPSnugTight}Under Assumption \ref{aMinimize}, we have
\begin{enumerate}[\rm(1)]
\item $| \partial P \cap \alpha | \leq C_0(t)|\partial P|$, where $C_0(t)$
depends on $t$, the number of tetrahedra;
%\eric{this is problematic at the moment}

\item $\partial P$ is tight and essential; and
\item $\{\partial P\} \cup \{\alpha\} \cup \Gamma$ is pairwise snug.
\end{enumerate}
\end{lemma}

\begin{proof}
We have $\alpha \subset \partial \A$, where $\A$ is a collection of pairwise
disjoint essential annuli.  We choose $\A$ to minimize $|\A|$ subject to $\alpha
\subset \partial \A$.  Then $|\A| \leq |\alpha|$, because each $A \in \A$ must
contribute at least one unique component to $\alpha$.

By Proposition~\ref{pPMeetsA}, there is a homeomorphism of $X$ so that the image of the planar
surface, call it $P$, satisfies $|\partial P \cap \partial \A| \leq C(t)|\A|\cdot|\partial P|$
for a suitable $C(t)$.   Now Proposition~\ref{propEssentialAnnuli}
guarantees that  $|\alpha|\leq 4t$, and hence $|\partial P \cap \partial \A| \leq
4C(t)t|\partial P|$ by Assumption~\ref{aMinimize}.

Because $P$ is either essential or strongly irreducible,
$\partial P$ consists of essential curves
(Lemma~\ref{lemSIBSIhasEssentialBoundary}).
 Thus, $\partial P$ can be tightened; this may possibly
 increase $|P \cap \partial \A|$.

However, since $\alpha$ and $\Gamma$ are tight,
using Lemma~\ref{lemAddToSnug} repeatedly, we can make $\{\partial P\} \cup
\{\alpha\} \cup \Gamma$ pairwise snug within their normal isotopy classes.
In particular, after this step $|P \cap \partial \A| = i(\partial P, \partial
\A)$ where $i(...)$ is the geometric
intersection number. This again guarantees that $|P \cap \partial \A|$ is
minimized.

Therefore, we can simultaneously achieve
$\partial P$ tight and $\{\partial P\} \cup
\{\alpha\} \cup \Gamma$ pairwise snug.
Hence both of these properties hold under
Assumption~\ref{aMinimize}.
\end{proof}

%\martin{\confirm{I edited the proof above according to the instructions in the email. I
%will be glad for brief confirmation (maybe not necessarily from Eric).}}
%\eric{yes, good with me.}

\begin{lemma}
\label{lemNormalizeP}$P$ can be isotoped, without changing $\partial P$, to be normal or almost normal.
\end{lemma}

\begin{proof}
If $P$ is essential, then, since $\bd P$ is tight,
$P$ itself can be tightened without changing $\bd P$.
Then $P$ is normal by Proposition~\ref{propNormalize}.

 If $P$ is strongly irreducible and boundary strongly irreducible,
 then the main result of \cite{bachmanDerbyTalbotSedgwick} states
that $P$ is isotopic to an almost normal surface.   Moreover, in the proof of
Proposition~3.1 of \cite{bachmanDerbyTalbotSedgwick} it is assumed  that
$\partial P$ is least length (see Lemma~3.9), which is satisfied when $\partial
P$ is tight.   The  additional normalization steps taken there
isotope the interior of $P$ without changing its boundary, so $\partial P$
is also fixed in the almost normal case.
\end{proof}

\heading{The average length argument. }
We
mark the triangulation $\T$ of $X$ with marking
$M = (\alpha \cup \Gamma) \cap \T^1$.
Thus, the (almost) meridional, (almost) normal planar surface $P$
can be written as a sum of fundamental (almost) $M$-normal surfaces,
$ P = \sum k_i F_i$, and its boundary is the sum of the boundary
curves of the fundamentals:
\[
\partial P = \sum k_i \partial F_i.
\]

Since $\partial P$ is essential and tight, the boundary of each summand is essential and tight
by Lemma~\ref{lemCurveSummandsLeastComplexity}.
Each $F_i$ falls into at least one of the following categories:
\begin{enumerate}
\item $\bd F_i=\emptyset$; \label{case:emptybd}
\item $F_i$ is almost normal;\label{case:almostn}
\item $\chi(F_i)>0$ and $F_i$ is normal;
\label{case:chi>0}
\item $\chi(F_i)=0$, $F_i$ is normal, and one of the following hold:
\begin{enumerate}
\item $F_i$ is a compressible annulus;\label{case:a}
\item $F_i$ meets $\alpha$;
\label{case:b}
\item $F_i$ is a boundary parallel annulus disjoint from $\alpha$;
\label{case:c}
\item $F_i$ is an essential annulus or M\"{o}bius band
disjoint from $\alpha$. \label{iDisjointAnnuli}
\end{enumerate}
\item $\chi(F_i) < 0$ and $F_i$ is normal.\label{case:nla}
\end{enumerate}

We will bound the total length $\ell(\partial P)$ by bounding the coefficients of each of these types in the boundary sum $\partial P = \sum k_i \partial F_i$.
Obviously, we can ignore summands with empty boundary (Case~\ref{case:emptybd}).
Since there can be at most one exceptional piece in the almost normal
case, we have $k_i \leq 1$ in Case~\ref{case:almostn}.

\begin{lemma}
\label{lemNoPositive}
There are no normal %or almost normal
summands with $\chi(F_i)>0$.
\end{lemma}

\begin{proof}
Such summands are either spheres, or projective planes, or disks.  As
we mentioned in Section~\ref{s:def-0-eff},
 normal spheres contradict $0$-efficiency of the triangulation of~$X$.
%And because $X$ is not $S^3$,
%almost normal spheres  cannot occur either
% by \cite[Prop.~5.12]{Jaco:0-efficient-triangulations-of-3-manifolds-2003}.
%\jirka{\confirm{Almost normal were handled earlier, so I deleted them from
%this lemma}}
%\eric{yes, good}

Projective planes are excluded because  $X$ is irreducible.
As for disks, since $X$ has incompressible boundary, they also
 have trivial boundary.   But this contradicts the fact that each summand
has essential boundary.
%Disks are, by 0-efficiency, vertex linking disks. This means they survive the
%normal  addition which contradict the fact that $P$ has essential boundary.
\end{proof}

This excludes Case~\ref{case:chi>0}, and we proceed
with Case~\ref{case:a}.

\begin{lemma}
\label{lemNoCompressibleAnnulus}
No summand is a compressible annulus.
\end{lemma}

\begin{proof}
Because $X$ has incompressible boundary, a compressible annulus has trivial boundary by Proposition \ref{propSurfaceFacts}.   This would contradict the fact that every summand of $\partial P$ is essential and tight by Lemma~\ref{lemCurveSummandsLeastComplexity}.
\end{proof}

The next lemma supplies a bound for Case~\ref{case:b}, although it does not
need the assumptions of Case~\ref{case:b} in full strength.

\begin{lemma}
\label{lemMeetAlpha}
$\sum k_i < C_0(t)|\partial P|$, where the sum is restricted to those $F_i$
%that are boundary parallel annuli
for which $\partial F_i \cap \alpha\ne\emptyset$.
%\jirka{Added that these $F_i$ are boundary parallel annuli. Is that OK?
%I think at least we should say that we're talking about the $\chi=0$ case,
%but everything but boundary parallel annuli seems to be handled elsewhere.}
%\martin{I think that the proof works for arbitrary $F_i$ intersectiong $\alpha$
%and actually, I do not see where else is considered the case, say, of an essential
%annulus (M. band) intersecting $\alpha$. Thus I would suggest to return to the
%original version. However, the final choice should be prehaps on Eric.}
%\eric{Yes, Martin is right we need to include essential annuli here.   Later we rule out only those essential annuli that are disjoint from $\alpha$ (and are hence ``parallel'' to $\alpha$).    This lemma can be applied to anything that meets $\alpha$, regardless of Euler characteristic.  However we only need to apply to  Euler characteristic 0 summands that meet $\alpha$. }
%\martin{Removed `boundary parallel' and changed the comment above the lemma.
%Hoping that it is OK now?}
\end{lemma}

\begin{proof}
Addition of (almost) $M$-normal surfaces implies $M$-normal addition of their
boundary curves which  is additive with respect to intersections with the tight
essential fence  $\alpha$ by Proposition \ref{propMSummands}.   Thus the sum of
the coefficients $k_i$ of fundamentals $F_i$ meeting $\alpha$ is bounded by
the total number of intersection with $\alpha$, which in turn is bounded by
Lemma~\ref{lemPSnugTight}.
\end{proof}

Case~\ref{case:c} can be excluded:

\begin{lemma}
\label{lemNoBPA}
No $F_i$ is a normal boundary parallel annulus disjoint from~$\alpha$.
\end{lemma}

\begin{proof}
If not, then we can write $P = B + P'$, an (almost) $M$-normal sum, where
$B=F_i$ is a boundary parallel annulus.   All summands of $P$ have tight
boundary by Lemma~\ref{lemCurveSummandsLeastComplexity},
%\martin{This claim implicitly uses that if a curve is tight then each
%component of it is tight. (I think that this fact is implicitly used at more
%places, but now I cannot recall them.) It follows from
%Lemma~\ref{lemCurveSummandsLeastComplexity}. However, I wonder whether it would
%be useful to let this be a numbered corollary of
%Lemma~\ref{lemCurveSummandsLeastComplexity} that we would refer to.}
%\martin{Added by Lemma~\ref{lemCurveSummandsLeastComplexity}}
so $\partial B = 2b$ a pair  of normally parallel tight essential
curves.   By the construction of $\Gamma$ (Proposition~\ref{propB}),  $b$ is
normally isotopic
%\martin{Isotopic changed into normally isotopic here since I think that we need
%normal isotopy here. Also doing appropriate change in statement of
%Proposition~\ref{propB}.}
to either a component of $\alpha$ or an element of $\Gamma$, and
therefore to a fence.  By  Proposition~\ref{propMSummands}, each point of
$\partial P' \cap b$ has  the same normal sign.   But since $\partial B = 2b$,
all intersections in $\partial  B \cap \partial P'$ have the same normal sign.
This contradicts Lemma~\ref{lemPlusToMinus}.
\end{proof}

Next, we want a bound for Case~\ref{iDisjointAnnuli}.

\begin{lemma}
\label{lemDisjointFromAlpha}
$k_i < C_1(t)|\partial P|$ for each $F_i$ that is a normal  essential annulus or M\"{o}bius band disjoint from~$\alpha$, with a suitable $C_1(t)$.
\end{lemma}
%\martin{After my edits in the proof, the bound should perhaps be $k_i <
%1-2Ct\chi(P)$ ---adjusting the proof if $k_i$ is not even.}

%\eric{ please note that the statement of this lemma has changed rather dramatically.   It used to say $\sum k_i < -2Ct\chi(P)$.   We do not get any such bound and that statement was due to some stupid "editing" on my part.  }

%\martin{I could not verify the proof of this Lemma. I did not try hard, because
%you claim that it still needs an edit. The most confusing part seemed to me
%bounding $k'$ by $|F_i' \cap \alpha|$. If $F_i'$ come from fundamentals $F_i$
%disjoint from $\alpha$ would not be $|F_i' \cap \alpha|$ just $0$? Removing 2
%from the formula seems quite unimportant.}

%\eric{yes, you are right.  that was a mistaken edit that needed to be undone}

\begin{proof}
If $F_i$ is a M\"{o}bius band, let $F_i' := 2 F_i$ and $k_i' := \lfloor k_i/2
\rfloor$, and otherwise,
let $F_i':=F_i, k_i' := k_i$.
%\martin{$k'_i = k_i/2$ replaced with $k_i' = \lfloor k_i/2
%\rfloor$ to be sure ot get an integer.}
Then we write $P=P' + k_i' F_i'$, where $F_i'$ is
an essential annulus.   We wish to show that $k_i' \leq | \partial P \cap
\alpha| - 1$; then the result follows from Lemma~\ref{lemPSnugTight}.

%\martin{At four places below I have replaced $F_i$ with $F'_i$ hoping that it
%is correct.}
So we proceed by contradiction, assuming  $k_i' \geq |\partial P  \cap
\alpha|$.   Let $f_1$ and $f_2$ be the components of $\partial F'_i$.

Proposition~\ref{propEssentialAnnuli} guarantees that $f_1$ and $f_2$ are each
normally parallel to a component of the fence $\alpha$, and thus, by
Proposition~\ref{propMSummands}, $\partial P'$ meets each component $f_j$
in points with the same normal sign.

Since intersection arcs join intersection points of opposite
sign (Lemma~\ref{lemPlusToMinus}), each arc component of $P' \cap F'_i$ meets
both boundary components of $F'_i$ and is thus a spanning arc of $F'_i$.
There
are $n=\frac 1  2  |P' \cap \alpha|$ such spanning arcs of intersection, and
$\partial P'$ meets, say,  $f_1$ in $n$ positive intersections and $f_2$ in $n$
negative intersections.  From the view of boundary curves $\partial P =
\partial P' + k_i' f_1 + k_i' f_2$, where adding copies of $f_1$ and $f_2$ is a
fractional Dehn twist (with fraction $\frac {k_i'}{ n} $) in each of those
curves.

We have assumed that $k_i' \geq |\partial P \cap \alpha| = |\partial P' \cap \alpha| =2n$,  so the fraction is greater than 1.
Then $\partial P' + (k_i'-n) f_1 + (k_i'-n) f_2$ is homeomorphic to
$\partial P$.  Moreover, the homeomorphism can be extended over the annulus
$F_i'$ to a homeomorphism of $X$ that is a Dehn twist in $F_i'$.
 But this homeomorphism takes $P$ to a surface with shorter length and the same number of intersection with $\alpha$, contradicting our choice in
Assumption~\ref{aMinimize}. (This is another place where we may change
the embedding of~$X$.)
\end{proof}

Finally, it is straightforward to bound those summands in Case~\ref{case:nla}.

\begin{lemma}\label{l:negchi} We have
$\sum k_i \leq -\chi(P) < |\partial P|$, where the sum is restricted to those $F_i$
with $\chi(F_i)<0$.
\end{lemma}

\begin{proof}
We have observed that all summands of $P$ have $\chi \leq 0$.
Those with $\chi=0$ do not contribute to $\chi(P)$, and
so $\chi(P) = \sum k_i \chi(F_i)$ for the summands with $\chi \leq -1$.
It follows that $\sum k_i \leq -\chi(P)$ for these summands.
\end{proof}

%We now apply compute the   {\it average length estimate} as  introduced by Jaco and Rubinstein \cite{Jaco:Decision-problems-in-the-space-of-Dehn-fillings-2003,Jaco:Finding-planar-surfaces-in-knot--and-link-manifolds-2009}.

We are ready to bound the average length of a component
of~$\bd P$.

\begin{lemma}
We have
\[
\ell(\partial P) \leq L(t)|\partial P|,
\] where $L$ is a computable function of the number of tetrahedra~$t$.
\end{lemma}

\begin{proof}
Recall that we wrote $\partial P$ as a sum of boundaries of (almost) normal
$M$-fundamentals for the marked triangulation $(\T,M)$,
 where $M = (\alpha \cup \Gamma) \cap \T^1$.
 Proposition~\ref{propFundamentals} bounds both the weight of any fundamental
 solution and the total number of fundamental solutions by computable functions
 of  $t$ and the number of marked points $m=|M|=\ell(\alpha)+\ell(\Gamma)$.   These lengths are bounded by computable functions of $t$ by Propositions~\ref{propEssentialAnnuli} and~\ref{propB},
respectively.   Thus the weight of any $M$-fundamental solution and the total number of $M$-fundamental solutions are bounded by computable functions of $t$
only.

As in the proof outline in Section~\ref{s:outline},
let $\lmax := \max\{\ell(\partial F_i)\}$,
the maximum taken over all normal or almost normal $M$-fundamental
surfaces $F_i$ (in the marked triangulation of~$X$).
Because the length of a surface's boundary is at most its weight, $\lmax$
is bounded by a computable function of $t$.

Because length is additive, we have
\[
\ell(\partial P) = \sum k_i \ell(\partial F_i),
\]  where the sum is restricted to surfaces $F_i$ with non-empty boundary.

If $F_i$ is one of the four types of fundamentals that contribute to $\ell(\partial P)$, then $k_i \leq C_2(t)|\partial P|$,  where $C_2(t) = \max(C_0(t),C_1(t))$, by Lemmas~\ref{lemMeetAlpha},
\ref{lemDisjointFromAlpha}, and~\ref{l:negchi}
(sometimes the bound is much better).

%Indeed: For almost normal summands $\sum k_i \leq 1$.  For the summands $F_i$ that meet $\alpha$,
%$\sum k_i < C(t)(-\chi(P))$ (Lemma~\ref{lemMeetAlpha}).
%For the summands $F_i$ disjoint from $\alpha$ we have $k_i < C(t)(-\chi(P))$
%(Lemma~\ref{lemDisjointFromAlpha}).   And, for the $F_i$ with $\chi<0$ we have $\sum k_i \leq -C t \chi(P)$.

Since the number of distinct fundamentals is bounded by a computable function of $t$, call it $C_3(t)$, we have $\sum k_i \leq C_3(t)C_2(t)|\partial P|$ over all summands that contribute to~$\partial P$.
The total length is then bounded by $\lmax\cdot C_3(t)C_2(t)|\partial P|$
as the lemma claims.
%
% $$\ell(\partial P) \leq \sum k_i \bar \ell  \leq -2A(t)Ct\chi(P) \bar\ell  \leq -\chi(P)L(t) $$
%where $L(t) = -2A(t)Ct\bar\ell$  is a computable function of $t$.
\end{proof}

Theorem \ref{t:short-meridian} now follows.  Because $P$ is meridional or
almost meridional, at least $|\partial P| -1$ of its boundary components are meridians (note that $|\partial P| > 1$ by Lemma~\ref{lemNoPositive}).
Hence the average
length of a meridian is at most
%$$%\textit{average meridian length} \leq
%\frac {\ell(\partial P)}{b-1} \leq
%\frac{-\chi(P)L(t)}{b-1} = \frac{(b-2)L(t)}{b-1} \leq L(t).$$

$$\frac {\ell(\partial P)}{|\partial P|-1} \leq
\frac{L(t)|\partial P|}{|\partial P|-1} = \frac{|\partial P|}{|\partial P|-1}L(t) \leq 2L(t).$$

Unlike the children of Lake Wobegon, some meridian must be at most average, and
hence its length is bounded by $2L(t)$, a computable function of $t$.
This completes the proof of Theorem~\ref{t:short-meridian}.

%%%%%%%%%%%%%%%%%%%%%%%%%%%%%%%%%%%%%%%%%%%%%%%%%%%%%

\section{Embedding 3-dimensional complexes}
\label{s:emb33}

In this section we prove Corollary~\ref{c:emb33}: we provide an
algorithm for \EMBED33. It uses the algorithm for \EMBED23,
as well as an $S^3$ recognition algorithm and an algorithm for \EMBED22.

Let $K$ be a $3$-complex for which we want to test embeddability in
$\R^3$. We assume, w.l.o.g., that $K$ is connected.
The idea is to replace every $3$-simplex of $K$ by
a suitable $2$-dimensional
structure so that an embedding of this $2$-structure ensures
the embeddability of the $3$-simplex.

We call a vertex $v$ of $K$ a \emph{cut vertex} if
removing $v$ from $K$ disconnects $K$. We let $K' := (\sd K)^{(2)}$ to be the
$2$-skeleton of the barycentric subdivision of $K$ (see the paragraph below the
description of the algorithm). We will show that if $K$ is
connected and without cut vertices, then $K$ embeds in $\R^3$ if and only
if $K'$ does. And we will also show that the assumption that $K$
does not contain cut vertices is achievable.

%\martin{Barycentric subdivisions should be at least briefly described. In intro
%or here?}

\heading{Description of the algorithm (assuming $K$ connected).}
\begin{enumerate}
  \item If $K$ is homeomorphic to $S^3$ (which can be tested,
as in the algorithm for \EMBED23), return FALSE.
  \item\label{st:links} If there is a vertex whose link\footnote{We recall
that the \emph{link} of a vertex $v$ in a simplicial complex
$K$ consists of all simplices $\sigma$ of $K$ that do not contain $v$
and such that $\sigma$ together with $v$ forms a simplex of~$K$.}
is not embeddable in $S^2$, return FALSE.
(The embeddability in $S^2$ can be tested using \cite{grossRosen}
and $S^2$ recognition, for example.) %\eric{is this step in the thickening step for the 2-complex?}
%\uli{Formally, there is the difference that the thickening algorithm tests planarity of links in a $2$-complex, i.e., graphs, whereas this step tests planarity or embeddability in $S^2$ of a $2$-dimensional complex.}
%  \item Test whether $K$ is connected. If $K$ is not connected run the
%    algorithm for each component separately and return TRUE if and only if each
%    component embeds in $\R^3$. Otherwise, if $K$ is connected,
%    continue with the algorithm.
  \item\label{st:cutv} If $K$ contains a cut vertex
    $v$, consider two connected induced subcomplexes $K_1$ and $K_2$ of $K$
    such that $K_1 \cup K_2 = K$ and $K_1 \cap K_2 = \{v\}$, $K_1, K_2 \neq  K$. (Note that such $K_1,K_2$ exist:
    after removing $v$ from $K$ we can possibly obtain
    more than two components, but we can merge them into two groups.)
     Run the algorithm for $K_1$ and $K_2$ separately and return
    TRUE if and only if both $K_1$ and $K_2$ embed in $\R^3$.
%\martin{TODO: Explain induced subcomplex. Somewhere in preliminaries?}
  \item \label{st:emb23}
Run the algorithm for
\EMBED23 with $K':=(\sd K)^{(2)}$ and return its answer.
\end{enumerate}
%\martin{New: Wrote a paragraph about geometric realizations and barycentric
%subdivisions. I hope that it is on an appropriate level of detail. Not sure
%whether this is the correct place for this paragraph.}
%\eriC{certainly ok with me}

\heading{Geometric realizations and the barycentric subdivision.}
In this section we need to carefully distinguish a simplicial complex $K$ and
its \emph{geometric realization} $|K|$. (In this
section we use $|\cdot|$ solely
for geometric realizations, although earlier
it meant the number of connected components.)
Given a complex $K$, we denote its
barycentric subdivision by $\sd K$. See the next picture
for an example of barycentric subdivision, and e.g.
\cite{Mat-top} or almost any textbook on algebraic topology
for a detailed treatment of this notion.
\immfigw{sd}{3in}
Given a subcomplex (or a face) $L$ of $K$ we also denote $\sd L$ the
barycentric subdivision of $L$ regarded as a subcomplex of $\sd K$.
The geometric realizations of $K$ and $\sd K$ can be canonically chosen so that
$|K| = |\sd K|$ (and $|L| = |\sd L|$ for every subcomplex); we assume this
canonical choice.

\heading{Correctness of the algorithm.} Now we argue that the algorithm is
correct modulo two lemmas proved below. In first step we exclude the
case $K = S^3$ and thus, we can freely use that PL embeddability of $K$ in
$\R^3$ is equivalent to PL embeddability of $K$ in $S^3$.

Further, if $K$ PL embeds in $\R^3$ then the links of vertices PL embed in $S^2$,\footnote{Intersecting the
PL embedding of $K$ with a sufficiently small $2$-sphere around the image of a vertex $v$, we get an embedding of the
link of $v$ in~$K$.}  so the answer in Step~\ref{st:links} is correct,
and further we may assume that all the links embed in $S^2$.

The next lemma shows correctness of Step~\ref{st:cutv}.
\begin{lemma}
\label{l:cutvertex}
  Let $K$ be a connected simplicial complex such that link of each vertex
  embeds in $S^2$, and let $K_1$ and $K_2$ be two connected induced
  subcomplexes as in Step~\ref{st:cutv}.
  Then $K$ PL embeds in $S^3$ if and only if $K_1$ and $K_2$ PL embed
  in~$S^3$.
\end{lemma}

Finally, the correctness of Step~\ref{st:emb23} relies on the next lemma.
\begin{lemma}
\label{l:K'embed}
  Let $K$ be a connected simplicial complex without cut vertices,
and let $K' = (\sd K)^{(2)}$.
% be the $2$-skeleton of the barycentric subdivision of $K$.
  Then the following conditions are equivalent.

  \begin{enumerate}[\rm (i)]
    \item
      $K$ PL embeds in $S^3$;
    \item
      $K$ topologically embeds in $S^3$; and
    \item
      $K'$ topologically embeds in $S^3$.
\end{enumerate}
\end{lemma}

\heading{Proofs of the lemmas.}
To finish the proof of correctness of the algorithm, it is sufficient to prove
Lemmas~\ref{l:cutvertex}~and~\ref{l:K'embed}. The proofs rely on the
the $3$-dimensional PL Schoenflies theorem  (see, e.g., \cite[Theorem XIV.1]{Bing:The-geometric-topology-of-3-manifolds-1983}):

\begin{theorem}[PL Schoenflies theorem for $\R^3$]
\label{t:pl_schoenflies}
Let $f\colon S^2 \to \R^3$ be a PL embedding.\footnote{Formally, the standard PL model of $S^d$ is the boundary $\partial \Delta^{d+1}$,
and $f$ is a PL map of the complex $\partial \Delta^3$ in $\R^3$.} Then there is a PL homeomorphism $h\colon \R^3\to \R^3$ such that
$h\circ i$ is a standard embedding of $S^2$ in $\R^3$ (as the boundary of a $3$-simplex $\Delta^3$). Moreover, $h$ can be chosen to be the identity outside any
given open set $U$ that contains the bounded component of $\R^3\setminus f(S^2)$. In particular, the bounded component of $\R^3\setminus f(S^2)$ is a PL ball.
 \end{theorem}
This easily implies the following version for $S^3$ (which is also standard but
we did not find a reference exactly in this setting):
\begin{corollary}[PL Schoenflies for $S^3$]
\label{cor:S3-Schoenflies}
If $f\colon S^2\to S^3$ is a PL embedding, then there is a PL homeomorphism $g\colon S^3\rightarrow S^3$ such that $g\circ f$ is the standard inclusion of $S^2$ as the boundary of a hemisphere.
In particular, the closures of both components of $S^3\setminus f(S^2)$ are PL balls with boundary $f(S^2)$.
\end{corollary}
\begin{proof}[Proof of Corollary~\ref{cor:S3-Schoenflies}]
Choose a sufficiently fine PL triangulation of $S^3$ such that $f(S^2)$ avoids one of the closed $d$-simplices $\sigma$ of $S^3$.
By Newman's theorem~\cite[Cor.~3.13]{Rourke:Introduction-to-piecewise-linear-topology-1972}, the closure of the complement of a PL $3$-ball in $S^3$ is a PL $3$-ball, i.e., PL homeomorphic to a $3$-simplex; in particular, $\overline{S^3\setminus \sigma}$ is PL homeomorphic to a $3$-simplex $\Delta^3_1$.

Fix such such a PL homeomorphism $j\colon \overline{S^3\setminus \sigma}\cong \Delta^3_1\subseteq \R^3$. Then $j\circ f$ is a PL embedding of $S^2$ in $\R^3$. Thus, by Theorem~\ref{t:pl_schoenflies}, there is a PL homeomorphism $h$ of $\R^3$ such that $h\circ j\circ f$ is a standard embedding of $S^2$ as the boundary of some $3$-simplex $\Delta_2^3$. Moreover, the PL homeomorphism $h\circ j$ witnesses that the closure of the component of $S^3\setminus f(S^2)$ avoiding $\sigma$ is a PL ball $B^3\subset S^3$ with boundary $f(S^2)$. Furthermore, there is a PL homeomorphism $k$ from $\Delta^3_2$ to the closed lower hemisphere $H^3_- \subset S^3$ (e.g., with $S^3$ triangulated as the octahedral $3$-sphere).
By \cite[Cor.~3.15]{Rourke:Introduction-to-piecewise-linear-topology-1972}, the PL homeomorphism $k\circ h\circ j\colon B^3 \cong H_-^3$ can be extended to a PL homeomorphism $g\colon S^3\to S^3$, which has the desired property.
\end{proof}

\begin{proof}[Proof of Lemma~\ref{l:cutvertex}]
If $K$ PL embeds in $S^3$, then both $K_1$ and $K_2$ PL embed in $S^3$
since they are subcomplexes of $K$. In sequel we assume that $K_1$ and $K_2$ PL
embed in $S^3$ and we want to prove that $K$ PL embeds in $S^3$.

The idea is very simple, we just want to transform an embedding of $K_1$ and
$K_2$ so that the common vertex $v$ protrude on the boundary of each and thus
they can be joined together; see Figure~\ref{f:K1K2} in one dimension less. It
remains to show that such a transformation can be found.

\begin{figure}[t]
\begin{center}
  \includegraphics{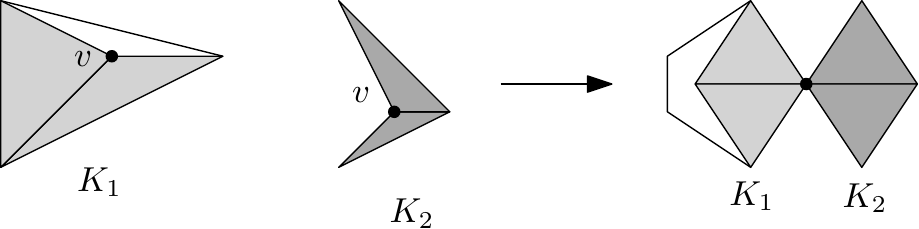}
\end{center}
\caption{Transforming $K_1$ and $K_2$.}
\label{f:K1K2}
\end{figure}

From the assumptions on links of vertices and from $K_1 \neq K$ we deduce
that the link of $v$ in $K_1$ is planar (and thus
different from $S^2$). Indeed, if it were homeomorphic to $S^2$, then the link
of $v$ in $K$ would not embed in $S^2$ since the link of $v$ in $K_2$ must be
nonempty ($K_2 \neq \{v\}$ since $K_1 \neq K$). Similarly we can deduce that
the link of $v$ in $K_2$ is planar.

Let $f_1\colon |K_1| \to S^3$ be a PL embedding. By the previous observation we
deduce that $f_1(v)$ is on the boundary of $f_1(|K_1|)$. Therefore, there is a
geometric simplex $\sigma$ in a small neighborhood of $f_1(v)$ such that
$\sigma \cap f_1(|K_1|) = \{v\}$. Consequently, by the PL Schoenflies theorem
(Corollary~\ref{cor:S3-Schoenflies}), there is a PL automorphism $\psi$ of $S^3$ mapping
$\partial \sigma$ to $S^2 \subset S^3$. In addition, we can assume that it maps
the interior of $\sigma$ to the upper hemisphere of $S^3$ and $f_1(v)$ to a
pre-chosen point $x$ on $S^2$. Altogether, $g_1 := \psi \circ f_1$
is a PL embedding mapping $|K_1|$ to the lower hemisphere of $S^3$ such
that $g_1(|K_1|) \cap S^2 = \{x\}$.

Similarly, we can find a PL map $g_2\colon |K_2| \to S^3$ such that $|K_2|$ is
mapped to the upper hemisphere of $S^3$ and $g_2(|K_2|) \cap S^2 = \{x\}$.
Finally, we can construct the desired PL embedding $g$ of $|K|$ by setting
$g(y) := g_1(y)$ if $y \in |K_1|$ and $g(y) := g_2(y)$ if $y \in |K_2|$.
\end{proof}

%\martin{Originally I wanted to draw a nice picture for $\psi$ (in one dimension
%  less). I did not manage
%to do so quickly with moderate effort. Would be such a picture helpful? Can
%spend more effort on it.}

%\martin{Do we use geometric simplicial complexes? In the text below I assume
%so.}

\begin{proof}[Proof of Lemma~\ref{l:K'embed}]
Clearly (i)$\Rightarrow$(ii), and (ii)$\Rightarrow$(iii)
 since  $K'$ is a subcomplex of a subdivision of $K$. It remains to show
 (iii)$\Rightarrow$(i).

  Since $K'$ is $2$-dimensional and since topological and PL embeddability coincide for
  embedding $2$-complexes in $S^3$, there is an PL embedding $f'\colon |K'|
  \to S^3$. Let $f_0$ be the restriction of $f'$ to $|K^{(2)}|$
(which is a subspace of $|K'|$).
%  Given a simplex $\sigma \in K$ by $\sd \sigma$ we mean the
%  barycentric subdivision of $\sigma$ as a subcomplex of $K'$. By the
%  properties of the barycentric subdivision we have $\sigma = |\sd \sigma|$
%  (as spaces). In particular, if $K^{(2)}$ is the $2$-skeleton of $K$, then
%  $|K^{(2)}| \subseteq |K'|$, and therefore there is a $PL$ embedding
%  $f_0\colon |K^{(2)}| \to S^3$ given by $f_0(x) := f'(x)$.
We want to extend $f_0$ to a PL embedding $f\colon |K|\to S^3$.

  We will describe how to extend $f_0$ to each tetrahedron independently, and
  then we argue that these extensions can be done simultaneously,
which yields the desired $f$. The argument is illustrated
in Figure~\ref{f:extend2to3}.

\begin{figure}
\begin{center}
\includegraphics{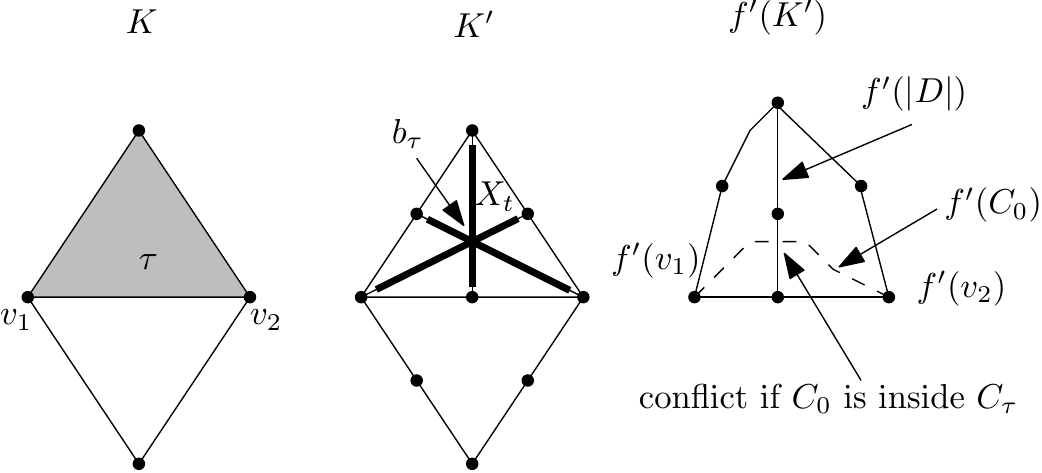}
\caption{$K$, $K'$ and $f(K')$ schematically drawn in the plane.}
\label{f:extend2to3}
\end{center}
\end{figure}

 Let $\tau$ be a tetrahedron of $K$. By the PL Schoenflies theorem
  (Corollary~\ref{cor:S3-Schoenflies}),
  $f_0(\partial \tau)$ splits the sphere $S^3$ in two open components whose
  closures are PL homeomorphic to $B^3$. Let $b_\tau$ be the barycentre of $\tau$ and let
  $C_\tau$ be the component that contains $f'(b_\tau)$. We will argue that
  $f_0(|K^{(2)}|) \cap C_\tau = \emptyset$.

  Recall  that $f_0(\partial \tau)$  is disjoint from $C_\tau$.
  For contradiction, let us assume that there is a component $C_0$ of
  $|K^{(2)}| \setminus \partial \tau$ such that $f_0(C_0) \cap C_\tau \neq
  \emptyset$, that is $f_0(C_0) \subseteq C_\tau$.
  Let $X_\tau := |(\sd \tau)^{(2)}| \setminus |(\sd \partial
  \tau)|$.
% be the set of those points of $|(\sd \tau)^{(2)}|$ which do not
%  belong to $|K^{(2)}|$.
  We have $f'(X_\tau)\subseteq C_\tau$ since
  $X_\tau$ is connected.
  Let $V_0$ be the set of
   those vertices $v$ of $\tau$ which belong to the closure of
   $C_0$. We have $|V_0| > 0$ since $K$ is connected. Furthermore, $|V_0|
   > 1$ since $K$ does not contain cut vertices. Therefore $|V_0| \geq 2$ and
   there is a path $P$ inside $|K^{(2)}|$ starting in a vertex $v_1$ of $\tau$,
   ending in a vertex $v_2$ of $\tau$, and
 with interior points in $C_0$.

Let $D$ be
   the subcomplex of $(\sd \tau)^{(2)}$, homeomorphic to a disk,
consisting of the simplices
   in the plane of symmetry of $v_1$ and $v_2$ (considering
   $\tau$ as a regular simplex); it passes  through the other two vertices of
   $\tau$ and the midpoint of $v_1v_2$. By a double application of the PL
   Schoenflies theorem and using that the interior of $|D|$ belongs to
   $X_\tau$ (which maps to $C_\tau$ under $f'$), we have that $f'(|D|)$ splits the closure of $C_\tau$ into two parts, both PL
   homeomorphic to a ball. Since $f_0(v_1)$ and $f_0(v_2)$ are in different
   parts, $f'(|D|)\cap f_0(C_0)\ne\emptyset$, which is impossible
since $f'$
%   that is $f'(|D|)$ and $f'(C_0)$ intersect considering $C_0$ as a subset of
%   $K'$. This contradicts that $f'$
is an embedding.

We have deduced that $f_0(|K^{(2)}|) \cap C_\tau =
\emptyset$ for each tetrahedron $\tau$. By the
PL Schoenflies theorem $f_0$ can be extended to a PL embedding
of $K^{(2)} \cup \{\tau\}$ so that the interior of $\tau$ is mapped to $C_\tau$.
Moreover, if we consider two distinct tetrahedra $\tau_1, \tau_2 \in K$, then
$C_{\tau_1} \cap C_{\tau_2} = \emptyset$. Indeed, if $C_{\tau_1} \cap
C_{\tau_2} \neq \emptyset$, then $C_{\tau_1} \subseteq C_{\tau_2}$ because
$\partial C_{\tau_2} = f_0(\partial \tau_2)$ is disjoint from $C_{\tau_1}$.
Similarly we deduce $C_{\tau_2} \subseteq C_{\tau_1}$ implying $\partial
C_{\tau_1} = \partial C_{\tau_2}$ contradicting the assumption
 that $\tau_1$ and $\tau_2$ are
distinct simplices.

Altogether, we can extend $f_0$ to every tetrahedron of $K$ independently,
obtaining the required PL embedding of $K$.
\end{proof}

\fi

\section*{Acknowledgment}
We thank David Bachman, Ryan Derby-Talbot, and William Jaco for very helpful
conversations. We further thank Tony Huynh for kind answers to our questions,
Helena Nyklov\'{a} for help with many pictures, and
Isaac Mabillard for remarks on a preliminary write-up of the algorithm for
\EMBED33.

\bibliographystyle{alpha}
\bibliography{Postnikov,3Manifolds}
\end{document}